\documentclass[twoside,11pt]{article}
\usepackage{amssymb, amsmath, amsthm, graphicx, tikz, tikz-cd, bbm, appendix, cancel}
\usepackage{wrapfig}
\usepackage[inline]{enumitem}
\usepackage{subfiles}
\usepackage{blindtext}
\usepackage{adjustbox}
\usepackage[colorlinks,bookmarks=true]{hyperref}
\usepackage[capitalize,nameinlink,noabbrev]{cleveref}
\crefformat{footnote}{#2\footnotemark[#1]#3}
\usepackage{manfnt}
\usepackage{mymacros} 
\newtheorem{conject}[thm]{Conjecture}
 
\usepackage[normalem]{ulem}
\newcommand{\fun}[2][2]{[#2 \to #1]}
\newcommand{\Ss}{\mathcal{S}}

\newcommand{\intrr}{\operatorname{int}}

\newcommand{\link}{\operatorname{link}}
\newcommand{\codim}{\operatorname{codim}}
\newcommand{\reg}{\operatorname{reg}}
\newcommand{\kk}{\mathbbm{k}}

\renewcommand{\parity}{\textsc{parity}}
\newcommand{\monconj}{\textsc{monconj}}
\newcommand{\cDelta}{\textsc{delta}}
\newcommand{\CLconj}{\textsc{conj}}
\newcommand{\linthr}{\textsc{linthr}}
\newcommand{\linfun}{\textsc{linfun}}
\newcommand{\polythr}{\textsc{polythr}}

\newcommand{\cochain}[1]{\mathcal{C}^{#1}}

\newcommand{\betti}{\beta}
\newcommand{\lcm}{\operatorname{lcm}}
\newcommand{\es}{\mathbf}
\newcommand{\resf}{\mathfrak} 
\newcommand{\cresf}{\mathsf} 
\newcommand{\clsf}{\mathtt} 
\newcommand{\pf}{\mathsf} 
\newcommand{\hdim}{\operatorname{\dim_{\mathrm{h}}}}
\newcommand{\srdim}{\operatorname{\dim_{\mathrm{SR}}}}
\newcommand{\vcdim}{\operatorname{\dim_{\mathrm{VC}}}}

\newcommand{\mingen}{\operatorname{mingen}}

\newcommand{\xx}{\mathbf{x}}
\newcommand{\splx}{\triangle}
\newcommand{\Sbpx}[1]{\mathcal{S}_{\clsf{#1}}}
\newcommand{\cani}[1]{I_{#1}^\star} 
\newcommand{\ind}{\mathbb{I}} 
\newcommand{\ex}{\operatorname{ex}}
\newcommand{\vcr}{\operatorname{rad}_{\mathrm{VC}}}
\newcommand{\ballres}{\cresf{BALL}}
\newcommand{\flagres}{\cresf{FLAG}}
\newcommand{\cocres}{\cresf{COC}}
\newcommand{\rH}{\widetilde{H}}
\newcommand{\emptycell}{\varnothing}
\newcommand{\cartcap}{\boxtimes}
\newcommand{\thr}{\operatorname{thr}}
\newcommand{\pbd}{\overline{\pd}} 
\newcommand{\nei}{\sim}
\newcommand{\Shx}{\mathcal{SH}}
\newcommand{\pclsf}{\mathcal}
\newcommand{\bomega}{\boldsymbol{\omega}}
\newcommand{\bgamma}{\boldsymbol{\gamma}}

\newcommand{\FM}{\mathbf{FM}} 
\newcommand{\cube}{\text{\mancube}}
\newcommand{\Monom}{\mathbf{M}}
\newcommand{\projdim}{\operatorname{projdim}}
\newcommand{\ocell}{\mathring}
\newcommand{\emptyfun}{\dagger}
\newcommand{\filt}{\downharpoonright}

\newcommand{\St}{\operatorname{St}}
\newcommand{\monom}[1]{\mathbf{m}^{#1}}
\newcommand{\makeconvo}[1]{
\vskip 6mm
\begin{center}
{\fontfamily{qag}\fontshape{it}\selectfont #1}
\end{center}
\vskip 6mm}
\makeatletter
\newcommand*\supstar{\@ifnextchar_\supst@r\supst@@}
\newcommand*\supst@@{\sup\nolimits^{\star}}
\newcommand*\supst@r[2]{\mathchoice
    {\mathop{\supst@@}_{#2\phantom{\star}}}%
    {\jj@@n_{#2}}%
    {\jj@@n_{#2}}%
    {\jj@@n_{#2}}%
}
\makeatother
\usepackage[numbers]{natbib}
\usepackage{multirow}
\usepackage{pbox}
\usepackage{hhline}
\usepackage{subcaption}
\usepackage[font=small,labelfont=bf]{caption}
\usepackage[margin=1in]{geometry}
\newcolumntype{L}{>{\centering\arraybackslash}m{2cm}}
\newcolumntype{M}{>{$}l<{$}} 

\usepackage{fancyhdr}
\fancyhfoffset{.05\textwidth}

\title{A Homological Theory of Functions}
\author{Greg Yang\\
 	{\small Harvard University}\\
     \texttt{\small gyang@college.harvard.edu}
}
\makeatletter
\let\theauthor\@author
\let\thetitle\@title
\makeatother

\begin{document}
\pagenumbering{roman}
\maketitle

\begin{abstract}
In computational complexity, a complexity class is given by a set of problems or functions, and a basic challenge is to show separations of complexity classes $\clsf A \not= \clsf B$ especially when $\clsf A$ is known to be a subset of $\clsf B$.
In this paper we introduce a homological theory of functions that can be used to establish complexity separations, while also providing other interesting consequences.
We propose to associate a topological space $\Sbpx{A}$ to each class of functions $\clsf A$, such that,
to separate complexity classes $\clsf A \sbe \clsf B'$, it suffices to observe a change in ``the number of holes'', i.e. homology, in $\Sbpx A$ as a subclass $\clsf B \sbe \clsf B'$ is added to $\clsf A$.
In other words, if the homologies of $\Sbpx A$ and $\Ss_{\clsf A \cup \clsf B}$ are different, then $\clsf A \not = \clsf B'$.
We develop the underlying theory of functions based on combinatorial and homological commutative algebra and Stanley-Reisner theory, and recover Minsky and Papert's result \cite{minsky_perceptrons:_1969} that parity cannot be computed by nonmaximal degree polynomial threshold functions.
In the process, we derive a ``maximal principle'' for polynomial threshold functions that is used to extend this result further to arbitrary symmetric functions.
A surprising coincidence is demonstrated, where the maximal dimension of ``holes'' in $\Sbpx A$ upper bounds the VC dimension of $\clsf A$, with equality for common computational cases such as the class of polynomial threshold functions or the class of linear functionals in $\Fld_2$, or common algebraic cases such as when the Stanley-Reisner ring of $\Sbpx A$ is Cohen-Macaulay.
As another interesting application of our theory, we prove a result that a priori has nothing to do with complexity separation: it characterizes when a vector subspace intersects the positive cone, in terms of homological conditions.
By analogy to Farkas' result doing the same with {\it linear conditions}, we call our theorem the Homological Farkas Lemma.
\end{abstract}

%
%
%
\section{Introduction}

\pagenumbering{arabic}

\subsection{Intuition}
Let $\clsf A \sbe \clsf B'$ be classes of functions.
To show that $\clsf B' \not = \clsf A$, it suffices to find some $\clsf B \sbe \clsf B'$ such that
$$\clsf A \cup \clsf B \not = \clsf A.$$
In other words, we want to add something to $\clsf A$ and watch it {\it change}.

\makeconvo{Let's take a step back}

Consider a more general setting, where $A$ and $B$ are ``nice'' subspaces of a larger topological space $C$.
We can produce a certificate of $A \cup B \not = A$ by observing a difference in the number of ``holes'' of $A \cup B$ and $A$.
\cref{fig:union_changes_homology} shows two examples of such certificates.

\begin{figure*}[h]
	\begin{subfigure}[t]{0.5\textwidth}
	\centering
	\includegraphics[height=0.12\textheight]{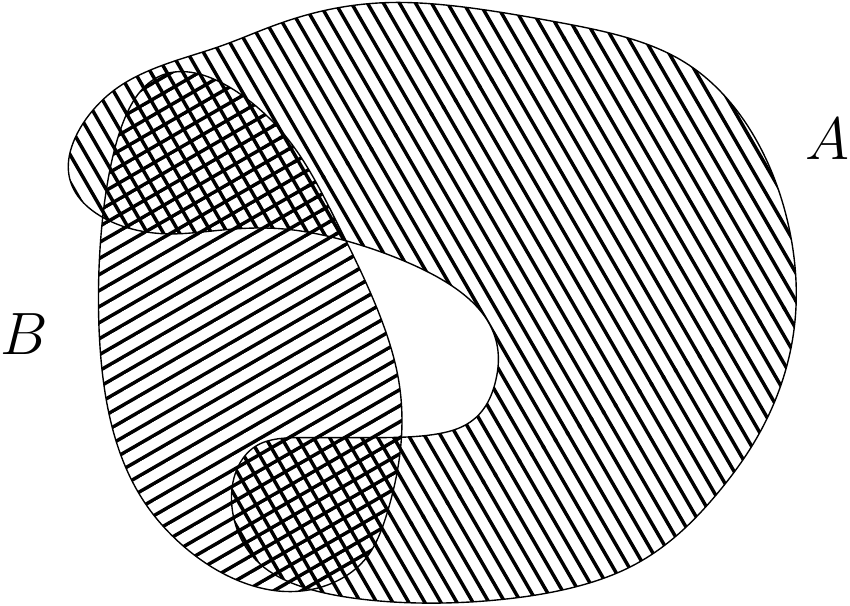}
	\caption{$A$ and $B$ are both contractible (do not have holes), but their union $A \cup B$ has a hole.}
	\label{fig:union_creates_hole}
	\end{subfigure}
	~
	\begin{subfigure}[t]{0.5\textwidth}
	\centering
	\includegraphics[height=0.12\textheight]{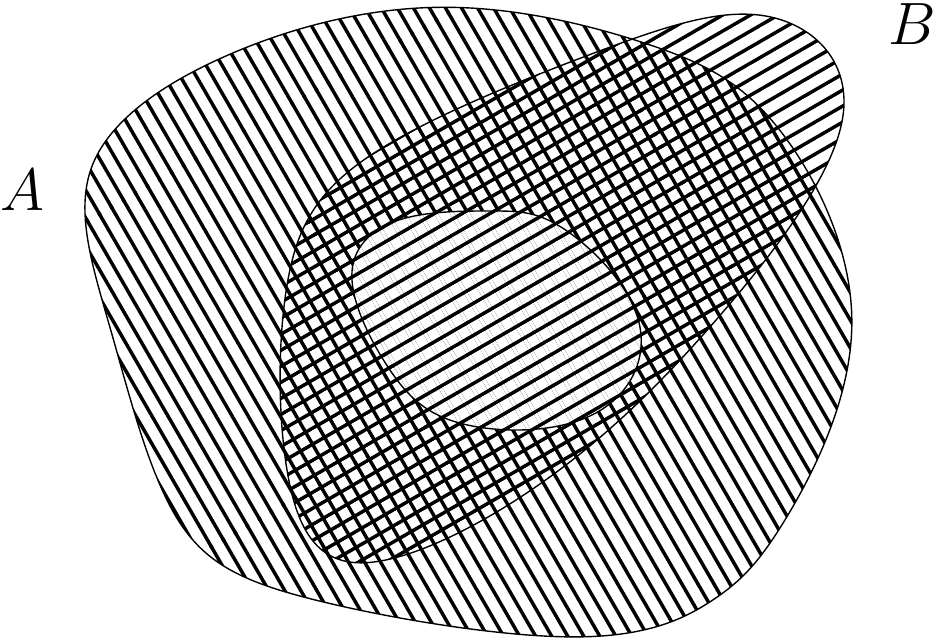}
	\caption{$A$ has a hole in its center, but $B$ covers it, so that $A \cup B$ is now contractible.}
	\label{fig:union_removes_hole}
	\end{subfigure}
\caption{Certifying $A\cup B \not=A$ by noting that the numbers of 1-dimensional holes are different between $A\cup B$ and $A$.}
\label{fig:union_changes_homology}
\end{figure*}

Sometimes, however, there could be no difference between the number of holes in $A \cup B$ and $A$.
For example, if $B$ in \cref{fig:union_creates_hole} is slightly larger, then $A \cup B$ no longer has a hole in the center (see \cref{fig:slice_homology1}).
But if we take a slice of $A \cup B$, we observe a change in the number of connected components (zeroth dimensional holes) from $A$ to $A \cup B$.

\begin{figure}[h]
\centering
\includegraphics[height=0.16\textheight]{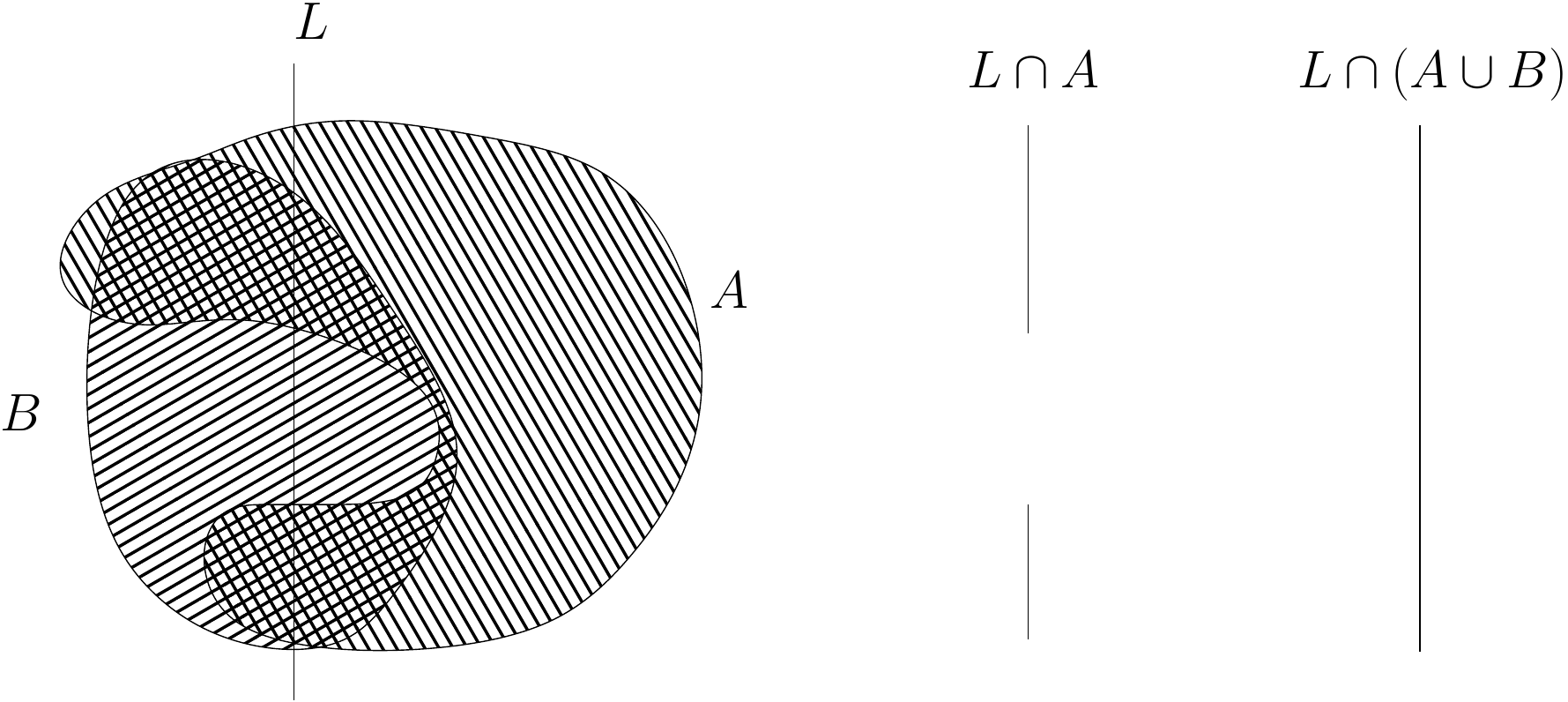}
\caption{$A \cup B$ and $A$ are both contractible, but if we look at a section $L$ of $A \cup B$, we see that $L \cap A$ has 2 connected components, but $L \cap (A \cup B)$ has only 1.}
\label{fig:slice_homology1}
\end{figure}

From this intuition, one might daydream of attacking complexity separation problems this way:
\begin{enumerate}
  \item
	For each class $\clsf A$, associate a unique topological space (specifically, a simplicial complex) $\Sbpx A$.
  \item
	Compute the number of holes in $\Sbpx A$ and $\Sbpx{A \cup B}$ of each dimension, and correspondingly for each section by an affine subspace.
  \item
	Attempt to find a difference between these quantities (a ``homological'' certificate).
\end{enumerate}

\makeconvo{It turns out this daydream is not so dreamy after all!}

This work is devoted to developing such a homological theory of functions for complexity separation, which incidentally turns out to have intricate connection to other areas of computer science and combinatorics.
Our main results can be summarized as follows: 1) Through our homological framework, we recover Marvin Minsky and Seymour Papert's classical result that polynomial threshold functions do not compute parity unless degree is maximal \cite{minsky_perceptrons:_1969}, and in fact we discover multiple proofs, each ``coresponding to a different hole''; the consideration of lower dimension holes yields a {\it maximal principle} for polynomial threshold functions that is used to extend Minsky and Papert's result to arbitrary symmetric functions \cite{aspnes_expressive_1994}. 2) We show that an algebraic/homological quantity arising in our framework, the homological dimension $\hdim \clsf A$ of a class $\clsf A$, upper bounds the VC dimension $\vcdim \clsf A$ of $\clsf A$.
Informally, this translates to the following remarkable statement: ``The highest dimension of any holes in $\Sbpx A$ or its sections upper bounds the number of samples needed to learn an unknown function from $\clsf A$, up to multiplicative constants.''
We furthermore show that equality holds in many common cases in computation (for classes like polynomial thresholds, $\Fld_2$ linear functionals, etc) or in algebra (when the Stanley-Reisner ring of $\Sbpx A$ is Cohen-Macaulay).
3) We formulate the {\it Homological Farkas Lemma}, which characterizes by homological conditions when a linear subspace intersects the interior of the positive cone, 
and obtain a proof for free from our homological theory of functions.

While the innards of our theory relies on homological algebra and algebraic topology, we give an extended introduction in the remainder of this section to the flavor of our ideas in what follows, assuming only comfort with combinatorics, knowledge of basic topology, and a geometric intuition for ``holes.''
A brief note about notation: $[n]$ denotes the set $\{0, \ldots, n-1\}$, and $[n \to m]$ denotes the set of functions from domain $[n]$ to codomain $[m]$.
The notation $\pf f:\sbe A \to B$ specifies a partial function from domain $A$ to codomain $B$.
$\emptyfun$ represents the partial function with empty domain.
%

\subsection{An Embarassingly Simple Example}
\newcommand{\oneff}{\mathbf{1}}
\newcommand{\indone}{\ind_\oneff}
Let $\linfun_d \cong (\Fld_2^d)^*$ be the class of linear functionals of a $d$-dimensional vector space $V$ over $\Fld_2$.
If $d \ge 2$, then $\linfun_d$ does not compute the indicator function $\indone$ of the singleton set $\{\mathbf 1 := 11\cdots 1\}$.
This is obviously true, but let's try to reason via a ``homological way.''
This will provide intuition for the general technique and set the stage for similar analysis in more complex settings.

Let $\pf g: \mathbf 0 \to 0, \mathbf 1 \to 1$.
Observe that for every partial linear functional $\pf h \supset \pf g$ {\it strictly extending} $\pf g$, $\indone$ intersects $\pf h$ nontrivially.
(Because $\indone$ is zero outside of $\pf g$, and every such $\pf h$ must send at least one element to zero outside of $\pf g$).
I claim this completes the proof.

\makeconvo{Why?}

{\it Combinatorially}, this is because if $\indone$ were a linear functional, then for any 2-dimensional subspace $W$ of $V$ containing $\{\mathbf 0, \mathbf 1\}$, the partial function $\pf h:\sbe V \to \Fld_2, \dom \pf h = W$,
$$
	\pf h(u) = \begin{cases}
		\pf g(u) & \text{if $u \in \dom \pf g$}\\
	
	1 - \indone(u) & \text{if $u \in \dom \pf h \setminus \dom \pf g$}
		\end{cases} 
$$
	 is a linear functional, and by construction, does not intersect $\indone$ on $W \setminus \{\mathbf 0, \mathbf 1\}$.

  	{\it Homologically}, we are really showing the following
  	\begin{equation*}
  	\parbox{.8\textwidth}{\it The space associated to $\linfun_d$, in its section by an affine subspace corresponding to $\pf g$, ``has a hole'' that is ``filled up'' when $\indone$ is added to $\linfun_d$.}
  	\end{equation*}
  	
\makeconvo{``Wait, what?
I'm confused.
I don't see anything in the proof resembling a hole?''}

\subsection{The Canonical Suboplex}
OK. No problem. Let's see where the holes come from.

Let's first define the construction of the simplicial complex $\Sbpx C$
associated to any function class $\clsf C$, called the {\bf canonical suboplex}.
In parallel, we give the explicit construction in the case of $\clsf C = \linfun_d' := \linfun_2 \filt \{\mathbf 0 \mapsto 0\}$.
This is the same class as $\linfun_2$, except we delete $\mathbf 0$ from the domain of every function.
It gives rise to essentially the same complex as $\linfun_2$, and we will recover $\Sbpx{\linfun_2}$ explicitly at the end.

Pick a domain, say $[n] = \{0, \ldots, n-1\}$.
Let $\clsf C \sbe [n \to 2]$ be a class of boolean functions on $[n]$.
We construct a simplicial complex $\Sbpx C$ as follows:

\begin{enumerate}
	\item To each $f \in \clsf C$ we associate an $(n-1)$-dimensional simplex $F_f \cong \splx^{n-1}$, which will be a facet of $\Sbpx C$.


	\item Each of the $n$ vertices of $F_f$ is labeled by an input/output pair $i \mapsto f(i)$ for some $i \in [n]$, and each face $G$ of $F_f$ is labeled by a partial function $\pf f \sbe f$, whose graph is specified by the labels of the vertices of $G$.
	See \cref{fig:linfun2_associate_and_label} for the construction in Step 1 and Step 2 for $\linfun_2'$.

\begin{figure*}
\centering
\begin{subfigure}[t]{.6\textwidth}
\centering
\includegraphics[height=.14\textheight]{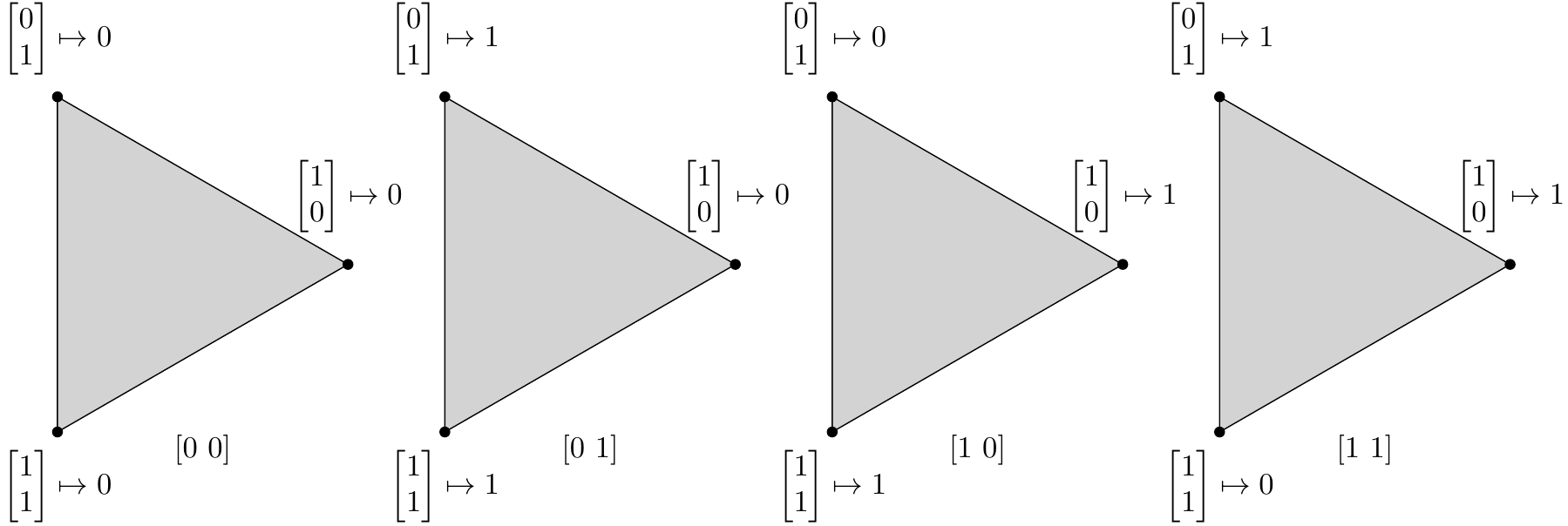}
\caption{Step 1 and Step 2 for $\linfun'_2$.
Step 1: Each simplex is labeled with a function $f \in \linfun_2'$, represented as a row vector.
Step 2: Each vertex of each simplex is labeled by an input/output pair, here presented in the form of a column vector to a scalar.
The collection of input/output pairs in a simplex $F_f$ recovers the graph of $f$.
Each face of $F_f$ has an induced partial function label, given by the collection of input/output pairs on its vertices (not explicitly shown).}
\label{fig:linfun2_associate_and_label}
\end{subfigure}
~
\begin{subfigure}[t]{.3\textwidth}
\centering
\includegraphics[height=.15\textheight]{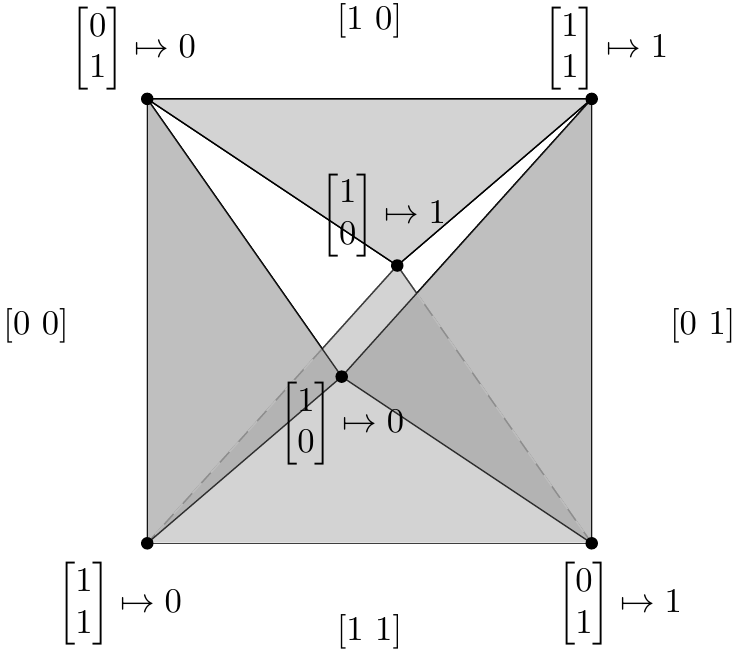}
\caption{Step 3 for $\linfun'_2$.
The simplices $F_f$ are glued together according to their labels.
For example, $F_{[0\ 0]}$ and $F_{[0\ 1]}$ are glued together by their vertices with the common label $[1\ 0]^T \mapsto 0$, and not anywhere else because no other faces share a common label.}
\label{fig:linfun2_glue_together}
\end{subfigure}
\caption{}
\end{figure*}

	\item For each pair $f, g \in \clsf C$, $F_f$ is glued together with $F_g$ along the subsimplex $G$ (in both facets) with partial function label $f \cap g$.
	See \cref{fig:linfun2_glue_together} for the construction for $\linfun_2'$.

\end{enumerate}

This is the simplicial complex associated to the class $\clsf C$, called the {\bf canonical suboplex} $\Sbpx C$ of $\clsf C$.
Notice that in the case of $\linfun_d'$, the structure of ``holes'' is not trivial at all: $\Ss_{\linfun_d'}$ has 3 holes in dimension 1 but no holes in any other dimension. 
An easy way to visualize this it to pick one of the triangular holes;
If you put your hands around the edge, pull the hole wide, and flatten the entire complex onto a flat plane, then you get \cref{fig:linfun2_stretched_apart}.
\begin{figure*}
\centering
\begin{subfigure}[t]{.45\textwidth}
\centering
\includegraphics[height=.15\textheight]{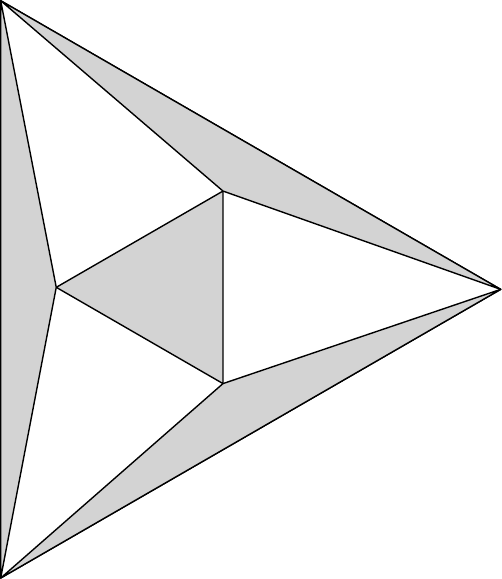}
\caption{The shape obtained by stretching $\Ss_{\linfun_d'}$ along one of its triangular holes and then flatten everything onto a flat plane.
This deformation preserves all homological information, and from this picture we see easily that $\Ss_{\linfun_d'}$ has 3 holes, each of dimension 1.}
\label{fig:linfun2_stretched_apart}
\end{subfigure}
~
\begin{subfigure}[t]{.45\textwidth}
\centering
\includegraphics[height=0.15\textheight]{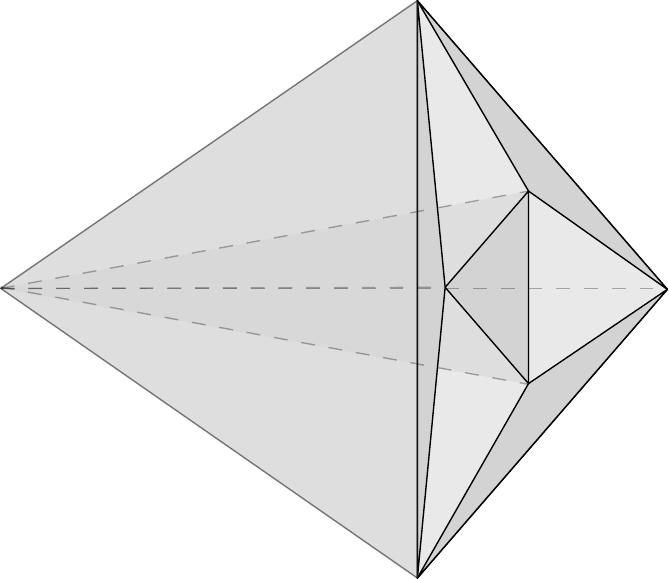}
\caption{The canonical suboplex of $\linfun_d$ is just a cone over that of $\linfun'_d$.
Here we show the case $d = 2$.}
\label{fig:linfun2_cone}
\end{subfigure}
\caption{}
\end{figure*}

It is easy to construct the canonical suboplex of $\linfun_d$ from that of $\linfun'_d$: $\Ss_{\linfun_d}$ is just a cone over $\Ss_{\linfun'_d}$, where the cone vertex has the label $[0\ 0]^T \mapsto 0$ (\cref{fig:linfun2_cone}).
This is because every function in $\linfun_d$ shares this input/output pair.
Note that a cone over any base has no hole in any dimension, because any hole can be contracted to a point in the vertex of the cone.
This is a fact we will use very soon.

Let's give another important example, the class of all functions.
If $\clsf C = [n \to 2]$, then one can see that $\Sbpx C$ is isomorphic to the 1-norm unit sphere (also known as {\it orthoplex}) $S_1^{n-1} := \{\|x\|_1 = 1: x \in \R^n\}$ (\cref{fig:canonical_suboplex_cAll3}).
For general $\clsf C$, $\Sbpx C$ can be realized as a subcomplex of $S_1^{n-1}$.
Indeed, for $\clsf C = \linfun_2' \sbe [3 \to 2]$, it is easily seen that $\Sbpx C$ is a subcomplex of the boundary of an octahedron, which is isomorphic to $S_1^2$.

Let $\clsf C \sbe [n \to 2]$, and let $\pf f:\sbe [n] \to [2]$ be a partial function.
Define the {\bf filtered class} $\clsf C\filt \pf f$ to be
$$\{g \setminus \pf f: g \in \clsf C, g \spe \pf f \} \sbe [[n] \setminus \dom \pf f \to [2]]$$ 

Unwinding the definition: $\clsf C \filt \pf f$ is obtained by taking all functions of $\clsf C$ that extend $\pf f$ and ignoring the inputs falling in the domain of $\pf f$.
 
The canonical suboplex $\Sbpx {\clsf C \filt \pf f}$ can be shown to be isomorphic to an affine section of $\Sbpx C$, when the latter is embedded as part of the $L_1$ unit sphere $S_1^{n-1}$.
\cref{fig:filtered_subcomplex-linear_section} shows an example when $\pf f$ has a singleton domain.
Indeed, recall $\linfun_d'$ is defined as $\linfun_d \filt \{\mathbf 0 \mapsto 0\}$, and we may recover $\Ss_{\linfun_d'}$ as a linear cut through the ``torso'' of $\Ss_{\linfun_d}$ (\cref{fig:linfun2_section}).  
\begin{figure}
\centering
\begin{subfigure}[t]{.3\textwidth}
\centering
\includegraphics[height=0.2\textheight]{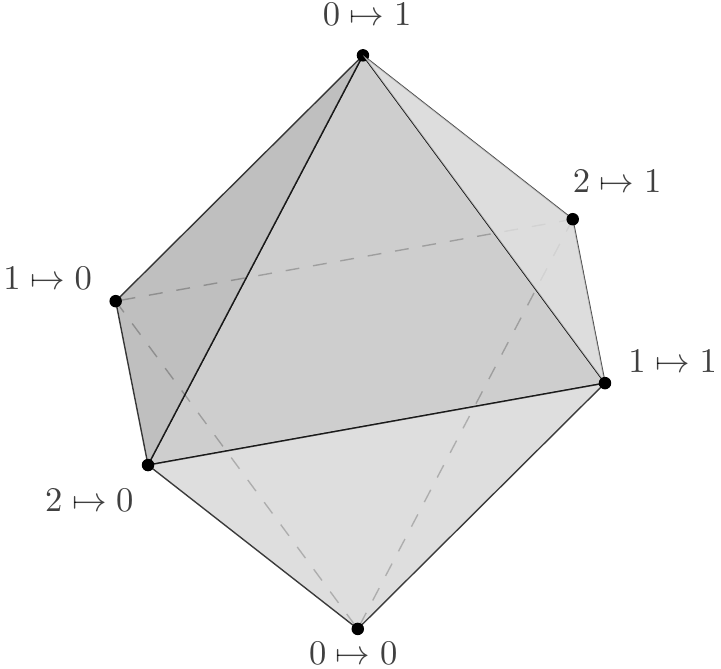}
\caption{The canonical suboplex of $[3 \to 2]$.}
\label{fig:canonical_suboplex_cAll3}
\end{subfigure}
~
\begin{subfigure}[t]{.3\textwidth}
\centering
\includegraphics[height=.18\textheight]{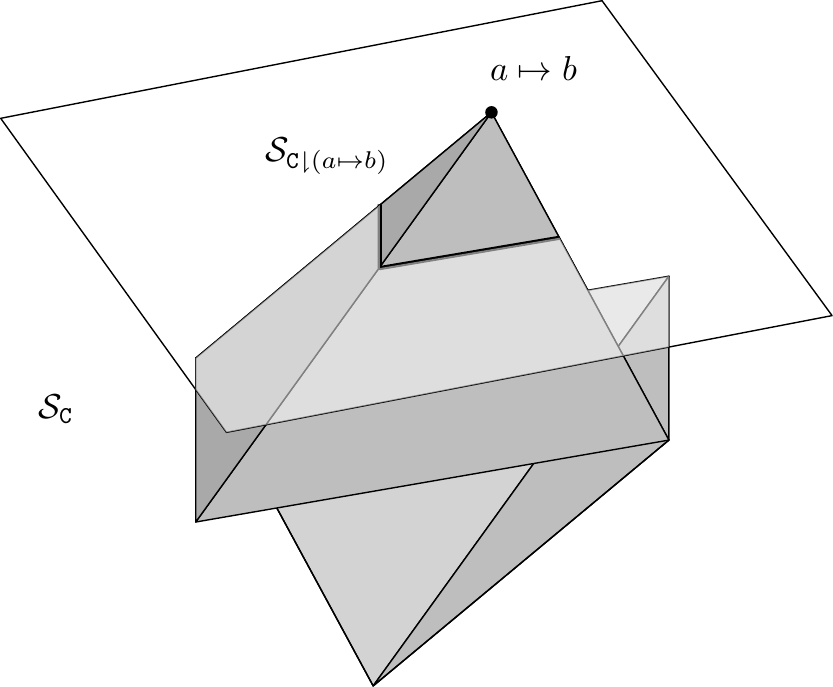}
\caption{$\Ss_{\clsf C \filt (a \mapsto b)}$ is an affine section of $\Sbpx C$.}
\label{fig:filtered_subcomplex-linear_section}
\end{subfigure}
~
\begin{subfigure}[t]{.3\textwidth}
\centering
\includegraphics[height=.16\textheight]{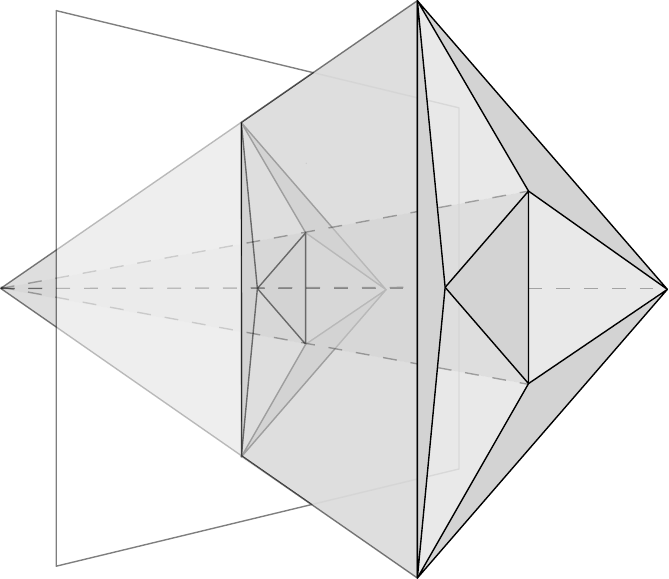}
\caption{we may recover $\Ss_{\linfun_d'}$ as a linear cut through the ``torso'' of $\Ss_{\linfun_d}$.}
\label{fig:linfun2_section}
\end{subfigure}
\caption{}
\end{figure}

\makeconvo{``OK. I see the holes. But how does this have anything to do with our proof of $\indone \not \in \linfun_d$?''}
Hold on tight! We are almost there.

First let me introduce a ``duality principle'' in algebraic topology called the {\bf Nerve Lemma}.
Readers familiar with it can skip ahead to the next section.

\subsection{Nerve Lemma}
Note that the canonical suboplex of $\linfun_2'$ can be continuously deformed as shown in \cref{fig:linfun2_nerve_deform} into a 1-dimensional complex (a graph), so that all of the holes are still preserved.
Such a deformation produces a complex 
\begin{itemize*}
    \item whose vertices correspond exactly to the facets of the original complex, and
    \item whose edges correspond exactly to intersections of pairs of facets,
\end{itemize*}
all the while preserving the holes of the original complex, and producing no new ones.
\begin{figure}
\centering
\includegraphics[height=.15\textheight]{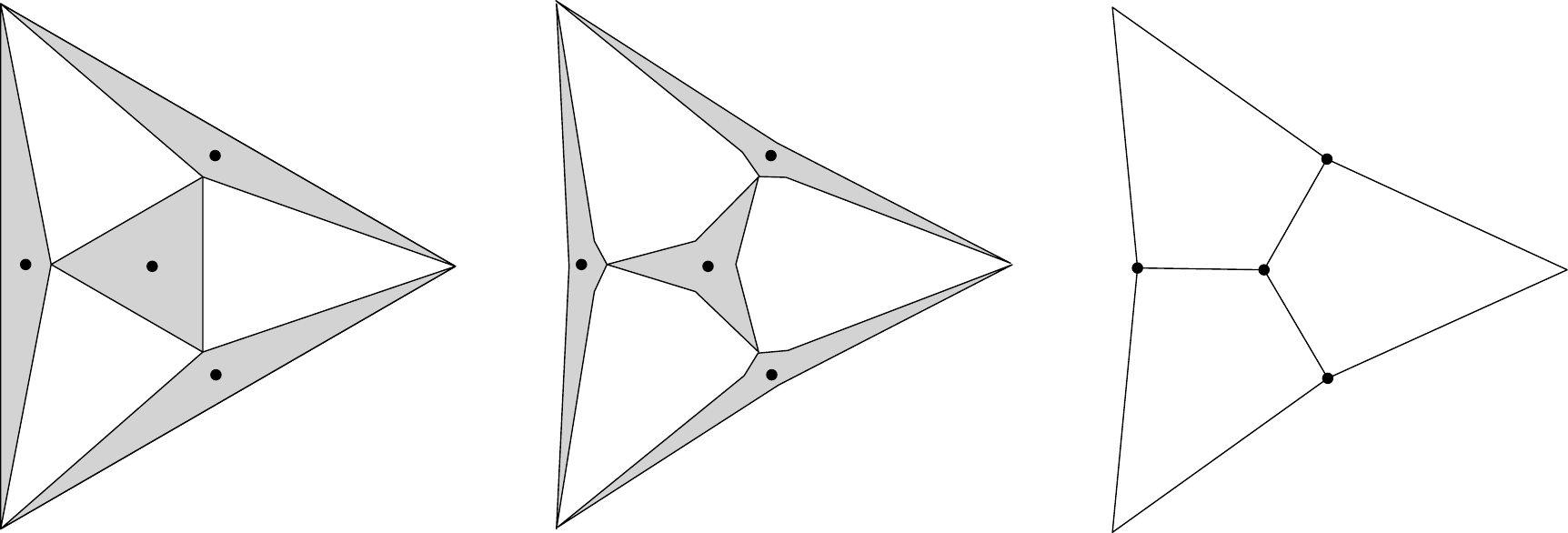}
\caption{A continuous deformation of $\Ss_{\linfun_2'}$ into a complete graph with 4 vertices (where we ignore the sharp bends of the ``outer'' edges).}
\label{fig:linfun2_nerve_deform}
\end{figure}

Such an intuition of deformation is vastly generalized by the Nerve Lemma:
\begin{lemma}[Nerve Lemma (Informal)]
Let $\mathcal U = \{U_i\}_i$ be a \textnormal{``nice''} cover (to be explained below) of a topological space $X$.
The {\bf nerve} $\mathcal N_{\mathcal U}$ of $\mathcal U$ is defined as the simplicial complex with vertices $\{V_i: U_i \in \mathcal U\}$, and with simplices $\{V_i\}_{i \in S}$ for each index set $S$ such that $\bigcap \{U_i: i \in S\}$ is nonempty.

Then, for each dimension $d$, the set of $d$-dimensional holes in $X$ is bijective with the set of $d$-dimensional holes in $\mathcal N_{\mathcal U}$.
\end{lemma}

\begin{wrapfigure}{r}{0.2\textwidth} 
\centering
\includegraphics[height=0.1\textheight]{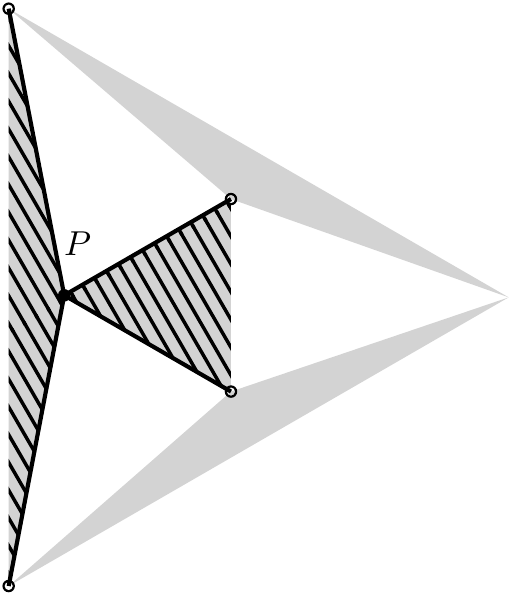}
\caption{The open star $\St P$ of vertex $P$}
\label{fig:linfun2_star_example}
\end{wrapfigure}
What kind of covers are nice?
Open covers in general spaces, or subcomplex covers in simplicial (or CW) complexes, are considered ``nice'', if in addition they satisfy the following requirements ({\it acyclicity}).
\begin{itemize}
    \item Each set of the cover must have no holes.
    \item Each nontrivial intersection of a collection of sets must have no holes.
\end{itemize}
The example we saw in \cref{fig:linfun2_star_example} is an application of the Nerve Lemma for the cover by facets.
Another example is the star cover:
For vertex $V$ in a complex, the {\bf open star} $\St V$ of $V$ is defined as the union of all open simplices whose closure meets $V$ (see \cref{fig:linfun2_star_example} for an example).
If the cover $\mathcal U$ consists of the open stars of every vertex in a simplicial complex $X$, then $\mathcal N_{\mathcal U}$ is isomorphic to $X$ as complexes.

\makeconvo{
OK! We are finally ready to make the connection to complexity!}

\subsection{The Connection}
It turns out that $\Sbpx{\linfun_d'} = \Ss_{\linfun_d \filt (\mathbf 0 \mapsto 0)}$ (a complex of dimension $2^d - 2$) has holes in dimension $d-1$. 
The proof is omitted here but will be given in \cref{sec:linear_functionals}.
This can be clearly seen in our example when $d = 2$ (\cref{fig:linfun2_stretched_apart}), which has 3 holes in dimension $d-1 = 1$.
Furthermore, for every partial linear functional $\pf h$ (a linear functional defined on a linear subspace), $\Ss_{\linfun_d \filt \pf h}$ also has holes, in dimension $d - 1 - \dim(\dom \pf h)$.
\cref{fig:linfun2_linear_section} show an example for $d = 2$ and $\pf h = [1 \ 1]^T \mapsto 1$.
\begin{figure*}
\centering
\begin{subfigure}[t]{.45\textwidth}
\centering
\includegraphics[height=.16\textheight]{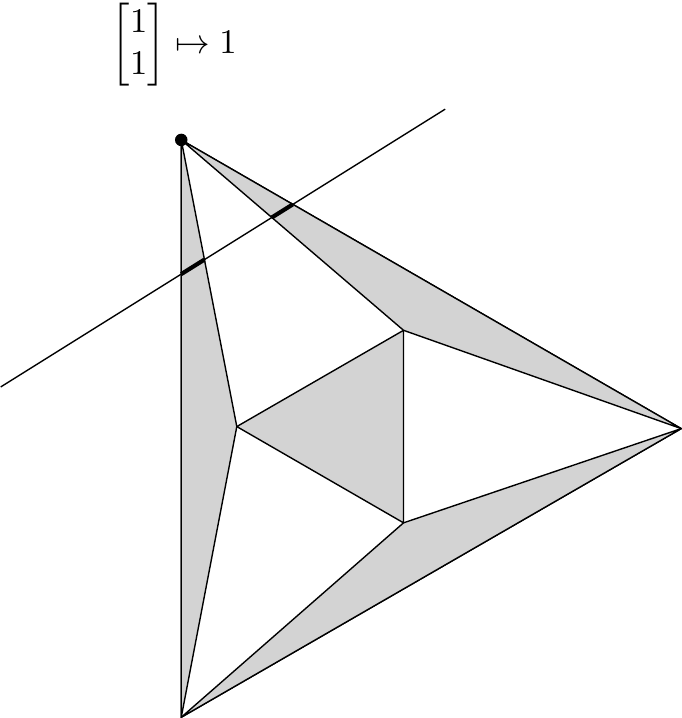}
\caption{The canonical suboplex of $\linfun_2 \filt \{[0\ 0]^T \mapsto 0, [1\ 1]^T \mapsto 1\}$ is isomorphic to the affine section as shown, and it has two disconnected components, and thus ``a single zeroth dimensional hole.''
}
\label{fig:linfun2_linear_section}
\end{subfigure}
~
\begin{subfigure}[t]{.45\textwidth}
\centering
\includegraphics[height=.16\textheight]{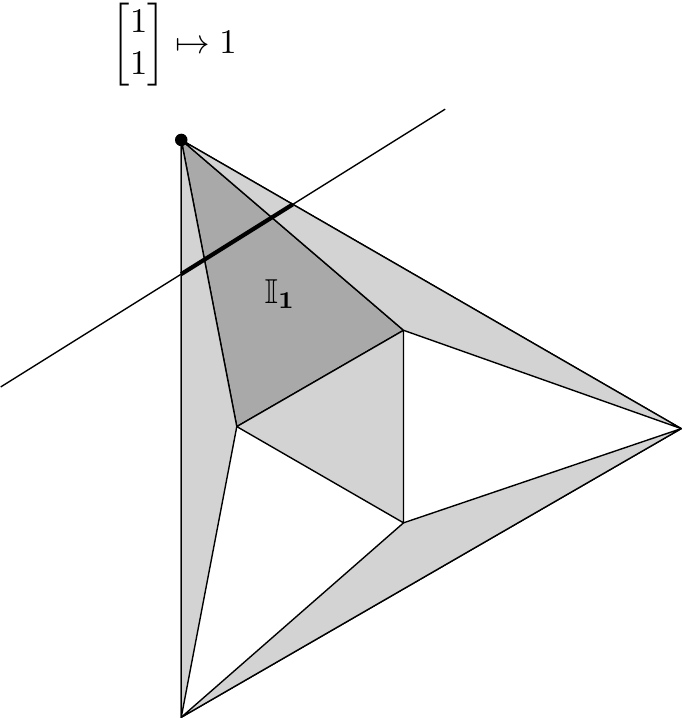}
\caption{When we add $\indone$ to $\linfun_d$ to obtain $\clsf D := \linfun_d \cup \{\indone\}$, $\Ss_{\clsf D \filt \pf g}$ now does not have any hole!}
\label{fig:linfun2_add_f}
\end{subfigure}
\caption{}
\end{figure*}

But when we add $\indone$ to $\linfun_d$ to obtain $\clsf D := \linfun_d \cup \{\indone\}$, $\Ss_{\clsf D \filt \pf g}$ now does not have any hole!
\cref{fig:linfun2_add_f} clearly demonstrates the case $d = 2$.
For general $d$, note that $\Ss_{\linfun_d'}$ has a ``nice'' cover by the open stars
$$\mathcal C := \{\St V: V \text{ has label $u \mapsto r$ for some $u \in \Fld_2^d \setminus \{\mathbf 0\}$ and $r \in \Fld_2$}\}.$$
When we added $\indone$ to form $\clsf D$, the collection $\mathcal C' := \mathcal C \cup \splx_{\indone}$ obtained by adding the simplex of $\indone$ to $\mathcal C$ is a ``nice'' cover of $\Ss_{\clsf D}$. 
Thus the nerve $\mathcal N_{\mathcal C'}$ has the same holes as $\Ss_{\clsf D}$, by the Nerve Lemma.
But observe that $\mathcal N_{\mathcal C'}$ is a cone!
\ldots which is what our ``combinatorial proof'' of $\indone \not \in \linfun_d$ really showed.

\begin{wrapfigure}{r}{.3\textwidth}
\centering
\includegraphics[height=0.16\textheight]{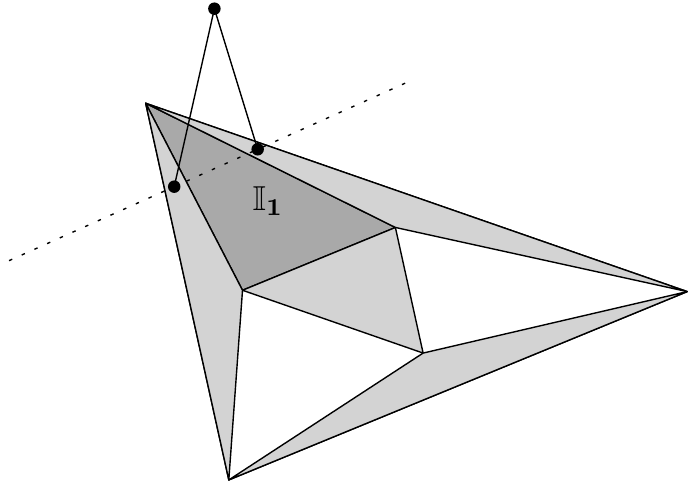}
\caption{The nerve $\mathcal N_{\mathcal C'}$ overlayed on $\clsf D = \linfun_2 \cup \{\indone\}$.
Note that $\mathcal N_{\mathcal C'}$ is a cone over its base of 2 points.}
\label{fig:linfun2_nerve_is_cone}
\end{wrapfigure}
More precisely, 
\begin{enumerate*}[label=\!]
    \item a collection of stars $S := \{\St V: V \in \mathcal V\}$ has nontrivial intersection iff there is a partial linear functional extending the labels of each $V \in \mathcal V$.
    \item We showed $\indone$ intersects every partial linear functional strictly extending $\pf g: \mathbf 0 \mapsto 0, \mathbf 1 \mapsto 1$.
    \item Therefore, a collection of stars $S$ in $\mathcal C'$ intersects nontrivially iff $\bigcap (S \cup \{\splx_{\indone}\}) \not=\emptyset$.
\end{enumerate*}

In other words, in the nerve of $\mathcal C'$, $\splx_{\indone}$ forms the vertex of a cone over all other $\St V \in \mathcal C$.
In our example of $\linfun_2$, this is demonstrated in \cref{fig:linfun2_nerve_is_cone}.

Thus, to summarize,
\begin{itemize*}
    \item $\mathcal N_{\mathcal C'}$, being a cone, has no holes.
    \item By the Nerve Lemma, $\Ss_{\clsf D \filt \pf g}$ has no holes either.
    \item Since $\Ss_{\linfun_d \filt \pf g}$ has holes, we know $\clsf D \not = \linfun_d$, i.e. $\indone \not \in \linfun_d,$ as desired. 
\end{itemize*}

While this introduction took some length to explain the logic of our approach, much of this is automated in the theory we develop in this paper, which leverages existing works on Stanley-Reisner theory and cellular resolutions. 

\begin{center}
***
\end{center}
In our proof, we roughly did the following
\begin{itemize}
    \item (Local) Examined the intersection of $\indone$ with fragments of functions in $\linfun_d$.
    \item (Global) Pieced together the fragments with nontrivial intersections with $\indone$ to draw conclusions about the ``holes'' $\indone$ creates or destroys.
\end{itemize}

This is the {\it local-global philosophy} of this homological approach to complexity, inherited from algebraic topology.
This is markedly different from conventional wisdom in computer science, which seeks to show that a function, such as $f = \textsc{3sat}$, has some property that no function in a class, say $\clsf C = \P$, has.
In that method, there is no {\it global} step that argues that some global property of $\clsf C$ changes after adding $f$ into it.

Using our homological technique, we show, in \cref{chap:app}, a proof of Minsky and Papert's classical result that the class $\polythr_d^k$ of polynomial thresholds of degree $k$ in $d$ variables does not contain the parity function $\parity_d$ unless $k=d$ (\cref{exmp:parity_polythr}).
Homologically, there are many reasons.
By considering high dimensions, we deduce that $\Ss_{\polythr_d^k}$ has a hole in dimension $\sum_{i=0}^k \binom{d}{i}$ that is filled in by $\parity_d$.
By considering low dimensions, we obtain a {\bf maximal principle} for polynomial threshold functions from which we obtain not only Minsky and Papert's result but also extensions to arbitrary symmetric functions.
This maximal principle \cref{thm:maximal_principle_thr} says
\begin{thm}[Maximal Principle for Polynomial Threshold]
Let $\clsf C := \polythr^k_d$, and let $f: \{-1, 1\}^d \to \{-1, 1\}$ be a function.
We want to know whether $f \in \clsf C$.

Suppose there exists a function $g \in \clsf C$ (a ``local maximum'' for approximating $g$) such that
\begin{itemize}
    \item for each $h \in \clsf C$ that differs from $g$ on exactly one input $u$, we have $g(u) = f(u) = \neg h(u)$.
\end{itemize}
If $g \not = f$, then $f \not \in \clsf C$. (In other words, if $f \in \clsf C$, then the ``local maximum'' $g$ must be a ``global maximum'').
\end{thm}
Notice that the maximal principle very much follows the local-global philosophy.
The ``local maximum'' condition is saying that when one looks at the intersection with $f$ of $g$ and its ``neighbors'' (local), these intersections together form a hole that $f$ creates when added to $\clsf C$ (global).
The homological intuition, in more precise terms, is that a local maximum $g \not= f\in \clsf C$ implies that the filtered class $\clsf C \filt (f \cap g)$ consists of a single point with label $g$, so that when $f$ is added to $\clsf C$, a zero-dimensional hole is created.

We also obtain an interesting characterization of when a function can be {\it weakly represented} by a degree bounded polynomial threshold function.
A real function $\varphi: U \to \R$ on a finite set $U$ is said to {\bf weakly represent} a function $f: U \to \{-1, 1\}$ if $\varphi(u) > 0 \iff f(u) = 1$ and $\varphi(u) < 0 \iff f(u) = -1$, but we don't care what happens when $\varphi(u) = 0$.
Our homological theory of function essentially says that $f \in \polythr_d^k$ (``$f$ is strongly representable by a polynomial of degree $k$'') iff $\Ss_{\polythr_d^k \cup \{f\} \filt \pf g}$ has the same number of holes as $\Ss_{\polythr_d^k \filt \pf g}$ in each dimension and for each $\pf g$.
But, intriguingly, $f$ is {\it weakly representable} by a polynomial of degree $k$ iff $\Ss_{\polythr_d^k \cup \{f\}}$ has the same number of holes as $\Ss_{\polythr_d^k }$ in each dimension (\cref{cor:weak_rep_betti}) --- in other words, we only care about filtering by $\pf g = \emptyfun$ but no other partial functions.

\subsection{Dimension theory}

Let $\clsf C \sbe [n \to 2]$.
The {\bf VC Dimension} $\vcdim \clsf C$ of $\clsf C$ is the size of the largest set $U \sbe [n]$ such that $\clsf C \restrict U = \{0, 1\}^U$.

Consider the following setting of a learning problem:
You have an oracle, called the sample oracle, such that every time you call upon it, it will emit a sample $(u, h(u))$ from an unknown distribution $P$ over $u \in [n]$, for a fixed $h \in \clsf C$.
This sample is independent of all previous and all future samples.
Your task is to learn the identity of $h$ with high probability, and with small error (weighted by $P$).

A central result of statistical learning theory says roughly that
\begin{thm}[\cite{kearns_introduction_1994}]
In this learning setting, one only needs $O(\vcdim \clsf C)$ samples to learn $h \in \clsf C$ with high probability and small error. 
\end{thm}
It is perhaps surprising, then, that the following falls out of our homological approach.
\begin{thm}[Colloquial version of \cref{thm:hdim_ge_vcdim}]
Let $\clsf C \sbe [n \to 2]$.
Then $\vcdim \clsf C$ is upper bounded by one plus the highest dimension, over any partial function $\pf g$, of any hole in $\Ss_{\clsf C \filt \pf g}$.
This quantity is known as the {\bf homological dimension} $\hdim \clsf C$ of $\clsf C$.
\end{thm}

In fact, equality holds for common classes in the theory of computation like $\linfun_d$ and $\polythr_d^k$, and also when certain algebraic conditions hold.
More precisely --- for readers with algebraic background ---
\begin{thm}[Colloquial version of \cref{cor:CM_implies_hdim=vcdim}]
$\vcdim \clsf C = \hdim \clsf C$ if the Stanley-Reisner ring of $\Sbpx C$ is Cohen-Macaulay.
\end{thm}


These results suggest that our homological theory captures something essential about computation, that it's not a coincidence that we can use ``holes'' to prove complexity separation.

\subsection{Homological Farkas}
Farkas' Lemma is a simple result from linear algebra, but it is an integral tool for proving weak and strong dualities in linear programming, matroid theory, and game theory, among many other things.

\begin{lemma}[Farkas' Lemma]
Let $L \sbe \R^n$ be a linear subspace not contained in any coordinate hyperplanes, and let $P = \{x \in \R^n: x > 0\}$ be the positive cone.
Then either 
\begin{itemize}
    \item $L$ intersects $P$, or
    \item $L$ is contained in the kernel of a nonzero linear functional whose coefficients are all nonnegative.
\end{itemize}
but not both.
\end{lemma}

Farkas' Lemma is a characterization of when a linear subspace intersects the positive cone in terms of {\it linear conditions}.
An alternate view important in computer science is that Farkas' Lemma provides a {\it linear certificate} for when this intersection does not occur.
Analogously, our Homological Farkas' Lemma will characterize such an intersection in terms of {\it homological conditions}, and simultaneously provide a {\it homological certificate} for when this intersection does not occur.

Before stating the Homological Farkas' Lemma, we first introduce some terminology.
\newcommand{\onef}{\mathbf{1}}

For $g: [n] \to \{1, -1\}$, let $P_g \sbe \R^n$ denote the open cone whose points have signs given by $g$.
Consider the intersection $\splx_g$ of $\setcl{P_g}$ with the unit sphere $S^{n-1}$ and its interior $\ocell \splx_g$.
$\ocell\splx_g$ is homeomorphic to an open simplex.
For $g \not = \neg \onef$, define $\Lambda(g)$ to be the union of the facets $F$ of $\splx_g$ such that $\ocell\splx_g$ and $\ocell\splx_{\mathbbm 1}$ sit on opposite sides of the affine hull of $F$.
Intuitively, $\Lambda(g)$ is the part of $\pd \splx_g$ that can be seen from an observer in $\ocell\splx_{\onef}$ (illustrated by \cref{fig:lambda_illustration}).

\begin{figure}
\centering
\begin{subfigure}[t]{.3\textwidth}
\centering
\includegraphics[height=0.16\textheight]{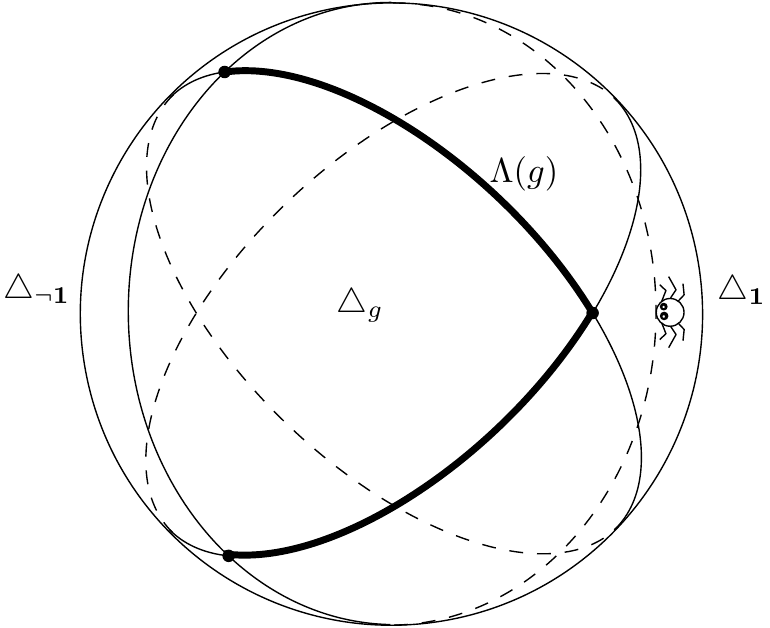}
\caption{An example of a $\Lambda(g)$.
Intuitively, $\Lambda(g)$ is the part of $\pd \splx_g$ that can be seen from an observer in $\splx_{\onef}$.}
\label{fig:lambda_illustration}
\end{subfigure}
~
\begin{subfigure}[t]{.6\textwidth}
\centering
\includegraphics[height=0.16\textheight]{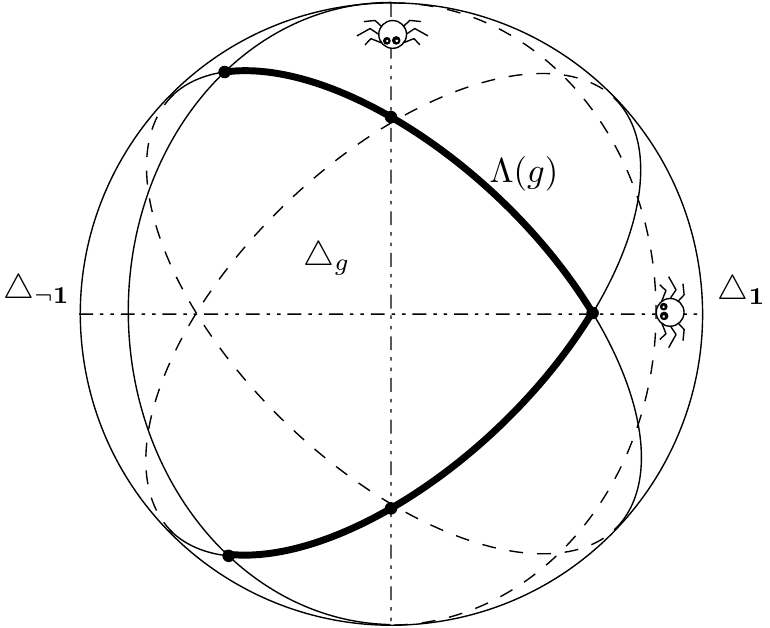}
\caption{An illustration of Homological Farkas' Lemma.
The horizontal dash-dotted plane intersects the interior of $\splx_\onef$, but its intersection with any of the $\Lambda(f), f\not=\onef, \neg\onef$ has no holes.
The vertical dash-dotted plane misses the interior of $\splx_\onef$, and we see that its intersection with $\Lambda(g)$ as shown has two disconnected components.}
\label{fig:homological_farkas_illustration}
\end{subfigure}
\caption{}
\end{figure}
The following homological version of Farkas' Lemma naturally follows from our homological technique of analyzing the complexity of threshold functions.
\begin{thm}[Homological Farkas' Lemma \cref{thm:hom_farkas}]\label{thm:informal_hom_farkas}
Let $L \sbe \R^n$ be a linear subspace.
Then either
\begin{itemize}
    \item $L$ intersects the positive cone $P = P_\onef$, or
    \item $L \cap \Lambda(g)$ for some $g \not = \onef, \neg \onef$ is nonempty and has holes.
\end{itemize}
but not both.
\end{thm}

\cref{fig:homological_farkas_illustration} illustrates an example application of this result.

One direction of the Homological Farkas' Lemma has the following intuition.
As mentioned before, $\Lambda(g)$ is essentially the part of $\pd\splx_g$ visible to an observer Tom in $\ocell\splx_{\onef}$.
Since the simplex is convex, the image Tom sees is also convex.
Suppose Tom sits right on $L$ (or imagine $L$ to be a subspace of Tom's visual field).
If $L$ indeed intersects $\ocell\splx_{\onef}$, then for $L \cap \Lambda(g)$ he sees some affine space intersecting a convex body, and hence a convex body in itself.
Since Tom sees everything (i.e. his vision is homeomorphic with the actual points), $L \cap \Lambda(g)$ has no holes, just as Tom observes.

In other words, if Tom is inside $\ocell\splx_{\onef}$, then he cannot tell $\Lambda(g)$ is nonconvex by his vision alone, for any $g$.
Conversely, the Homological Farkas' Lemma says that if Tom is outside of $\ocell\splx_{\onef}$ and if he looks away from $\ocell \splx_{\onef}$, he will always see a nonconvex shape in some $\Lambda(g)$.

As a corollary to \cref{thm:informal_hom_farkas}, we can also characterize when a linear subspace intersects a region in a linear hyperplane arrangement (\cref{cor:linear_hom_farkas}), and when an affine subspace intersects a region in an affine hyperplane arrangement (\cref{cor:affine_glue_hom_farkas}), both in terms of homological conditions.
A particular simple consequence, when the affine subspace either intersects the interior or does not intersect the closure at all, is illustrated in \cref{fig:intro_generalized_homological_farkas}.

\begin{figure}
\centering
\includegraphics[height=.2\textheight]{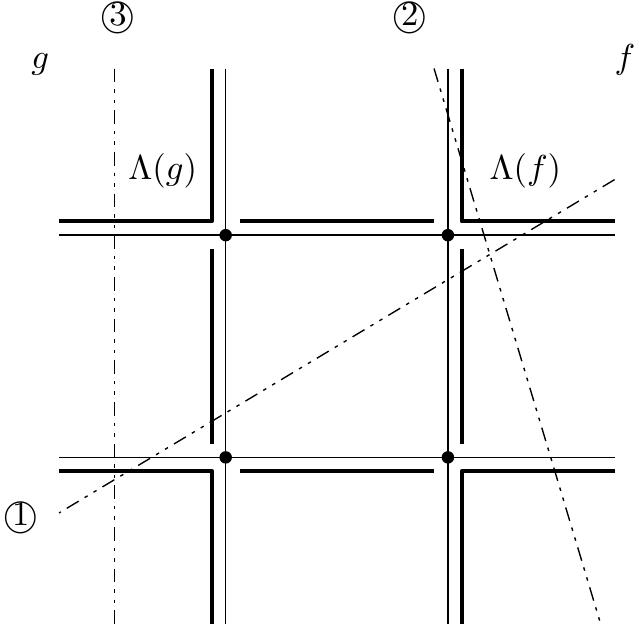}
\caption{Example application of \cref{cor:simple_affine_hom_farkas}.
Let the hyperplanes (thin lines) be oriented such that the square $S$ at the center is on the positive side of each hyperplane.
The bold segments indicate the $\Lambda$ of each region.
Line 1 intersects $S$, and we can check that its intersection with any bold component has no holes.
Line 2 does not intersect the closure $\setcl{S}$, and we see that its intersection with $\Lambda(f)$ is two points, so has a ``zeroth dimension'' hole.
Line 3 does not intersect $\setcl{S}$ either, and its intersection with $\Lambda(g)$ consists of a point in the finite plane and another point on the circle at infinity.}
\label{fig:intro_generalized_homological_farkas}
\end{figure}

\vskip 10mm

The rest of this paper is organized as follows. 
\cref{chap:theory} builds the theory underlying our complexity separation technique.
\cref{sec:background} explains some of the conventions we adopt in this work and more importantly reviews basic facts from combinatorial commutative algebra and collects important lemmas for later use.
\cref{sec:canonical_ideal} defines the central objects of study in our theory, the {\it Stanley-Reisner ideal} and the {\it canonical ideal} of each function class.
The section ends by giving a characterization of when an ideal is the Stanley-Reisner ideal of a class.
\cref{sec:resolutions} discusses how to extract homological data of a class from its ideals via cellular resolutions.
We construct cellular resolutions for the canonical ideals of many classes prevalent in learning theory, such as conjunctions, linear thresholds, and linear functionals over finite fields.
\cref{sec:PF_classes} briefly generalizes definitions and results to partial function classes, which are then used in \cref{sec:combining}.
This section explains, when combining old classes to form new classes, how to also combine the cellular resolutions of the old classes into cellular resolutions of the new classes.

\cref{chap:app} reaps the seeds we have sowed so far. \cref{sec:dimension} looks at notions of dimension, the Stanley-Reisner dimension and the homological dimension, that naturally appear in our theory and relates them to VC dimension, a very important quantity in learning theory.
We observe that in most examples discussed in this work, the homological dimension of a class is almost the same as its VC dimension, and prove that the former is always at least the latter.
\cref{sec:CM} characterizes when a class has Stanley-Reisner ideal and canonical ideal that induce Cohen-Macaulay rings, a very well studied type of rings in commutative algebra.
We define Cohen-Macaulay classes and show that their homological dimensions are always equal to their VC dimensions.
\cref{sec:separation} discusses separation of computational classes in detail, and gives simple examples of this strategy in action.
Here a consequence of our framework is the Homological Farkas Lemma.
\cref{sec:maximal_principle} formulates and proves the maximal principle for threshold functions, and derives an extension of Minsky and Papert's result for general symmetric functions.
\cref{sec:hom_farkas} further extends Homological Farkas Lemma to general linear or affine hyperplane arrangements.
\cref{sec:hilbert_function} examines a probabilistic interpretation of the Hilbert function of the canonical ideal, and shows its relation to hardness of approximation.

Finally, \cref{sec:future_work} considers major questions of our theory yet to be answered and future directions of research.

\section{Theory}
\label{chap:theory}

\subsection{Background and Notation}
\label{sec:background}
In this work, we fix $\kk$ to be an arbitrary field.
We write $\N = \{0, 1, \ldots, \}$ for the natural numbers.
Let $n, m \in \N$ and $A, B$ be sets.
The notation $\pf f:\sbe A \to B$ specifies a partial function $\pf f$ whose domain $\dom \pf f$ is a subset of $A$, and whose codomain is $B$.
The words ``partial function'' will often be abbreviated ``PF.''
We will use Sans Serif font for partial (possibly total) functions, ex. $\pf f, \pf g, \pf h$, but will use normal font if we know a priori a function is total, ex. $f, g, h$.
We denote the empty function, the function with empty domain, by $\emptyfun$.
We write $[n]$ for the set $\{0, 1, \ldots, n-1\}$.
We write $[A \to B]$ for the set of total functions from $A$ to $B$ and $[\sbe A \to B]$ for the set of partial functions from $A$ to $B$.
By a slight abuse of notation, $[n \to m]$ (resp. $[\sbe n \to m]$ is taken to be a shorthand for $[[n] \to [m]]$ (resp. $[\sbe [n] \to [m]]$).
The set $[2^d]$ is identified with $[2]^d$ via binary expansion (ex: $5 \in [2^4]$ is identified with $(0, 1, 0, 1) \in [2]^4$).
A subset of $[A \to B]$ (resp. $[\sbe A \to B]$) is referred to as a {\it class} (resp. {\it partial class}), and we use $\clsf C, \clsf D$ (resp. $\pclsf C, \pclsf D$), and so on to denote it.
Often, a bit vector $v \in [2^d]$ will be identified with the subset of $[d]$ of which it is the indicator function.

For $A \sbe B$, relative set complement is written $B \setminus A$; when $B$ is clearly the universal set from context, we also write $A^\setc$ for the complement of $A$ inside $B$.
If $\{a, b\}$ is any two-element set, we write $\neg a = b$ and $\neg b = a$. 

Denote the $n$-dimensional simplex $\{\es v \in \R^n: \sum_i v_i = 1\}$ by $\splx^n$.
Let $X, Y$ be topological spaces (resp. simplicial complexes, polyhedral complexes).
The join of $X$ and $Y$ as a topological space (resp. simplicial complex, polyhedral complex) is denoted by $X \star Y$.
We abbreviate the quotient $X/\pd X$ to $X/\pd$.

We will use some terminologies and ideas from matroid theory in \cref{sec:threshold_functions} and \cref{sec:separation}.
Readers needing more background can consult the excellently written chapter 6 of \cite{ziegler_lectures_1995}.

\subsubsection{Combinatorial Commutative Algebra}
Here we review the basic concepts of combinatorial commutative algebra.
We follow \cite{miller_combinatorial_2005} closely.
Readers familiar with this background are recommended to skip this section and come back as necessary; the only difference in presentation from \cite{miller_combinatorial_2005} is that we say a labeled complex {\it is} a cellular resolution when in more conventional language it {\it supports} a cellular resolution.

Let $\kk$ be a field and $S = \kk[\xx]$ be the polynomial ring over $\kk$ in $n$ indeterminates $\xx = x_0, \ldots, x_{n-1}$.

\begin{defn}
A \textbf{monomial} in $\kk[\xx]$ is a product $\xx^{\es a} = x_0^{a_0}\cdots x_{n-1}^{a_{n-1}}$ for a vector $\es a = (a_0, \ldots, a_{n-1}) \in \N^n$ of nonnegative integers.
Its {\bf support} $\supp \xx^{\es a}$ is the set of $i$ where $a_i \not = 0$.
We say $\xx^{\es a}$ is squarefree if every coordinate of $\es a$ is 0 or 1.
We often use symbols $\sigma, \tau$, etc for squarefree exponents, and identify them with the corresponding subset of $[n]$.

An ideal $I \sbe \kk[\xx]$ is called a {\bf monomial ideal} if it is generated by monomials, and is called a {\bf squarefree monomial ideal} if it is generated by squarefree monomials.
\end{defn}

Let $\Delta$ be a simplicial complex.
\begin{defn}
The {\bf Stanley-Reisner ideal} of $\Delta$ is defined as the squarefree monomial ideal
$$I_\Delta = \la \xx^\tau: \tau \not\in \Delta\ra$$
generated by the monomials corresponding the nonfaces $\tau$ of $\Delta$.
The {\bf Stanley-Reisner ring} of $\Delta$ is the quotient ring $S/I_\Delta$.
\end{defn}

\begin{defn}
The squarefree {\bf Alexander dual} of squarefree monomial ideal $I = \la \xx^{\sigma_1}, \ldots, \xx^{\sigma_r}\ra$ is defined as
$$I^\star = \mm^{\sigma_1} \cap \cdots \cap \mm^{\sigma_r}.$$
If $\Delta$ is a simplicial complex and $I = I_\Delta$ its Stanley-Reisner ideal, then the simplicial complex $\Delta^\star$ {\bf Alexander dual} to $\Delta$ is defined by $I_{\Delta^\star} = I_{\Delta}^\star$. 
\end{defn}

\begin{prop}[Prop 1.37 of \cite{miller_combinatorial_2005}]
The Alexander dual of a Stanley-Reisner ideal $I_\Delta$ can in fact be described as the ideal $\la \xx^\tau: \tau^\setc \in \Delta\ra$, with minimal generators $\xx^\tau$ where $\tau^\setc$ is a facet of $\Delta$.
\label{prop:alexdual_mingen}
\end{prop}

\begin{defn}
The {\bf link} of $\sigma$ inside the simplicial complex $\Delta$ is
$$\link_\sigma \Delta = \{\tau \in \Delta: \tau \cup \sigma \in \Delta \And \tau \cap \sigma = \emptyset\},$$
the set of faces that are disjoint from $\sigma$ but whose unions with $\sigma$ lie in $\Delta$.
\end{defn}

\begin{defn}
The {\bf restriction} of $\Delta$ to $\sigma$ is defined as
$$\Delta\restrict \sigma = \{\tau \in \Delta : \tau \sbe \sigma\}.$$
\end{defn}

\begin{defn}
A sequence
$$\resf F_\bullet: 0 \gets F_0 \xlar{\phi_1} F_1 \gets \cdots \gets F_{l-1} \xlar{\phi_l} F_l \gets 0$$
of maps of free $S$-modules is called a {\bf complex} if $\phi_i \circ \phi_{i+1} = 0$ for all $i$.
The complex is {\it exact} in homological degree $i$ if $\ker \phi_i = \im \phi_{i+1}$.
When the free modules $F_i$ are $\N^n$-graded, we require that each homomorphism $\phi_i$ to be degree-preserving.

Let $M$ be a finitely generated $\N^n$-graded module $M$.
We say $\resf F_\bullet$ is a {\bf free resolution} of $M$ over $S$ if $\resf F_\bullet$ is exact everywhere except in homological degree 0, where $M = F_0/\im \phi_1$.
The image in $F_i$ of the homomorphism $\phi_{i+1}$ is the $i$th {\bf syzygy module} of $M$.
The {\bf length} of $\resf F_\bullet$ is the greatest homological degree of a nonzero module in the resolution, which is $l$ here if $F_l \not= 0$.
\end{defn}

The following lemma says that if every minimal generator of an ideal $J$ is divisible by $x_0$, then its resolutions are in bijection with the resolutions of $J/x_0$, the ideal obtained by forgetting variable $x_0$.
\begin{lemma} \label{lemma:forget_var}
Let $I \sbe S = \kk[x_0, \ldots, x_{n-1}]$ be a monomial ideal generated by monomials not divisible by $x_0$.
A complex
$$\resf F_\bullet: 0 \gets F_0 \gets F_1 \gets \cdots \gets F_{l-1} \gets F_l \gets 0$$
resolves $x_0 I$ iff for $S/x_0 = \kk[x_1, \ldots, x_{n-1}]$,
$$\resf F_\bullet \otimes_S S/x_0: 0 \gets F_0/x_0 \gets F_1/x_0 \gets \cdots \gets F_{l-1}/x_0 \gets F_l/x_0 \gets 0$$
resolves $I \otimes_S S/x_0$.
\end{lemma}

\begin{defn}
Let $M$ be a finitely generated $\N^n$-graded module $M$ and
$$\resf F_\bullet: 0 \gets F_0 \gets F_1 \gets \cdots \gets F_{l-1} \gets F_l \gets 0$$
be a minimal graded free resolution of $M$.
If $F_i = \bigoplus_{\es a \in\N^n} S(-\es a)^{\beta_{i, \es a}}$, then the $i$th {\bf Betti number} of $M$ in degree $\es a$ is the invariant $\betti_{i, \es a} = \betti_{i, \es a}(M)$.
\end{defn}

\begin{prop}[Lemma 1.32 of \cite{miller_combinatorial_2005}] \label{prop:Tor_betti}
$\betti_{i, \es a}(M) = \dim_\kk \Tor_i^S(\kk, M)_{\es a}$.
\end{prop}

\begin{prop}[Hochster's formula, dual version] \label{prop:dual_hochster}
All nonzero Betti numbers of $I_\Delta$ and $S/I_\Delta$ lie in squarefree degrees $\sigma$, where
$$\betti_{i, \sigma}(I_\Delta) = \betti_{i+1, \sigma}(S/I_\Delta)
= \dim_{\kk} \widetilde H_{i-1}(\link_{\sigma^\setc}{\Delta^*}; \kk).$$

\end{prop}
\begin{prop}[Hochster's formula]\label{prop:hochster}
All nonzero Betti numbers of $I_\Delta$ and $S/I_\Delta$ lie in squarefree degrees $\sigma$, where
$$\betti_{i-1, \sigma}(I_\Delta) = \betti_{i, \sigma}(S/I_\Delta)
= \dim_{\kk} \widetilde H^{|\sigma|-i-1}(\Delta\restrict \sigma; \kk).$$

\end{prop}
Note that since we are working over a field $\kk$, the reduced cohomology can be replaced by reduced homology, since these two have the same dimension.

Instead of algebraically constructing a resolution of an ideal $I$, one can sometimes find a {\it labeled simplicial complex} whose simplicial chain is a free resolution of $I$.
Here we consider a more general class of complexes, {\it polyhedral cell complexes}, which can have arbitrary polytopes as faces instead of just simplices.

\begin{defn}
A {\bf polyhedral cell complex} $X$ is a finite collection of convex polytopes, called {\it faces} or {\it cells} of $X$, satisfying two properties:
\begin{itemize}
  \item If $\mathcal P$ is a polytope in $X$ and $F$ is a face of $\mathcal P$, then $F$ is in $X$.
  \item If $\mathcal P$ and $\mathcal Q$ are in $X$, then $\mathcal P \cap \mathcal Q$ is a face of both $\mathcal P$ and $\mathcal Q$.
\end{itemize}
In particular, if $X$ contains any point, then it contains the empty cell $\emptycell$, which is the unique cell of dimension $-1$.

Each closed polytope $\mathcal P$ in this collection is called a {\bf closed cell} of $X$; the interior of such a polytope, written $\ocell {\mathcal P}$, is called an {\bf open cell} of $X$.
By definition, the interior of any point polytope is the empty cell.

The complex with only the empty cell is called {\bf the irrelevant complex}.
The complex with no cell at all is called {\bf the void complex}.

The void complex is defined to have dimension $-\infty$; any other complex $X$ is defined to have dimension $\dim(X)$ equal to the maximum dimension of all of its faces. 
\end{defn}
Examples include any polytope or the boundary of any polytope.

Each polyhedral cell complex $X$ has a natural {\it reduced chain complex}, which specializes to the usual reduced chain complex for simplicial complexes $X$.
\begin{defn}
Suppose $X$ is a {\bf labeled cell complex}, by which we mean that its $r$ vertices have {\bf labels} that are vectors $\es a_1, \ldots, \es a_r$ in $\N^r$.
The {\bf label} $\es a_F$ on an arbitrary face $F$ of $X$ is defined as the coordinatewise maximum $\max_{i \in F} \es a_i$ over the vertices in $F$.
The {\bf monomial label} of the face $F$ is $\xx^{\es a_F}$.
In particular, the empty face $\emptycell$ is labeled with the exponent label $\mathbf 0$ (equivalently, the monomial label $1 \in S$).
When necessary, we will refer explicitly to the labeling function $\lambda$, defined by $\lambda(F) = \es a_F$, and express each labeled cell complex as a pair $(X, \lambda)$.
\end{defn}

\begin{defn}\label{defn:cell_res}
Let $X$ be a labeled cell complex.
The {\bf cellular monomial matrix} supported on $X$ uses the reduced chain complex of $X$ for scalar entries, with the empty cell in homological degree 0.
Row and column labels are those on the corresponding faces of $X$.
The {\bf cellular free chain complex} $\resf F_X$ supported on $X$ is the chain complex of $\N^n$-graded free $S$-modules (with basis) represented by the cellular monomial matrix supported on $X$.
The free complex $\resf F_X$ is a {\bf cellular resolution} if it has homology only in degree 0.
We sometimes abuse notation and say $X$ itself is a cellular resolution if $\resf F_X$ is.
\end{defn}
\begin{prop}
Let $(X, \lambda)$ be a labeled complex.
If $\resf F_X$ is a cellular resolution, then it resolves $S/I$ where $I = \la \xx^{\es a_V} : V \in X \text{ is a vertex}\}$. 
$\resf F_X$ is in addition minimal iff for each cell $F$ of $X$, $\lambda(F) \not= \lambda(G)$ for each face $G$ of $F$.
\end{prop}
\begin{prop}
If $X$ is a minimal cellular resolution of $S/I$, then $\betti_{i, \es a}(I)$ is the number of $i$-dimensional cells in $X$ with label $\es a$.  
\end{prop}
Given two vectors $\es a, \es b\in \N^n$, we write $\es a \preceq \es b$ and say $\es a $ precedes $\es b$, $\es b - \es a \in \N^n$.
Similarly, we write $\es a \prec \es b$ if $\es a \preceq \es b$ but $\es a \not = \es b$.
Define $X_{\preceq \es a} = \{F \in X: \es a_F \preceq \es a\}$ and $X_{\prec \es a} = \{F \in X: \es a_F \prec \es a\}$.

Let us say a cell complex is {\it acyclic} if it is either irrelevant or has zero reduced homology.
In the irrelevant case, its only nontrivial reduced homology lies in degree $-1$.
\begin{lemma}[Prop 4.5 of \cite{miller_combinatorial_2005}]
$X$ is a cellular resolution iff $X_{\preceq \es b}$ is acyclic over $\kk$ for all $\es b \in \N^n$.
For $X$ with squarefree monomial labels, this is true iff $X_{\preceq \es b}$ is acyclic over $\kk$ for all $\es b \in [2]^n$. 
When $\resf F_X$ is acyclic, it is a free resolution of the monomial quotient $S/I$ where $I = \la \xx^{\es a_v}: v \in X \text{ is a vertex}\ra$ generated by the monomial labels on vertices.
\label{lemma:acyclicity}
\end{lemma}

It turns out that even if we only have a nonminimal cellular resolution, it can still be used to compute the Betti numbers. 
\begin{prop}[Thm 4.7 of \cite{miller_combinatorial_2005}]
If $X$ is a cellular resolution of the monomial quotient $S/I$, then the Betti numbers of $I$ can be calculated as
$$\betti_{i, \es b}(I) = \dim_{\kk}\widetilde H_{i-1}(X_{\prec \es b}: \kk)$$
as long as $i \ge 1$.
\label{prop:betti_from_cellular_resolution}
\end{prop} 

\cref{lemma:acyclicity} and \cref{prop:betti_from_cellular_resolution} will be used repeatedly in the sequel.

We will also have use for the dual concept of cellular resolutions, {\it cocellular resolutions}, based on the cochain complex of a polyhedral cell complex.

\begin{defn}
Let $X' \sbe X$ be two polyhedral cell complexes.
The cochain complex $\cochain \bullet (X, X'; \kk)$ of the cellular pair $(X, X')$ is defined by the exact sequence
$$0 \to \cochain\bullet (X, X'; \kk) \to \cochain \bullet (X; \kk) \to \cochain{\bullet} (X' ; \kk) \to 0.$$
The $i$th {\bf relative cohomology} of the pair is $H^i(X, X'; \kk) = H^i \cochain \bullet (X, X'; \kk)$.
\end{defn}

\begin{defn}
Let $Y$ be a cell complex or a cellular pair.
Then $Y$ is called {\bf weakly colabeled} if the labels on faces $G \sbe F$ satisfy $\es a_G \succeq \es a_F$.
In particular, if $Y$ has an empty cell, then it must be labeled as well.
$Y$ is called {\bf colabeled} if, in addition, every face label $\es a _G$ equals the join $\bigvee \es a_F$ of all the labels on facets $F \spe G$.
Again, when necessary, we will specifically mention the labeling function $\lambda(F) = \es a_F$ and write the cell complex (or pair) as $(Y, \lambda)$.
\end{defn}
We have the following well known lemma from the theory of CW complexes.
\begin{lemma}\label{lemma:closed_iff_subcomplex}
Let $X$ be a cell complex.
A collection $\mathcal R$ of open cells in $X$ is a subcomplex of $X$ iff $\bigcup \mathcal R$ is closed in $X$. 
\end{lemma}

If $Y = (X, X')$ is a cellular pair, then we treat $Y$ as the collection of (open) cells in $X \setminus X'$, for the reason that $\cochain i(X, X', \kk)$ has as a basis the set of open cells of dimension $i$ in $X \setminus X'$.
As $Y$ being a complex is equivalent to $Y$ being the pair $(Y, \{\})$ (where $\{\}$ is the void subcomplex), in the sense that the reduced cochain complex of $Y$ is isomorphic to the cochain complex of the pair $(Y, \{\})$, we will only speak of cellular pairs from here on when talking about colabeling.

\begin{defn}
Let $Y = (X, A)$ be a cellular pair and $\mathcal U$ a subcollection of open cells of $Y$.
We say $\mathcal U$ is {\bf realized by a subpair} $(X', A') \sbe (X, A)$ (i.e. $X' \sbe X, A' \sbe A$) if $\mathcal U$ is the collection of open cells in $X' \setminus A'$. 
\end{defn}
\begin{defn}
Define $Y_{\preceq \es b}$ (resp. $Y_{\prec \es b}$ and $Y_{\es b}$) as the collection of open cells with label $\preceq \es b$ (resp. $\prec \es b$ and $\es b$).
\end{defn}
We often consider $Y_{\preceq \es b}$, $Y_{\prec \es b}$, and $Y_{\es b}$ as subspaces of $Y$, the unions of their open cells.
\begin{prop}\label{lemma:subcollection_realization}
Let $Y$ be a cellular pair and $\mathcal U = Y_{\preceq \es b}$ (resp. $Y_{\prec \es b}$ and $Y_{\es b}$).
Then $\mathcal U$ is realized by the pair $(\overline{\mathcal U}, \pbd \mathcal U)$, where the first of the pair is the closure of $\mathcal U$ as a subspace in $Y$, and the second is the {\bf partial boundary} $\pbd \mathcal U : = \overline{\mathcal U} \setminus \mathcal U$.
\end{prop}
\begin{proof}
See \cref{sec:omitted_proofs}.
\end{proof}
Note that if $X'$ is the irrelevant complex, then $H^i(X, X'; \kk) = H^i(X; \kk)$, the unreduced cohomology of $X$.
If $X'$ is the void complex, then $H^i(X, X'; \kk) = \rH^i(X; \kk)$, the reduced cohomology of $X$.
Otherwise $X'$ contains a nonempty cell, and it is well known that $H^i(X, X'; \kk) \cong \rH^i(X/X'; \kk)$.
In particular, when $X' = X$, $H^i(X, X'; \kk) \cong \rH^i(\bullet; \kk) = 0$.

\begin{defn}
Let $Y$ be a cellular pair $(X, X')$, (weakly) colabeled.
The {\bf (weakly) cocellular monomial matrix} supported on $Y$ has the cochain complex $\cochain \bullet (Y; \kk)$ for scalar entries, with top dimensional cells in homological degree 0.
Its row and column labels are the face labels on $Y$.
The {\bf (weakly) cocellular free complex} $\resf F^Y$ supported on $Y$ is the complex of $\N^n$-graded free $S$-modules (with basis) represented by the cocellular monomial matrix supported on $Y$.
If $\resf F^Y$ is acyclic, so that its homology lies only in degree 0, then $\resf F^Y$ is a (weakly) {\bf cocellular resolution}.
We sometimes abuse notation and say $Y$ is a (weakly) cocellular resolution if $\resf F^Y$ is.
\end{defn}

\begin{prop}
Let $(Y, \lambda)$ be a (weakly) colabeled complex or pair.
If $\resf F^Y$ is a (weakly) cocellular resolution, then $\resf F^Y$ resolves $I = \la \xx^{\es a_F}: F \text{ is a top dimensional cell of $Y$}\ra$.
It is in addition minimal iff for each cell $F$ of $Y$, $\lambda(F) \not= \lambda(G)$ for each cell $G$ strictly containing $F$.
\end{prop}

We say a cellular pair $(X, X')$ is of dimension $d$ if $d$ is the maximal dimension of all (open) cells in $X \setminus X'$.
If $Y$ is a cell complex or cellular pair of dimension $d$, then a cell $F$ of dimension $k$ with label $\es a_F$ corresponds to a copy of $S$ at homological dimension $d - k$ with degree $\xx^{\es a_F}$.
Therefore,
\begin{prop} \label{prop:cocell_betti}
If $Y$ is a $d$-dimension minimal (weakly) cocellular resolution of ideal $I$, then $\betti_{i, \es a}(I)$ is the number of $(d-i)$-dimensional cells in $Y$ with label $\es a$.
\end{prop}

We have an acyclicity lemma for cocellular resolutions similar to \cref{lemma:acyclicity}
\begin{lemma}
Let $Y = (X, A)$ be a weakly colabeled pair of dimension $d$.
For any $U \sbe X$, write $\overline U$ for the closure of $U$ inside $X$.
$Y$ is a cocellular resolution iff for any exponent sequence $\mathbf a$, $K := Y_{\preceq \mathbf a}$ satisfies one of the following:

1) The partial boundary $\pbd K := \overline K \setminus K$ contains a nonempty cell, and $H^i(\overline K, \pbd K)$ is 0 for all $i \not= d$ and is either 0 or $\kk$ when $i = d$, or

2) The partial boundary $\pbd K$ is void (in particular does not contain the empty cell), and $\rH^i(K)$ is 0 for all $i \not= d$ and is either 0 or $\kk$ when $i = d$, or

3) $K$ is void.
\label{lemma:cocell_acyc}
\end{lemma}
\begin{proof}
See \cref{sec:omitted_proofs}.
\end{proof}

\begin{lemma} \label{lemma:betti_from_cocell}
Suppose $Y = (X, A)$ is a weakly colabeled pair of dimension $d$.
If $Y$ supports a cocellular resolution of the monomial ideal $I$, then the Betti numbers of $I$ can be calculated for all $i$ as
$$\betti_{i, \es b}(I) = \dim_{\kk} H^{d-i}(\overline Y_{\es b}, \pbd Y_{\es b}; \kk).$$
\end{lemma}
\begin{proof}
See \cref{sec:omitted_proofs}.
\end{proof}

Like with boundaries, we abbreviate the quotient $\overline K/\pbd K$ to $K/\pbd$, so in particular, the equation above can be written as
$$\betti_{i, \es b}(I) = \dim_{\kk} \rH^{d-i}(Y_{\es b}/\pbd; \kk).$$

\subsection{The Canonical Ideal of a Function Class}
\label{sec:canonical_ideal}
\begin{defn}
An $n$-dimensional {\bf orthoplex} (or $n$-orthoplex for short) is defined as any polytope combinatorially equivalent to $\{x \in \R^n: \|x\|_1 \le 1\}$, the unit disk under the 1-norm in $\R^n$.
Its boundary is a simplicial complex and has $2^n$ facets.
A {\bf fleshy $(n-1)$-dimensional suboplex}, or {\bf suboflex} is the simplicial complex formed by any subset of these $2^n$ facets.
The {\bf complete $(n-1)$-dimensional suboplex} is defined as the suboplex containing all $2^n$ facets.
In general, a {\bf suboplex} is any subcomplex of the boundary of an orthoplex.
\end{defn}
For example, a 2-dimensional orthoplex is equivalent to a square; a 3-dimensional orthoplex is equivalent to an octahedron.

Let $\clsf{C} \sbe [n \to 2]$ be a class of finite functions.
There is a natural fleshy $(n-1)$-dimensional suboplex $\Sbpx C$ associated to $\clsf C$.
To each $f \in \clsf C$ we associate an $(n-1)$-dimensional simplex $F_f \cong \splx^{n-1}$, which will be a facet of $\Sbpx C$.
Each of the $n$ vertices of $F_f$ is labeled by a pair $(i, f(i))$ for some $i \in [n]$, and each face $G$ of $F_f$ is labeled by a partial function $\pf f \sbe f$, whose graph is specified by the labels of the vertices of $G$.
For each pair $f, g \in \clsf C$, $F_f$ is glued together with $F_g$ along the subsimplex $G$ (in both facets) with partial function label $f \cap g$.
This produces $\Sbpx C$, which we call the {\bf canonical suboplex} of $\clsf C$.

\begin{exmp} \label{exmp:complete_suboplex}
Let $[n \to 2]$ be the set of boolean functions with $n$ inputs.
Then $\Ss_{\fun n}$ is the complete $(n-1)$-dimensional suboplex.
Each cell of $\Sbpx{\fun n}$ is association with a unique partial function $\pf f: \sbe [n] \to [2]$, so we write $F_{\pf f}$ for such a cell.
\end{exmp}

\begin{exmp}
Let $f \in [n \to 2]$ be a single boolean function with domain $[n]$.
Then $\Ss_{\{f\}}$ is a single $(n-1)$-dimensional simplex.
\end{exmp}

\begin{exmp}
Let $\linfun_d^2 \sbe [2^d \to 2]$ be the class $(\Fld_2)^{d*}$ of linear functionals mod 2.
\cref{figure:parity_suboplexes} shows $\Ss_{\linfun_d^2}$ for $d=1$ and $d=2$.
\end{exmp}
\begin{figure*}
	\begin{subfigure}[t!]{.5\textwidth}
	\centering
	\includegraphics[width=0.52\textwidth]{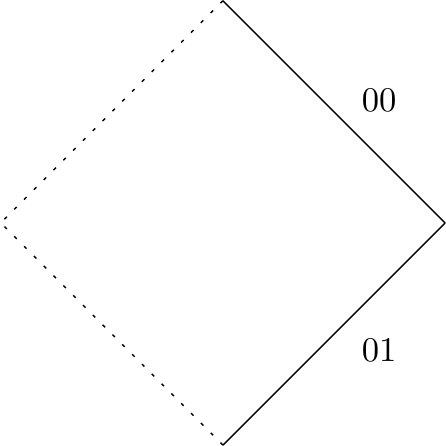}
	\caption{$\linfun_1$ suboplex.
	Dashed lines indicate facets of the complete suboplex not in $\Sbpx{\linfun_1}$.
	Label $00$ is the identically zero function; label $01$ is the identity function.}
	\end{subfigure}
	~
	\begin{subfigure}[t!]{.5\textwidth}
	\centering
	\includegraphics[width=.5\textwidth]{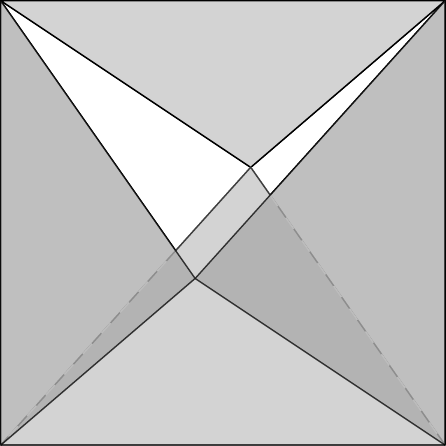}
	\caption{$\Sbpx{\linfun_2}$ is a cone of what is shown, which is a subcomplex of the boundary complex of an octahedron.
	The cone's vertex has label $((0,0), 0)$, so that every top dimensional simplex meets it, because every linear functional sends $(0, 0) \in (\Fld_2)^2$ to 0.}
	\end{subfigure}
\caption{$\linfun_1$ and $\linfun_2$ suboplexes.}
\label{figure:parity_suboplexes}
\end{figure*}

The above gluing construction actually make sense for any $\clsf C \sbe [n \to m]$ (with general codomain $[m]$), even though the resulting simplicial complex will no longer be a subcomplex of $\Sbpx{[n \to 2]}$.
However, we will still call this complex the {\bf canonical suboplex} of $\clsf C$ and denote it $\Sbpx C$ as well.
We name any such complex an {\bf $m$-suboplex}.
The $(n-1)$-dimensional $m$-suboplex $\Sbpx{[n \to m]}$ is called the {\bf complete $(n-1)$-dimensional $m$-suboplex}.

The canonical suboplex of $\clsf C \sbe [n \to m]$ can be viewed as the object generated by looking at the metric space $\clsf C_p$ on $\clsf C$ induced by a probability distribution $p$ on $[n]$, and varying $p$ over all distributions in $\splx^{n-1}$.
This construction seems to be related to certain topics in computer science like derandomization and involves some category theoretic techniques.
It is however not essential to the homological perspective expounded upon in this work, and thus its details are relegated to the appendix (See \cref{sec:cosheaf}).

\begin{defn}
Let $\clsf C \sbe [n \to m]$.
Write $S$ for the polynomial ring $\kk[\xx]$ with variables $x_{i, j}$ for $i \in [n], j \in [m]$.
We call $S$ the {\bf canonical base ring} of $\clsf C$.
The {\bf Stanley-Reisner ideal} $I_{\clsf C}$ of $\clsf C$ is defined as the Stanley-Reisner ideal of $\Sbpx{C}$ with respect to $S$, such that $x_{i, j}$ is associated to the ``vertex'' $(i, j)$ of $\Sbpx C$ (which might not actually be a vertex of $\Sbpx C$ if no function $f$ in $\clsf C$ computes $f(i) = j$).

The {\bf canonical ideal} $\cani{\clsf C}$ of $\clsf C$ is defined as the Alexander dual of its Stanley-Reisner ideal.
\end{defn}
By \cref{prop:alexdual_mingen}, the minimal generators of $\cani{\clsf C}$ are monomials $\xx^\sigma$ where $\sigma^\setc$ is the graph of a function in $\clsf C$.
Let us define $\Gamma \pf f$ to be the complement of $\graph \pf f$ in $[n] \times [m]$ for any partial function $\pf f : \sbe [n] \to [m]$.
Therefore, $\cani{\clsf C}$ is minimally generated by the monomials $\{\xx^{\Gamma f}: f \in \clsf C\}$.
When the codomain $[m] = [2]$, $\Gamma f = \graph (\neg f)$, the graph of the negation of $f$, so we can also write
$$\cani{\clsf C} = \la \xx^{\graph \neg f}: f \in \clsf C\ra.$$ 

\begin{exmp}
Let $\fun n$ be the set of boolean functions with domain $[n]$.
Then $I_{\fun n}$ is the ideal $\la x_{i, 0} x_{i, 1}: i \in [n]\ra$, and $\cani{\fun n}$ is the ideal $\la \xx^{\Gamma f}: f \in \fun n\ra = \la \xx^{\graph g}: g \in \fun n\ra$.
\label{exmp:fun_n_ideals}
\end{exmp}

\begin{exmp}
Let $f \in \fun n$.
The singleton class $\{f\}$ has Stanley-Reisner ideal $\la x_{i, \neg f(i)}: i \in [n] \ra$ and canonical ideal $\la \xx^{\Gamma f} \ra$.
\end{exmp}

The Stanley-Reisner ideal $I_{\clsf C}$ of a class $\clsf C$ has a very concrete combinatorial interpretation.

\begin{prop}
\label{prop:mingen_SR}
Let $\clsf C \sbe [n \to m]$.
$I_{\clsf C}$ is generated by all monomials of the following forms:
\begin{enumerate}
  \item $x_{u, i} x_{u, j}$ for some $u \in [n], i \not= j \in [m]$, or
  \item $\xx^{\graph \pf f}$ for some partial function $\pf f:\sbe [n] \to [m]$ such that $\pf f$ has no extension in $\clsf C$, but every proper restriction of $\pf f$ does.
\end{enumerate}
\end{prop}
It can be helpful to think of case 1 as encoding the fact that $\clsf C$ is a class of functions, and so for every function $f$, $f$ sends $u$ to at most one of $i$ and $j$.
For this reason, let us refer to monomials of the form $x_{u, i}x_{u, j}, i \not=j$ as \textbf{functional monomials} with respect to $S$ and write $\FM_S$, or $\FM$ when $S$ is clear from context, for the set of all functional monomials.
Let us also refer to a PF $\pf f$ of the form appearing in case 2 as an {\bf extenture} of $\clsf C$, and denote by $\ex \clsf C$ the set of extentures of $\clsf C$.
In this terminology, \cref{prop:mingen_SR} says that $I_{\clsf C}$ is minimally generated by all the functional monomials and $\xx^{\graph \pf f}$ for all extentures $\pf f \in \ex \clsf C$.

\begin{proof}
\renewcommand{\IC}{I_{\clsf C}}
The minimal generators of $\IC$ are monomials $\xx^{\es a} \in \IC$ such that $\xx^{\es a}/x_{u, i} \not \in \IC$ for any $(u, i) \in \es a$.
By the definition of $\IC$, $\es a$ is a nonface, but each subset of $\es a$ is a face of the canonical suboplex $\Sbpx C$ of $\clsf C$.
Certainly pairs of the form $\{(u, i), (u, j)\}$ for $u \in [n], i \not= j \in [m]$ are not faces of $\Sbpx C$, but each strict subset of it is a face unless $(u, i) \not \in \Sbpx C$ or $(u, j) \not \in \Sbpx C$.
In either case $\xx^{(u, i)}$ or $\xx^{(u, j)}$ or fall into case 2.
If a minimal generator $\bomega$ is not a pair of such form, then its exponent $\es b$ cannot contain such $\{(u, i), (u, j)\}$ either, or else $r$ is divisible by $x_{u, i}x_{u, j}$.
Therefore $\es b$ is the graph of a partial function $\pf f:\sbe [n \to m]$.
In particular, there is no $f \in \clsf C$ extending $\pf f$, or else $\graph \pf f$ is a face of $\Sbpx C$. But every proper restriction of $\pf f$ must have an extension in $\clsf C$.
Thus $\bomega$ is of the form stated in the proposition.
One can also quickly see that $\xx^{\graph \pf f}$ for any such $\pf f$ is a minimal generator of $\IC$.
\end{proof}
Taking the minimal elements of the above set, we get the following
\begin{prop}
The minimal generators of $\clsf C \sbe [n \to m]$ are $$\{\xx^{\Gamma \pf f}: \pf f \in \ex \clsf C\} \cup \{x_{u, i}x_{u, j}\in\FM: (u \mapsto i) \not \in \ex \clsf C, (u \mapsto j) \not \in \ex \clsf C\}.$$
\end{prop}
Are all ideals with minimal generators of the above form a Stanley-Reisner ideal of a function class?
It turns out the answer is no.
If we make suitable definitions, the above proof remains valid if we replace $\clsf C$ with a class of partial functions (see \cref{prop:mingen_SR_PF}).
But there is the following characterization of the Stanley-Reisner ideal of a (total) function class.

\begin{prop}
Let $I \sbe S$ be an ideal minimally generated by $\{\xx^{\graph \pf f}: \pf f \in \mathcal F\} \cup \{x_{u, i}x_{u, j}\in\FM: (u \mapsto i) \not \in \mathcal F, (u \mapsto j) \not \in \mathcal F\}$ for a set of partial functions $\mathcal F$.
Then $I$ is the Stanley-Reisner ideal of a class of total functions $\clsf C$ precisely when 

\begin{equation}
\label{stmt:SR_char}
\text{\parbox{.7\textwidth}{
For any subset $F \sbe \mathcal F$, if $F(u)$ defined as $\{\pf f(u): \pf f \in F, u \in \dom \pf f\}$ is equal to $[m]$ for some $u \in [n]$, then either $|F(v)| > 1$ for some $v \not = u$ in $[n]$, or $$\vee^u \mathcal F := \bigcup_{\pf f \in \mathcal F, u \in \dom \pf f} \pf f\restrict (\dom \pf f \setminus \{u\})$$ is a partial function extending some $\pf h \in \mathcal F$.}} \tag{$\star$}
\end{equation}
\label{prop:SR_ideal_char}
\end{prop}

\begin{lemma}
\renewcommand{\IC}{I_{\clsf C}}
For $I$ minimally generated as above, $I = \IC$ for some $\clsf C$ iff for any partial $\pf f: \sbe [n] \to [m]$, $\xx^{\graph \pf f} \not \in I$ implies $\xx^{\graph f} \not \in I$ for some total $f$ extending $\pf f$.
\label{lemma:SR_PF_cond}
\end{lemma}
\begin{proof}[Proof of \cref{lemma:SR_PF_cond}]
\renewcommand{\IC}{I_{\clsf C}}
Let $\Delta_I$ be the Stanley-Reisner complex of $I$.
Then each face of $\Delta_I$ is the graph of a partial function, as $I$ has all functional monomials as generators.
A set of vertices $\sigma$ is a face iff $\xx^{\sigma} \not \in I$.
$I = \IC$ for some $I$ iff $\Delta_I$ is a generalized suboflex, iff the maximal cells of $\Delta_I$ are all $(n-1)$-dimensional simplices, iff every cell is contained in such a maximal cell, iff $\xx^{\graph \pf f} \not \in I$ implies $\xx^{\graph f} \not \in I$ for some total $f$ extending $\pf f$.
\end{proof}
\begin{proof}[Proof of \cref{prop:SR_ideal_char}]
\renewcommand{\IC}{I_{\clsf C}}

$(\Rightarrow)$.
We show the contrapositive.
Suppose for some $F \sbe \mathcal F$ and $u \in [n]$, $F(u) = [m]$ but $|F(v)| \le 1$ for all $v \not= u$ and $\pf g:= \vee^u \mathcal F$ does not extend any $\pf f \in \mathcal F$.
Then $\xx^{\graph \pf g} \not \in I$, and every total $f \spe \pf g$ must contain one of $\pf f \in \mathcal F$, and so $\xx^{\graph f} \in I$.
Therefore $I \not= \IC$ for any $\clsf C$.

$(\Leftarrow)$.
Suppose $\eqref{stmt:SR_char}$ is true.
We show that for any nontotal function $\pf f:\sbe [n] \to [m]$ such that $\xx^{\graph \pf f} \not \in I$, there is a PF $\pf h$ that extends $\pf f$ by one point, such that $\xx^{\graph \pf h} \not \in I$.
By simple induction, this would show that $I = \IC$ for some $\clsf C$.

Choose $u \not \in \dom \pf f$.
Construct $F := \{\pf g \in \mathcal F: u \in \dom \pf g, \pf f \spe \pf g \restrict(\dom \pf g \setminus \{u\}) \}$.

If $F(u) \not = [m]$, then we can pick some $i \not \in F(u)$, and set $\pf h(u) = i$ and $\pf h(v) = \pf f(v), \forall v \not = u$.
If $\pf h \spe \pf k$ for some $\pf k \in \mathcal F$, then $\pf k \in F$, but then $\pf k(u) \not= \pf h(u)$ by assumption.
Therefore $\pf h$ does not extend any PF in $\mathcal F$, and $\xx^{\graph \pf h} \not \in I$.

If $F(u) = [m]$, then by $\eqref{stmt:SR_char}$, either $|F(v)| > 1$ for some $v \not= u$ or $\vee^u \mathcal F$ extends some $\pf h \in \mathcal F$.
The former case is impossible, as $\pf f \spe \pf g \restrict(\dom \pf g \setminus \{u\})$ for all $\pf g \in F$.
The latter case is also impossible, as it implies that $\xx^{\graph \pf f} \in I$.

\end{proof}
\subsection{Resolutions}
\label{sec:resolutions}
Sometimes we can find the minimal resolution of the Stanley-Reisner ideal of a class.
For example, consider the complete class $[n \to 2]$.
Its Stanley-Reisner ideal is $\la x_{i, 0} x_{i, 1}: i \in [n]\ra$ as explained in \cref{exmp:fun_n_ideals}.

\begin{thm}
Let $X$ be an $(n-1)$-simplex, whose vertex $i$ is labeled by monomial $x_{i, 0}x_{i, 1}$.
Then $X$ is a minimal cellular resolution of $S/I_{[n \to 2]}$.
\end{thm}
\begin{proof}
The vertex labels of $X$ generate $I_{\fun n}$, and each face label is distinct from other face labels, so if $X$ is a cellular resolution, then it resolves $S/I_{\fun n}$ and is minimal.
Therefore it suffices to show that $\resf F_X$ is exact. 
By \cref{lemma:acyclicity}, we need to show that $X_{\preceq \es b}$ is acyclic over $\kk$ for all $\es b \sbe [n] \times [2]$.
$X_{\preceq \es b}$ can be described as the subcomplex generated by the vertices $\{i : (i, 0), (i, 1) \in \es b\}$, and hence is a simplex itself and therefore contractible.
This completes the proof.
\end{proof}
\begin{cor}
The Betti numbers of $I_{\fun n}$ are nonzero only at degrees of the form
$$\sigma = \prod_{i \in U} x_{i, 0} x_{i, 1}$$
for subset $U \sbe [n]$.
In such cases,
$$\betti_{i, \sigma}(I_{\fun n}) = \ind(i = |U| - 1).$$
\end{cor}

Similar reasoning also gives the minimal resolution of any singleton class.
\begin{thm} \label{thm:minres_singleton}
Suppose $f \in \fun n$.
Let $X$ be an $(n-1)$-simplex, whose vertex $i$ is labeled by variable $x_{i, \neg f(i)}$.
Then $X$ is a minimal cellular resolution of $S/I_{\{f\}}$.
\end{thm}
\begin{cor}
The Betti numbers of $I_{\{f\}}$ are nonzero only at degrees of the form
$$\sigma = \prod_{i \in U} x_{i, \neg f(i)}$$
for subset $U \sbe [n]$.
In such cases,
$$\betti_{i, \sigma}(I_{\{f\}}) = \ind(i = |U| - 1).$$

\end{cor}

However, in general, minimally resolving the Stanley-Reisner ideal of a class seems difficult.
Instead, we turn to the canonical ideal, which appears to more readily yield cellular resolutions, and as we will see, whose projective dimension corresponds to the VC dimension of the class under an algebraic condition.
For example, a single point with label $\xx^{\Gamma f}$ minimally resolves $S/\cani{\{f\}}$ for any $f \in \fun n$.

We say $(X, \lambda)$ is a {\bf cellular resolution of a class $\clsf C$} if $(X, \lambda)$ is a cellular resolution of $S/\cani{\clsf C}$.
In the following, we construct the cellular resolutions of many classes that are studied in Computational Learning Theory.
As a warmup, we continue our discussion of $[n \to 2]$ by constructing a cellular resolution of its canonical ideal.
\begin{thm}
Let $P$ be the $n$-dimensional cube $[0, 1]^n$, where vertex $v \in [2]^n$ is labeled with the monomial $\prod_{i=1}^n x_{i, v_i}$.
Then $P$ minimally resolves $\fun n$.
\label{thm:fun_n_can_res}
\end{thm}
\begin{proof}
We first show that this labeled cell complex on the cube is a cellular resolution.
Let $\sigma \sbe [n] \times [2]$.
We need to show that $P_{\preceq \sigma}$ is acyclic.
If for some $i$, $(i, 0) \not \in\sigma \And (i,1 ) \not \in \sigma$, then $P_{\preceq \sigma}$ is empty and thus acyclic.
Otherwise, $\sigma^\setc$ defines a partial function $\pf f:\sbe [n] \to [2]$.
Then $P_{\preceq \sigma}$ is the ``subcube''
$$\{v \in [0, 1]^n: v_i = \neg f(i), \forall i \in \dom f\},$$ 
and is therefore acyclic.
This shows that $P$ is a resolution.
It is easy to see that all faces of $P$ have unique labels in this form, and hence the resolution is minimal as well.

$P$ resolves $S/\cani{\fun n}$ by \cref{exmp:fun_n_ideals}
\end{proof}
The above proof readily yields the following description of $\fun n$'s Betti numbers.
\begin{cor}
The Betti numbers for $\cani{\fun n}$ are nonzero only at degrees of the form $\Gamma \pf f$ for partial functions $\pf f : \sbe [n] \to [2]$.
More precisely,
$$\betti_{i, \Gamma \pf f}(\cani{\fun n}) = \ind(|\dom f | = n - i)$$
\end{cor}

We made a key observation in the proof of \cref{thm:fun_n_can_res}, that when neither $(i, 0)$ nor $(i, 1)$ is in $\sigma$ for some $i$, then $P_{\preceq \sigma}$ is empty and thus acyclic.
A generalization to arbitrary finite codomains is true for all complexes $X$ we are concerned with:
\begin{lemma}
 
Let $(X, \lambda)$ be a labeled complex in which each vertex $i$ is labeled with $\Gamma \pf f_i$ for partial function $\pf f_i: [n] \to [m]$.
Then the face label $\lambda(F)$ for a general face $F$ is $\Gamma\lp \bigcap_{i \in F} \pf f_i \rp$. 
A fortiori $X_{\preceq \sigma}$ is empty whenever $\sigma$ is not of the form $\Gamma \pf g$ for some partial function $\pf g: \sbe [n] \to [m]$.
\label{lemma:pf_label_propagate}
\end{lemma}
\begin{proof}
Treating the exponent labels, which are squarefree, as sets, we have
\begin{align*}
\lambda(F) = \bigcup_{i \in F} \Gamma \pf f_i = \bigcup_{i \in F} (\graph \pf f_i)^\setc = \lp\graph \bigcap_{i \in F} \pf f_i\rp^\setc = \Gamma\left(\bigcap_{i \in F} \pf f_i\right)
\end{align*}
If $\sigma$ is not of the form $\Gamma \pf g$, then for some $a \in [n]$ and $b \not= b' \in [m]$, $(a, b), (a, b') \not \in \sigma$.
But every exponent label is all but at most one of the pairs $(a, *)$.
So $X_{\preceq \sigma}$ is empty.
\end{proof}
If we call a complex as described in the lemma {\bf partial-function-labeled}, or PF-labeled for short, then any PF-labeled complex has a set of {\it partial function labels}, or PF labels for short, along with its monomial/exponent labels.
If $\pf f_F$ denotes the partial function label of face $F$ and $\es a_F$ denotes the exponent label of face $F$, then they can be interconverted via
$$\es a_F = \Gamma\pf f_F\qquad \pf f_F = \es a_F^\setc$$
where on the right we identify a partial function with its graph.
\cref{lemma:pf_label_propagate} therefore says that $F \sbe G$ implies $\pf f_F \spe \pf f_G$, and $\pf f_F = \bigcap_{i \in F} \pf f_i$, for faces $F$ and $G$.
When we wish to be explicit about the PF labeling function, we use the symbol $\mu$, such that $\mu(F) = \pf f_F$, and refer to labeled complexes as pairs $(X, \mu)$ or triples $(X, \lambda, \mu)$.
We can furthermore reword \cref{lemma:acyclicity} for the case of PF-labeled complexes.
Write $X_{\spe \pf f}$ (resp. $X_{\supset \pf f}$) for the subcomplex with partial function labels weakly (resp. strictly) extending $\pf f$.
\begin{lemma}
A PF-labeled complex $X$ is a cellular resolution iff $X_{\spe \pf f}$ is acyclic over $\kk$ for all partial functions $\pf f$.
\label{lemma:pf_acyclicity}
\end{lemma}
A {\it PF-colabeled complex or pair} is defined similarly.
The same interconversion equations hold.
We can likewise reword \cref{lemma:cocell_acyc}.
\begin{lemma}
Let $(X, A)$ be a weakly PF-colabeled complex or pair of dimension $d$.
$(X, A)$ is a cocellular resolution if for any partial function $\pf f$, $(X, A)_{\spe \pf f}$ is either

1) representable as a cellular pair $(Y, B)$ -- that is, $X \setminus A$ as a collection of open cells is isomorphic to $Y \setminus B$ as a collection of open cells, such that $H^i(Y, B)$ is 0 for all $i \not= d$, or

2) a complex $Y$ (in particular it must contain a colabeled empty cell) whose reduced cohomology vanishes at all dimensions except $d$. 
\label{lemma:pf_cocell_acyc}
\end{lemma}

Because any cellular resolution of a class $\clsf C$ only has cells with degree $\Gamma \pf f$ for some PF $\pf f$, the Betti numbers $\betti_{i, \sigma}(\cani{\clsf C})$ can be nonzero only when $\sigma = \Gamma \pf g$ for some PF $\pf g$.
We define {\bf the Betti numbers of a class $\clsf C$} as the Betti numbers of its canonical ideal $\cani{\clsf C}$, and we denote $\betti_{i, \pf f}(\clsf C) := \betti_{i, \Gamma \pf f}(\cani{\clsf C})$.

Finally we note a trivial but useful proposition and its corollary.
\begin{prop}
Let $\clsf C \sbe \fun[m]n$, and let $\pf f : \sbe[n] \to [m]$.
The subset of functions extending $\pf f$, $\{h \in \clsf C: \pf f \sbe h\}$, is the intersection of the collection of sets which extend the point restrictions of $f$, $\bigcap_{i \in \dom \pf f} \{h \in \clsf C: (i, \pf f(i)) \sbe h\}$.

If partial functions $\pf g_1, \ldots, \pf g_k \in \fun [m] n$ satisfy $\bigcup_t \pf g_t = \pf f$, then we also have
$$\{h \in \clsf C: \pf f \sbe h\} = \bigcap_{t = 1}^k \{h \in \clsf C: \pf g_t \sbe h\}.$$
\label{prop:extension_set_as_intersection}
\end{prop}
\begin{cor}
Let $\pf f : \sbe [n] \to [m]$.
Suppose $X$ is a PF-labeled complex.
If partial functions $\pf g_1, \ldots, \pf g_k \in \fun [m] n$ satisfy $\bigcup_t \pf g_t = \pf f$, then 
$$X_{\spe \pf f} = \bigcap_{t=1}^k X_{\spe \pf g_t}.$$
\label{cor:extension_complex_as_intersection}
\end{cor}

With these tools in hand, we are ready to construct cellular resolutions of more interesting function classes.

\subsubsection{Delta Functions}

Let $\cDelta_n \sbe [n \to 2]$ be the class of delta functions $\delta_{i}(j) = \ind(i = j)$.
Form the abstract simplex $X$ with vertices $[n]$.
Label each vertex $i$ with $\delta_i$ and induce PF labels on all higher dimensional faces in the natural way.
One can easily check the following lemma.
\begin{lemma}
For any face $F \sbe [n]$ with $|F| > 1$, its PF label $\pf f_F$ is the function defined on $[n] \setminus F$, sending everything to 0.
Conversely, for every partial $\pf f:\sbe [n] \to [2]$ with $\im \pf f \sbe \{0\}$, there is a unique face $F$ with $\pf f_F = \pf f$ as long as $n - |\dom f| \ge 2$.
\label{lemma:delta_PFlabel}
\end{lemma}

\begin{thm}
$X$ is a $(n-1)$-dimensional complex that minimally resolves $\cDelta_n$.
\end{thm}
\begin{proof}
We apply \cref{lemma:pf_acyclicity}:
We show for any $\pf f:\sbe [n] \to [2]$, $X_{\spe \pf f}$ is acyclic.

If $\pf f$ sends two distinct elements to 1, then $X_{\spe \pf f}$ is empty.
If $\pf f$ sends exactly one element $i$ to 1, then $X_{\spe \pf f}$ is the single point $i$.
If $\pf f$ is the empty function, then $X_{\spe \pf f}$ is the whole simplex and thus acyclic.
Otherwise, $\im \pf f = \{0\}$.
If $n - | \dom f| = 1$, then there is exactly one delta function extending $\pf f$, so $X_{\spe \pf f}$ is again a point.
If $n - |\dom f| \ge 2$, then by \cref{lemma:delta_PFlabel}, $X_{\spe \pf f}$ is exactly one face $F$ with $\pf f_F = \pf f$, and therefore acyclic.

$X$ is furthermore minimal because all PF labels are distinct.
\end{proof}

Tabulating the faces by their labels, we obtain
\begin{cor}
For $i > 0$, $\betti_{i, \pf f}({\cDelta_n})$ is nonzero only when $\im \pf f \sbe \{0\}$ and $n - |\dom \pf f| \ge 2$, and $i = n - |\dom \pf f| - 1$.
In that case, $\betti_{i, \pf f}({\cDelta_n}) = 1$.
In particular, the top dimensional Betti number is $\betti_{n-1, \emptyfun}(\cDelta_n) = 1$.
\end{cor}

\subsubsection{Weight-$k$ Functions}
\newcommand{\cYes}{\textsc{wt}}
Write $o := \mathbf 0 \in [n \to 2]$, the function that sends all inputs to 0. 
Let $\cYes(f, k)_n \sbe [n \to 2]$ be the class consisting of all functions $g$ such that there are exactly $k$ inputs $u \in [n]$ such that $g(u) \not = f(u)$.
This is a generalization of $\cDelta$, as $\cYes(o, 1)_n = \cDelta$.
WLOG, we consider the case $f = o$ in this section.
Consider the hyperplane $H_k := \{v \in \R^n: \sum_i v_i = k\}$ and the polytope given by
$$P^k_n := [0, 1]^n \cap H_k.$$
We inductively define its labeling function $\mu_n^k$ and show that $(P^k_n, \mu_n^k)$ is a minimal cellular resolution of $\cYes(f, k)_n$.

For $n = 1$, $P_1^0$ and $P_1^1$ are both a single point.
Set $\mu_1^0(P_1^0) = (0 \mapsto 0)$ and $\mu_1^1(P_1^1) = (0 \mapsto 1)$.
Then trivially, $(P_1^0, \mu_1^0)$ is the minimal resolution of $\cYes(o, 0) = \{0 \mapsto 0\}$ and $(P_1^1, \mu_1^1)$ is the minimal resolution of $\cYes(o, 1) = \{0 \mapsto 1\}$.

Suppose that $\mu_m^k$ is defined and that $(P_m^k, \mu_m^k)$ is a minimal cellular resolution of $\cYes(o, k)$ for all $0 \le k \le m$.
Consider $n = m+1$ and fix $k$.
Write, for each $u \in [n], b \in [2]$, $F_{u, b} := [0, 1]^u \times \{b\} \times [0, 1]^{n - u -1}$ for the corresponding facet of $[0, 1]^n$.
Then $P_n^k$ has boundary given by
$$\bigcup_{u \in [n], b \in [2]} F_{u, b} \cap H_k.$$
But we have $F_{u, 0} \cap H_k \cong P_{n-1}^k$ and $F_{u, 1} \cap H_k \cong P_{n-1}^{k-1}$ (here $\cong$ means affinely isomorphic).
Thus, if $G$ is a face of $F_{u, b} \cap H_k$, we define the labeling functions 
\begin{align*}
\mu_n^k(G) : [n] &\to [2]\\
	i	& \mapsto \mu_{n-1}^{k-b}(G)(i) & &\text{if $i < u$}\\
	i	& \mapsto b					& &\text{if $i =u$}\\
	i	& \mapsto \mu_{n-1}^{k-b}(G)(i-1) & &\text{if $i > u$.}
\end{align*}
If we represent functions as a string of $\{0, 1, .\}$ (where $.$ signifies ``undefined''), then essentially $\mu_n^k(G)$ is obtained by inserting $b \in \{0, 1\}$ at the $u$th position in $\mu_{n-1}^{k-b}(G)$. 
It is easy to see that, when $G$ is both a face of $F_{u, b} \cap H_k$ and a face of $F_{u', b'} \cap H_k$, the above definitions of $\mu_n^k(G)$ coincide.
Finally, we set $\mu_n^k(P_n^k) = \emptyfun$.
This finishes the definition of $\mu_n^k$.

In order to show that $(P_n^k, \mu_n^k)$ is a minimal cellular resolution, we note that by induction hypothesis, it suffices to show that $(P_n^k)_{\spe \emptyfun} = P_n^k$ is acyclic, since $(P_n^k)_{\spe (u \mapsto b) \cup \pf f} \cong (P_{n-1}^{k-b})_{\spe \pf f}$ is acyclic. 
But of course this is trivial given that $P_n^k$ is a polytope.
By an easy induction, the vertex labels of $P_n^k$ are exactly the functions  of $\cYes(o, k)_n$.
Thus 
\begin{thm}
$(P_n^k, \mu_n^k)$ as defined above is a minimal resolution of $\cYes(o, k)_n$.
\end{thm}

\begin{cor} 
For $k \not = 0, n$, $\clsf C :=\cYes(o, k)_n$ has a Betti number $\betti_{i, \emptyfun}(\clsf C) = \ind(i = n-1)$.
Furthermore, for each PF $\pf f$, $\betti_{i, \pf f}(\cYes(o, k)_n)$ is nonzero for at most one $i$, where it is 1.
\end{cor}

\subsubsection{Monotone Conjunction}
Let $L = \{l_1, \ldots, l_d\}$ be a set of literals.
The class of {\bf monotone conjunctions} $\monconj_d$ over $L$ is defined as the set of functions that can be represented as a conjunction of a subset of $L$.
We represent each $h \in \monconj_d$ as the set of literals $L(h)$ in its conjunctive form, and for each subset (or indicator function thereof) $T$ of literals, let $\Lambda(T)$ denote the corresponding function.
For example, $\Lambda\{l_1, l_3\}$ is the function that takes $v \in [2]^d$ to 1 iff $v_1 = v_3 = 1$.
 
\begin{thm}
Let $X$ be the $d$-cube in which each vertex $V \in [2]^d$ has partial function label (that is in fact a total function) $\pf f_V = \Lambda(V)$, where on the RHS $V$ is considered an indicator function for a subset of literals.
Then $X$ resolves $\monconj_d$ minimally.
\label{thm:monconj_res}
\end{thm}

We first show that the induced face labels of $X$ are unique, and hence if $X$ is a resolution, it is minimal.
This will follow from the following three lemmas.

\begin{lemma}
Let ${\pf w}$ be a partial function ${\pf w}: \sbe [d] \to [2]$.
Let $\Sigma_{\pf w}$ be the set of monotone conjunctions $\{h: l_i \in L(h) \text{ if ${\pf w}(i) = 1$ and } l_i \not\in L(h) \text{ if ${\pf w}(i) = 0$} \}$.
Then the intersection of functions (not literals) $\bigcap \Sigma_{\pf w}$ is the partial function $\Lambda({\pf w}) := \pf f: \sbe [2^d] \to [2]$,
$$\pf f(v) = \begin{cases}
0 & \text{if $v_i = 0$ for some $i$ with ${\pf w}(i) = 1$}\\
1 & \text{if $v_i = 1$ for all $i$ with ${\pf w}(i) = 1$ and for all $i$ where ${\pf w}(i)$ is undefined}\\
\text{undefined} & \text{otherwise.}
\end{cases}
$$

When ${\pf w}$ is a total function considered as a bit vector, $\Lambda({\pf w})$ coincides with the previous definition of $\Lambda$.

If $F$ is the face of the cube resolution $X$ with the vertices $\{V: V \spe {\pf w}\}$ (here treating $V \in [2]^d \cong [d \to 2]$ as a function), then the partial function label of $F$ is $\pf f$.
\label{lemma:cap_monconj}
\end{lemma}
\begin{proof}
$\pf f$ is certainly contained in $\bigcap \Sigma_{\pf w}$.
To see that the inclusion is an equality, we show that for any $v$ not of the two cases above, there are two functions $h, h'$ that disagree on $v$.
Such a $v$ satisfies $v_i = 1$ for all ${\pf w}(i) = 1$ but $v_i = 0$ for some ${\pf w}(i)$ being undefined.
There is some $h \in \Sigma_{\pf w}$ with $L(h)$ containing the literal $l_i$ and there is another $h' \in \Sigma_{\pf w}$  with $l_i \not\in L(h')$.
These two functions disagree on $v$.

The second statement can be checked readily.
The third statement follows from \cref{lemma:pf_label_propagate}. 
\end{proof}

\begin{lemma}
For any partial function $\pf f$ of the form in \cref{lemma:cap_monconj}, there is a unique partial function ${\pf w}: \sbe d \to 2$ with $\pf f = \Lambda({\pf w})$, and hence there is a unique cell of $X$ with PF label $\pf f$.
\label{lemma:cube_label_bijection}
\end{lemma}
\begin{proof}
The set $A := \inv {\pf w} 1 \cup (\dom {\pf w})^{\setc}$ is the set $\{i \in d: v_i = 1, \forall v \in \inv {\pf f} 1\}$, by the second case in $\pf f$'s definition.
The set $B := \inv {\pf w} 1$ is the set of $i \in d$ such that the bit vector $v$ with $v_i = 0$ and $v_j = 1$ for all $j \not= i$ is in $\inv {\pf f} 0$, by the first case in $\pf f$'s definition.
Then $\dom {\pf w} = (A \setminus B)^\setc$, and $\inv {\pf w} 0 = (\dom {\pf w}) \setminus (\inv {\pf w} 1)$.
\end{proof}

\begin{lemma}
The face labels of $X$ are all unique.
\label{lemma:cube_label_unique}
\end{lemma}
\begin{proof}
Follows from \cref{lemma:cap_monconj} and \cref{lemma:cube_label_bijection}.
\end{proof}

\begin{proof}[Proof of \cref{thm:monconj_res}]
We show that $X$ is a resolution (minimal by the above) by applying \cref{lemma:pf_acyclicity}.
Let $\pf f: \sbe [2^d] \to [2]$ be a partial function and $\pf g_0, \pf g_1$ be respectively defined by $\pf g_t = \pf f \restrict \inv{\pf f}{t}$ for $t = 0, 1$, so that $\pf f = \pf g_0 \cup \pf g_1$.
By \cref{cor:extension_complex_as_intersection}, $X_{\spe \pf f} = X_{\spe \pf g_0} \cap X_{\spe \pf g_1}$.
We first show that $X_{\spe \pf g_1}$ is a face of $X$, and thus is itself a cube.
If $h \in \monconj_d$ is a conjunction, then it can be seen that $h$ extends $\pf g_1$ iff $L(h) \sbe L_1 := \bigcap_{v \in \dom \pf g_1} \{l_i: v_i = 1\}$ (check this!).
Thus $X_{\spe \pf g_1}$ is the subcomplex generated by the vertices $V$ whose coordinates $V_i$ satisfy $V_i = 0, \forall i \not \in L_1$.
This subcomplex is precisely a face of $X$.

Now we claim that each cell of $X_{\spe \pf g_0} \cap X_{\spe \pf g_1}$ is a face of a larger cell which contains the vertex $W$ with $W_1 = 1, \forall i \in L_1$ and $W_i = 0, \forall i \not \in L_1$.
This would imply that $X_{\spe \pf g_0} \cap X_{\spe \pf g_1}$ is contractible via the straight line homotopy to $W$.

We note that if $h, h' \in \monconj_d$ and $L(h) \sbe L(h')$, then $h$ extends $\pf g_0$ only if $h'$ also extends $\pf g_0$.
(Indeed, $h$ extends $\pf g_0$ iff $\forall v \in \dom \pf g_0$, $v_k = 0$ while $l_k \in L(h)$ for some $k$. This still holds for $h'$ if $h'$ contains all literals appearing in $h$).
This means that, if $F$ is a face of $X_{\spe \pf g_0}$, then the face $F'$ generated by $\{V': \exists V \in F, V \sbe V'\}$ (where $V$ and $V'$ are identified with the subset of literals they correspond to) is also contained in $X_{\spe \pf g_0}$; $F'$ can alternatively be described geometrically as the intersection $[0, 1]^d \cap (F + [0, 1]^d)$.
If furthermore $F$ is a face of $X_{\spe \pf g_1}$, then $F' \cap X_{\spe \pf g_1}$ contains $W$ as a vertex, because $W$ is inclusion-maximal among vertices in $X_{\spe \pf g_1}$ (when identified with sets for which they are indicator functions for).
This proves our claim, and demonstrates that $X_{\spe \pf f}$ is contractible.
Therefore, $X$ is a (minimal) resolution, of $\cani{\monconj_d}$ by construction.
\end{proof}

\begin{cor}
$\betti_{i, \pf f}(\monconj_d)$ is nonzero iff $\pf f = \Lambda({\pf w})$ for some PF ${\pf w}: \sbe [d] \to [2]$ and $i = d - |\dom {\pf w}|$, and in that case it is 1.
In particular, the top dimensional nonzero Betti number is $\betti_{d, \mathbf 1 \mapsto 1}(\monconj_d) = 1$.
\end{cor}
\begin{proof}
This follows from \cref{lemma:cap_monconj} and \cref{lemma:cube_label_bijection}.
\end{proof}

We will refer to $X$ as the {\bf cube resolution} of $\monconj_d$.

\subsubsection{Conjunction}
Define $L' := \bigcup_{i=1}^d \{l_i, \neg l_i\}$.
The class of conjunctions $\CLconj_d$ is defined as the set of functions that can be represented as a conjunction of a subset of $L'$.
In particular, $L'$ contains the null function $\bot: v \mapsto 0, \forall v$, which can be written as the conjunction $l_1 \wedge \neg l_1$.

We now describe the polyhedral cellular resolution of $\CLconj_d$, which we call the \textbf{cone-over-cubes} resolution, denoted $\cocres_d$.
Each nonnull function $h$ has a unique representation as a conjunction of literals in $L'$.
We define $L(h)$ to be the set of such literals and $\widetilde{\Lambda}$ be the inverse function taking a set of consistent literals to the conjunction function.
We assign a vertex $V_{h} \in \{-1, 0, 1\}^d\times\{0\} \in \R^{d+1}$ to each nonnull $h$ by
$$(V_{h})_i = \begin{cases}
1 & \text{if $l_i \in L(h)$}\\
-1 & \text{if $\neg l_i \in L(h)$}\\
0 & \text{otherwise}
\end{cases}$$
for all $1 \le i \le d$ (and of course $(V_{h})_{d+1} = 0$), so that the PF label $\pf f_{V_{h}} = h$.
We put in $\cocres_d$ all faces of the $\overbrace{(2, 2, \ldots, 2)}^d$ pile-of-cubes: these are the collection of $2^d$ $d$-dimensional unit cubes with vertices among $\{-1, 0, 1\}^d \times \{0\}$.
This describes all faces over nonnull functions.

Finally, we assign the coordinate $V_{\bot} = (0, \ldots, 0, 1) \in \R^{d+1}$, and put in $\cocres_d$ the $(d+1)$-dimensional polytope $C$ which has vertices $V_{h}$ for all $h \in \CLconj_d$, and which is a cone over the pile of cubes, with vertex $V_{\bot}$.
(Note that this is an improper polyhedron since the $2^d$ facets of $C$ residing on the base, the pile of cubes, all sit on the same hyperplane.)

\cref{figure:COC2} shows the cone-over-cubes resolution for $d=2$.

\begin{figure}
\centering
\includegraphics[width=.8\textwidth]{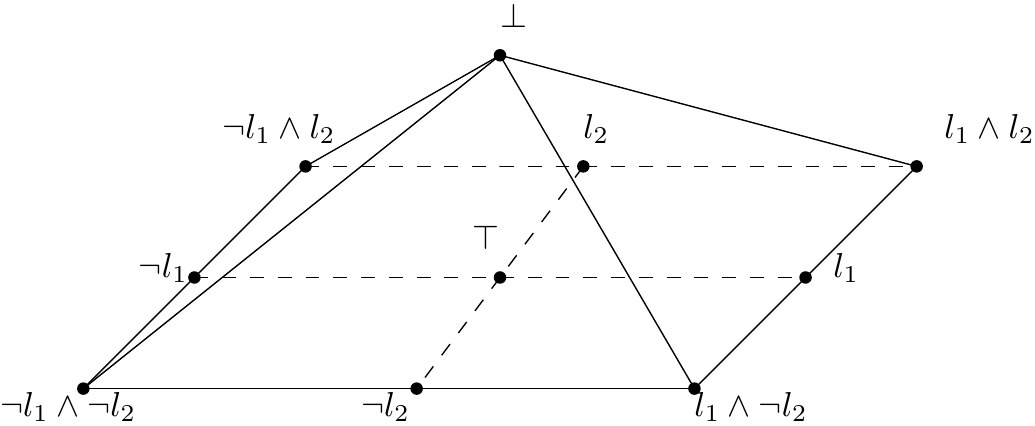}
\caption{Cone-over-cube resolution of $\CLconj_2$.
Labels are PF labels.}
\label{figure:COC2}
\end{figure}

\begin{thm}
$\cocres_d$ is a $(d+1)$-dimensional complex that minimally resolves $\CLconj_d$.
\end{thm}
\begin{proof}
Let $X = \cocres_d$.
We first shot that $X$ is a resolution of $\CLconj_d$.
We wish to prove that for any $\pf f:\sbe [2^d] \to [2]$, the subcomplex of $X_{\spe \pf f}$ is acyclic.

First suppose that $\im \pf f = \{0\}$.
Then $X_{\spe \pf f}$ is a subcomplex that is a cone with $V_{\bot}$ as the vertex, and hence contractible.

Otherwise $\pf f$ sends some point $u \in [2^d] \cong [2]^d$ to $1$.
All $h \in \CLconj_d$ extending $\pf f$ must have $L(h)$ be a subset of $\{l_i^{u_i}: i \in [d]\}$, where $l_i^{u_i}$ is the literal $l_i$ if $u_i = 1$ and $\neg l_i$ if $l_i = 0$.
The subcomplex of $X$ consisting of these $h$ is a single $d$-cube of the pile, given by the opposite pair of vertices $\mathbf 0$ and $2 u - \mathbf 1$ in $\R^d$ considered as the hyperplane containing the pile.
But then this case reduces to the reasoning involved in the proof that the cube resolution resolves $\monconj_d$.
Hence we conclude that $X_{\spe \pf f}$ is acyclic for all $\pf f$, and therefore $X$ resolves $\CLconj_d$.

We prove the uniqueness of PF labels and therefore the minimality of $X$ through the following series of propositions.
\end{proof}

Each face of $\cocres_d$ containing the vertex $V_{\bot}$ is a cone over some subpile-of-subcubes, which has vertices $\mathcal P_{\pf w} = \{V \in \{-1, 0, 1\}^d \times \{0\}: V_i = {\pf w}(i), \forall i \in \dom {\pf w}\}$ for some PF ${\pf w}:\sbe [d] \to \{-1, 1\}$.
We shall write $C_{\pf w}$ for a face associated with such a ${\pf w}$.
Obviously $\dim C_{\pf w} = d+1 - |\dom {\pf w}|$.

\begin{prop}
Let $W \in \{-1, 0, 1\}^d \times \{0\}$ be defined by $W_i = {\pf w}(i), \forall i \in \dom {\pf w}$, and $W_i = 0$ otherwise.
Thus $W$ is the ``center'' of the subpiles-of-subcubes mentioned above.
Its PF label is a total function $f = \pf f_W \in \CLconj_d$.

Then the face $C_{\pf w}$ has a unique PF label $\Lambda'({\pf w}) := f \restrict \inv f 0 = f \cap \bot$ as a partial function $:\sbe [2^d] \to [2]$.
\end{prop}
\begin{proof}
By \cref{lemma:pf_label_propagate}, the PF label of $C_{\pf w}$ is the intersection of the PF labels of its vertices.
Since $\Lambda'({\pf w}) = f \cap \bot$, $\Lambda'({\pf w}) \spe \pf f_{C_{\pf w}}$.

Because $L(f) \sbe L(\pf f_V)$ for all $V \in \mathcal P_{\pf w}$, $f(u) = 0$ implies $\pf f_V(u) = 0, \forall V \in \mathcal P_{\pf w}$.
Thus $\Lambda'({\pf w}) \sbe \pf f_V, \forall V \in C_{\pf w} \implies \Lambda'({\pf w}) = \pf f_{C_{\pf w}}$ as desired.

Uniqueness follows from the uniqueness of preimage of 0.
\end{proof}

The rest of the faces in $\cocres_d$ reside in the base, and for each face $F$, $\bigcup\{L(\pf f_V): V \in F\}$ contains at most one of each pair $\{\neg l_i, l_i\}$.
Define the partial order $\lhd$ on $\{-1, 0, 1\}^d$ as the product order of partial order $0 \lhd' +1, -1$.
It is easy to check that $V \unlhd W$ implies $L(\pf f_V) \sbe L(\pf f_W)$, which further implies $\inv{\pf f_V}(0) \sbe \inv{\pf f_W}(0)$ and $\inv{\pf f_V}(1) \spe \inv{\pf f_W}(1)$.
Each face $F$ can be described uniquely by the least and the greatest vertices in $F$ under this order, which we denote resp. as $\min F$ and $\max F$.
Then the vertices in $F$ are precisely those who fall in the interval $[\min F, \max F]$ under partial order $\lhd$.
\begin{prop}
Let $F$ be a face residing in the base of $\cocres_d$ and write $V := \min F$ and $W := \max F$.
Then $F$ has a unique PF label $\pf f_F = \Lambda(V, W): \sbe [2^d] \to \{-1, 0, 1\}$,
$$u \mapsto \begin{cases}
0 & \text{if $u_i = (1-V_i)/2$ for some $i$ with $V_i \not= 0$}\\
1 & \text{if $u_i = (1+W_i)/2$ for all $i$ with $W_i \not = 0$}\\
\text{undefined} & \text{otherwise.}
\end{cases}$$
\end{prop}
\begin{proof}
By the observation above, we see that $\pf f_F = \bigcap_{U \in F} \pf f_U$ has $\pf f_F^{-1}(0) = \pf f_V^{-1}(0)$ and $\pf f_F^{-1}(1) = \pf f_W^{-1}(1)$.
Both sets are exactly of the form described above.

It remains to check that the map $F \mapsto \pf f_F$ is injective.
Let $\pf f = \pf f_F$ for some face $F$.
We have $(\max F)_i = 1$ iff $\forall u \in \pf f^{-1}(1), u_i = 1$ and $(\max F)_0 = 0$ iff $\forall u \in \pf f^{-1}(1), u_i = -1$.
Thus $\pf f$ determines $\max F$.
Let $v$ be the bit vector defined by $v_j = (1+(\max F)_j)/2$ if $(\max F)_j \not = 0$ and $v_j = 0$ otherwise.
Let $v^i$ denote $v$ with the $i$th bit flipped.
Then $(\min F)_i \not = 0$ iff $\pf f(v^i) = 0$.
For all other $i$, we have $(\min F)_i = (\max F)_i$.
This proves the uniqueness of the label $\pf f_F$.
\end{proof}

\begin{prop}
Every face of $\cocres_d$ has a unique PF label.
\end{prop}
\begin{proof}
The only thing remaining to check after the two propositions above is that faces incident on the vertex $V_{\bot}$ have different PF labels from all other faces.
But it is obvious that functions of the form in the previous proposition have nonempty preimage of 1, so cannot equal $\Lambda'({\pf w})$ for any ${\pf w}$. 
\end{proof}

Summarizing our results, we have the following
\begin{thm}
$\betti_{i, \pf f}(\CLconj_d)$ is nonzero iff $\pf f = \Lambda'({\pf w})$ for some ${\pf w}: \sbe[d] \to \{-1, 1\}$ and $i = d + 1 - |\dom {\pf w}|$ or $\pf f = \Lambda(V, W)$ for some $V, W \in \{-1, 0, 1\}^d, V \unlhd W$.
In either case, the Betti number is 1.

In particular, the top dimensional nonzero Betti number is $\betti_{d+1, \emptyfun}(\CLconj_d) = 1$.
\end{thm}

\subsubsection{Threshold Functions}
\label{sec:threshold_functions}
Let $U \sbe \R^d$ be a finite set of points.
We are interested in the class of linear threshold functions $\linthr_U$ on $U$, defined as the set of functions of the form
$$u \mapsto \begin{cases}
1 & \text{if $c \cdot u > r$}\\
0 & \text{if $c \cdot u \le r$}
\end{cases}$$
for some $c \in \R^d, r \in \R$.
We shall assume $U$ affinely spans $\R^d$; otherwise, we replace $\R^d$ with the affine span of $U$, which does not change the class $\linthr_U$.

When $U = \bcube^d$, this is the class of linear threshold functions on $d$ bits, and we write $\linthr_d$ for $\linthr_U$ in this case.
Define
$$\monom s (u_0, \ldots, u_{d-1}) = (u_0\cdots u_{s-2} u_{s-1}, u_0 \cdots u_{s-2} u_s, \cdots, u_{d-s} \cdots u_{d-2} u_{d-1})$$
as the function that outputs degree $s$ monomials of its input.
For $U = \Monom^k$, the image of $\bcube^d$ under the map 
\begin{align*}
\monom{\le k}: u \mapsto (\monom 1 u, \cdots, \monom k u)
\end{align*}
$\linthr_U$ becomes $\polythr_d^k$, the class of polynomial threshold functions on $d$ bits with degree bound $k$.

We will construct a minimal \emph{cocellular} resolution for $\linthr_U$, which will turn out to be homeomorphic as a topological space to the $d$-sphere $S^d$.
\footnote{For readers familiar with hyperplane arrangements: The cocellular resolution is essentially $S^d$ intersecting the fan of the hyperplane arrangement associated with the matroid on $U$.
The partial function labels on the resolution are induced from the covector labelings of the fan.}

We first vectorize the set $U$ by mapping each point to a vector, $u \mapsto \vec u, (u_1, \ldots, u_d) \mapsto (u_1, \ldots, u_d, 1)$.
We refer to the image of $U$ under this vectorization as $\vec U$.
Each oriented affine hyperplane $H$ in the original affine space $\R^d$ (including the hyperplane at infinity, i.e. all points get labeled positive or all points get labeled negative) corresponds naturally and bijectively to a vector hyperplane $\vec H$ in $\R^{d+1}$ which can be identified by the normal vector $\nu(\vec H)$ on the unit sphere $S^d \sbe \R^{d+1}$ perpendicular to $\vec H$ and oriented the same way.

For each vector $\vec u \in \R^{d+1}$, the set of oriented vector hyperplanes $\vec H$ that contains $\vec u$ is exactly the set of those which have their normal vectors $\nu(\vec H)$ residing on the equator $E_{\vec u} := \nu(\vec H)^\perp \cap S^d$ of $S^d$.
This equator divides $S^d \setminus E_{\vec u}$ into two open sets: $v \cdot\vec u > 0$ for all $v$ in one (let's call this set $R_{\vec u}^+$) and $v\cdot\vec u < 0$ for all $v$ in the other (let's call this set $R_{\vec u}^-$). 
Note that $\bigcap \{E_{\vec u} : u \in U\}$ is empty, since we have assumed at the beginning that $U$ affinely spans $\R^d$, and thus $\vec U$ (vector) spans $\R^{d+1}$.
The set of all such equators for all $\vec u$ divides $S^d$ into distinct open subsets, which form the top-dimensional (open) cells of a cell complex.
More explicitly, each cell $F$ (not necessarily top-dimensional and possibly empty) of this complex has a presentation as $\bigcap \{A_{\vec u}: u \in U\}$ where each $A_{\vec u}$ is one of $\{E_{\vec u}, R_{\vec u}^+, R_{\vec u}^-\}$.
If the cell is nonempty, then this presentation is unique and 
we assign the PF label $\pf f_F: \sbe U \to [2]$ defined by
$$\pf f_F(u) = \begin{cases}
1 & \text{if $A_{\vec u} = R_{\vec u}^+$} \\
0 & \text{if $A_{\vec u} = R_{\vec u}^-$} \\
\text{undefined} & \text{otherwise.}
\end{cases}
$$
It is easily seen that any point in $F$ is $\nu(\vec H)$ for some oriented affine hyperplane $H$ such that $U \setminus \dom \pf f_F$ lies on $H$, $\pf f_F^{-1}(1)$ lies on the positive side of $H$, and $\pf f_H^{-1}(0)$ lies on the negative side of $H$.

If $F = \emptycell$ is the empty cell, then we assign the empty function $\pf f_\emptycell = \emptyfun$ as its PF label.

\cref{fig:linthr_cocell_res} illustrates this construction.

\begin{figure}
\centering
\includegraphics[width=.6\textwidth]{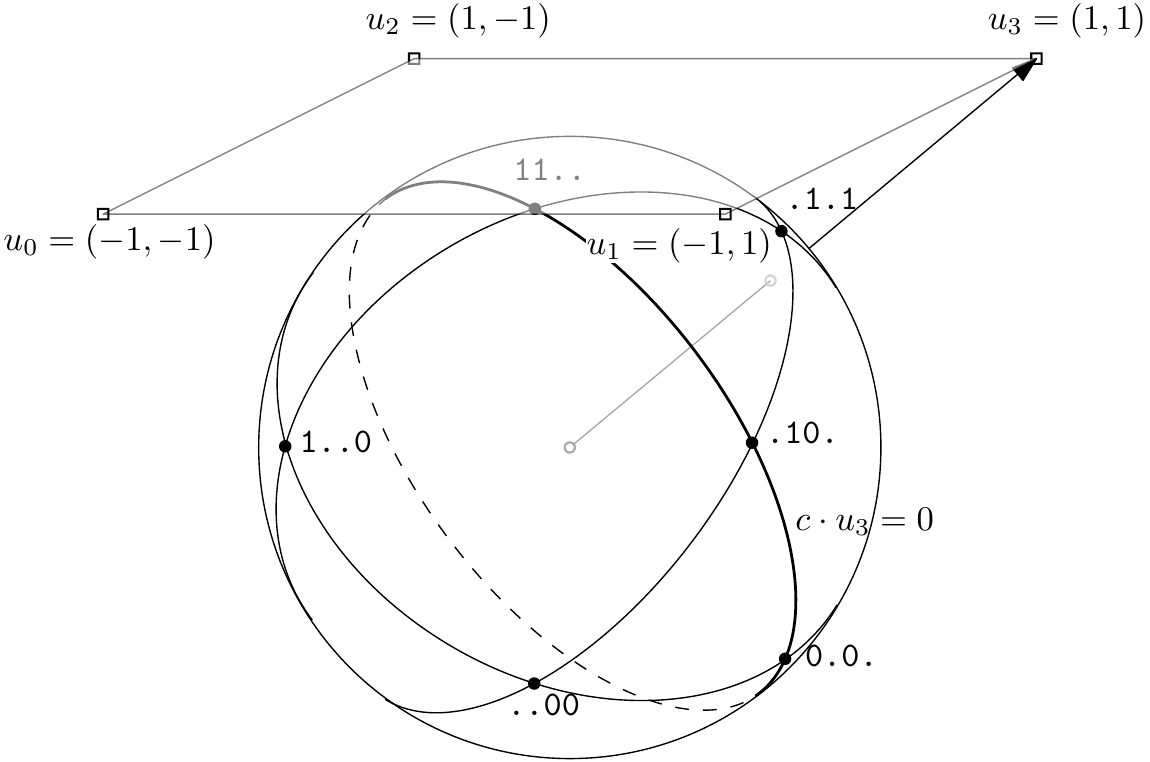}
\caption{Cocellular resolution of $\linthr_U$, where $U = \{u_0, u_1, u_2, u_3\}$ as labeled in the figure.
Each equator is the orthogonal space to the vector $\vec u_i$;
the case for $i = 3$ is demonstrated in the figure.
The text in monofont are the PF labels of vertices visible in this projection.
For example, $\texttt{1..0}$ represents the PF that sends $u_0$ to 1 and $u_3$ to 0, and undefined elsewhere.}
\label{fig:linthr_cocell_res}
\end{figure}

We claim this labeling gives a minimal cocellular resolution $X$ of $\linthr_U$.
We show this via \cref{lemma:pf_cocell_acyc}.

Suppose $\pf f$ is the empty function.
Then $X_{\spe \pf f} = X$, which is a complex with nontrivial reduced cohomology only at dimension $d$, where the rank of its cohomology is 1 (case 2 of \cref{lemma:pf_cocell_acyc}).
Now suppose $\pf f$ is nonempty.
Then $X_{\spe \pf f} = \bigcap_{u: \pf f(u) = 1} R_{\vec u}^+ \cap \bigcap_{u: \pf f(u) = 0} R_{\vec u}^-$ is an intersection of open half-spheres.
It is either empty (case 3 of \cref{lemma:pf_cocell_acyc}) or is homeomorphic, along with its boundary in $S^d$, to the open $d$-disk and its boundary $(D^d, \pd D^d)$, which has cohomology only at degree $d$ because $D^d/\pd \cong S^d$, where its rank is 1 (case 1 of \cref{lemma:pf_cocell_acyc}).

Thus our claim is verified.
$X$ is in fact minimal, as each cell has a unique monomial label.
We have proved the following.
\begin{thm}
The colabeled complex $X$ constructed as above is a minimal cocellular resolution of $\linthr_U$.
\end{thm}
\newcommand{\coballres}{\cresf{coBall}}
\begin{defn}
The colabeled complex $X$ is called the {\bf coBall resolution} of $\linthr_U$, written $\coballres_U$.
\end{defn}

$X$ can be made a polytope in an intuitive way, by taking the convex hull of all vertices on $X$ \footnote{As remarked in the previous footnote, $X$ is the intersection of a polyhedral fan with the unit sphere.
Instead of intersecting the fan with a sphere, we can just truncate the fans to get a polytope.}.
In addition,  we can obtain a minimal polyhedral {\it cellular} resolution $Y$ by taking the polar of this polytope and preserving the labels across polars.
Then the empty cell of $X$ becomes the unique dimension $d+1$ cell of $Y$.
We call this cellular resolution $Y$ the {\bf ball resolution}, written $\ballres_U$ or $\ballres$ when $U$ is implicitly understood, of $\linthr_U$. 

For any partial function $\pf f: \sbe U \to [2]$, define $\sigma \pf f$ to be the function 
$$\sigma \pf f(u) = \begin{cases}
+ & \text{if $\pf f(u) = 1$}\\
- & \text{if $\pf f(u) = 0$}\\
0 & \text{if $\pf f(u)$ is undefined.}
\end{cases}
$$

Let $\mathcal L(U) := \{\sgn(\psi \restrict U): \psi\text{ is a affine linear map}\}$ be the poset of covectors of $U$, under the pointwise order $0 < +, -$, with smallest element $\mathbf{0}$.
Therefore the cocircuits (minimal covectors) are the atoms of $\mathcal L(U)$.
Recall that $\mathcal L(U)$ has a rank function defined as
$$
\begin{cases}
\rank(a) = 0 & \text{if $a$ is a cocircuit;}\\
\rank(b) = 1 + \rank(a) & \text{if $b$ covers $a$.}
\end{cases}
$$
and $\rank(\mathbf 0) = -1$.

From the construction of $\coballres_U$, it should be apparent that each PF label is really a covector (identified by $\sigma$).
There is an isomorphism between $\mathcal L(U)$ and the face poset of $\coballres_U$:
$$F \sbe G \iff \pf f_F \sbe \pf f_G \iff \sigma \pf f_F \le \sigma \pf f_G.$$
Noting that $\rank(\sigma \pf f_F) = \dim F$ (and in particular, $\rank(\sigma \pf f_\emptycell) = -1 = \dim \emptycell$), this observation yields the following via \cref{prop:cocell_betti}
\begin{thm} \label{thm:linthr_betti}
The Betti number $\betti_{i, \pf f}(\linthr_U)$ is nonzero only when $\sigma \pf f$ is a covector of $U$.
In this case, $\betti_{i, \pf f}({\linthr_U}) = 1$ if $i = d - \rank(\sigma \pf f)$, and 0 otherwise.
In particular, the top dimensional Betti number of $\linthr_U$ is $\betti_{d+1, \emptyfun}(\linthr_U) = 1$.
\end{thm}
 
Via Hochster's dual formula, this means that the canonical suboplex of $\linthr_U$ is a homological $(d+1)$-sphere.

Let's look at the example of $U = \Monom^d$, so that $\linthr_U = \polythr_d^d = [\bcube^d \to 2]$.
In this case, $\vec U$ is an orthgonal basis for $\R^{2^d}$, and thus the equators of $\coballres_U$ are cut out by a set of pairwise orthogonal hyperplanes.
In other words, under a change of coordinates, $\coballres_U$ is just the sphere $S^{2^d-1}$ cut out by the coordinate hyperplanes, and therefore is combinatorially equivalent to the complete suboplex of dimension $2^d - 1$, with the PF labels given by the $\f 1 2(\sgn+1)$ function.
Its polar, $\ballres_U$, just recovers the cube resolution of $[2^d \to 2]$ as discussed in the beginning of \cref{sec:resolutions}.

When $\linthr_U = \polythr_d^k$, notice a very natural embedding of the cocellular resolution
$$\coballres_{\Monom^k} \hookrightarrow \coballres_{\Monom^{k+1}}$$ 
as the section of $\coballres_{\Monom^{k+1}}$ cut out by the orthogonal complement of 
$$\{\vec w_\gamma: \gamma \text{ is a monomial of degree $k+1$}\},$$
where $w_\gamma$ is all 0s except at the position where $\gamma$ appears in $\monom{\le k+1}$, and $\vec\cdot$ is the vectorization function as above, appending a 1 at the end.
This corresponds to the fact that a polynomial threshold of degree $k$ is just a polynomial threshold of degree $k+1$ whose coefficents for degree $k+1$ monomials are all zero.

This is in fact a specific case of a much more general phenomenon.
Let's call a subset $P \sbe \R^n$ {\bf openly convex} if $P$ is convex and
$$\forall u, v \in P, \exists \tau > 1: \tau u + ( 1 - \tau) v \in P.$$  
Examples include any open convex set in $\R^n$, any affine subspace of $\R^n$, and the intersections of any of the former and any of the latter.
Indeed, if $P$ and $Q$ are both openly convex, then, $P \cap Q$ is convex: for any $u, v \in P \cap Q$, if the definition of openly convex for $P$ yields $\tau = \rho > 1$ and that for $Q$ yields $\tau = \rho' > 1$, then we may take $\tau = \min(\rho, \rho')$ for $P \cap Q$, which works because $P \cap Q$ is convex.

An openly convex set is exactly one which is convex and, within its affine span, is equal to the interior of its closure.
 
Our proof that $\coballres_U$ is a minimal cocellular resolution can be refined to show the following
\begin{thm} \label{thm:convex_threshold_cocell_res}
Let $U \sbe \R^n$ be a point set that affinely spans $\R^n$.
Let $L$ be an openly convex cone of the vector space $\R^{n+1}$.
Define $Y$ to be the intersection of $X = \coballres_{U}$ with $L$, such that each nonempty open cell $Y \cap \ocell F$ of $Y$ gets the same exponent label $\lambda_Y(Y \cap \ocell F) = \lambda_X(\ocell F)$ as the open cell $\ocell F$ of $X$, and $Y$ has the empty cell $\emptycell$ with monomial label $\xx^{\lambda_Y(\emptycell)} = 1 \in S$ iff $L$ is vector subspace.
Then $Y$ is a minimal cocellular resolution of $\la \xx^{\lambda_Y(\ocell F)}: \ocell F \text{ is a top dimensional cell in $Y$}\ra$.  
\end{thm}
We will need a technical lemma, distinguishing the case when $L$ is a vector subspace and when it is not.
\begin{lemma}\label{lemma:openly_convex_dichotomy}
Let $L$ be an openly convex cone in $\R^q$.
Then either $L$ equals its vector span, or there is an open coordinate halfspace (i.e. $\{v \in \R^q: v_j > 0 \}$ or $\{v \in \R^q: v_j < 0 \}$) that contains $L$.
\end{lemma}
\begin{proof}
See \cref{sec:omitted_proofs}.
\end{proof}

\begin{proof}[Proof of \cref{thm:convex_threshold_cocell_res}]
It suffices to show that $Y_{\spe \pf f}$ for any PF $\pf f$ satisfies one of the three conditions of \cref{lemma:cocell_acyc}, and the minimality would follow from the uniqueness of labels.

If $L$ is a vector subspace, $Y_{\spe \emptyfun} = Y$ is a sphere (condition 2). Otherwise, $Y_{\spe \emptyfun} = Y$ is contained in an open halfspace $H$, and thus by projection from the origin onto an affine subspace $\pd H'$ parallel to $\pd H$, $Y$ is homeomorphic to $\pd H' \cap L$, an openly convex set of dimension $\dim Y$ (condition 1).
Whether $L$ is a vector space, for any nonempty PF $\pf f$, $Y_{\spe \pf f}$ is the intersection of the unit sphere (the underlying space of $\coballres_{\Monom^{d}}$), $L$, and a number of open halfspaces, and thus the intersection of openly convex sets contained in an open halfspace.
This is again homeomorphic to an openly convex set of dimension $\dim Y$ via projection to an affine subspace, if it is not empty. (condition 1/condition 3).
\end{proof}

Linear functionals on $\Monom^d$ are bijective with real functions on the boolean $d$-cube $\{-1, 1\}^d$.
Therefore the cone $L$ represents a cone of real functions when $U = \bcube^d$, and $Y$ is a minimal cellular resolution of the threshold functions of $L$.
In other words, we have the following corollary
\begin{cor} \label{cor:openly_convex_cocell_res}
Let $\clsf C \sbe [\{-1, 1\}^d \to 2]$ be the class obtained by {\it strongly thresholding} an openly convex cone $L$ of real functions $\{-1, 1\}^d \to \R$, i.e. $\clsf C = \{\f 1 2(\sgn(f)+1): f \in L, \forall u \in \{-1, 1\}^d[f(u) \not = 0]\} $.
Then $\clsf C$ has a minimal cocellular resolution of dimension equal to the dimension of the affine hull of $L$.
\end{cor}
This corollary specializes to the case when $L$ is any vector subspace of boolean functions.
The examples explored in the beginning of this section took $L$ as degree bounded polynomials.
We make the following formal definitions.
\begin{defn}
Let $L$ be a cone of real functions on $\{-1, 1\}^d$.
Suppose $\clsf C = \{\f 1 2(\sgn(f)+1): f \in L, f(u) \not = 0, \forall u \in \{-1, 1\}^d\}$.
We say $\clsf C$ is the {\bf strongly thresholded class of $L$}, written $\clsf C = \thr L$.
We call $\clsf C$ {\bf thresholded convex} if $L$ is openly convex.
We call $\clsf C$ {\bf thresholded linear} if $L$ is linear.
\end{defn}

While this corollary produces minimal cocellular resolutions for a large class of functions, it does not apply to all classes.
For example, the corollary shows that the Betti numbers of thresholded convex classes are either 0 or 1, but as we show in the next section, the linear functionals over finite fields have very large Betti numbers, so cannot be a thresholded convex class.
%

\subsubsection{Linear Functionals over Finite Fields}
\label{sec:linear_functionals}

Let $p$ be a prime power.
Define $\linfun^{p}_d \cong \Fld_p^{d*} \sbe [p^d \to p]$ to be the class of linear functionals over the $d$-dimensional vector space $[p]^d \cong \Fld_p^d$.
We will refer to elements of $\linfun^p_d$ as {\it covectors}.
Denote the affine span of a set of elements $g_1, \ldots, g_k$ by $\llp g_1, \ldots, g_k \rrp$.
In this section we construct the minimal resolution of $\linfun^d_p$.

Fix a linear order $\lhd$ on $\Fld_p^{d*}$.
We construct as follows a DAG $T_d$ of depth $d+1$ (with levels 1, ..., $d+1$), whose nodes are of the form $(f, V)$ where $V$ is an affine subspace of the dual space $\Fld_p^{d*}$ and $f$ is the $\lhd$-least element of $V$. (Therefore if any affine subspace appears in a node, then it appears only in that node --- indeed, every affine subspace appears in exactly one node.)

There is only one node at level 1, which we call the root.
This is the $\lhd$-least element along with $V = \Fld_p^{d*}$.

For any node $(f, V)$ where $\dim V > 1$, we add as its children the nodes $(g, W)$ where $W$ is a codimension-1 affine subspace of $V$ not containing $f$, and $g$ is the $\lhd$-least element of $W$.
By simple induction, one sees that all affine subspaces appearing on level $i$ of $T_d$ has dimension $d-i$.
In particular, the nodes at level $d+1$, the {\it leaf nodes}, are all of the form $(f, \{f\})$.
This completes the construction of $T_d$.

For each path $(f_1, V_1 = \Fld_p^{d*}), (f_2, V_2), \ldots, (f_{d+1}, V_{d+1})$ from the root to a leaf node, we have by construction $f_{1} \lhd f_{2} \lhd \cdots \lhd f_{d} \lhd f_{d+1}$.
Therefore, every such path is unique.

\begin{lemma}
Any node $(f, V)$ at level $i$ of $T_d$ has exactly $p^{d-i}-1$ children.
\end{lemma}
\begin{proof}
The children of $(f, V)$ are in bijection with the set of codimension-1 affine subspaces of $V$ not containing $f$.
Each nonzero covector in $V^*$ defines a vector hyperplane in $V$, whose cosets determine $p$ parallel affine hyperplanes.
Exactly one of these affine hyperplanes contain $f$.
Covectors $f$ and $g$ in $V^*$ determine the same hyperplane if $f = cg$ for some constant $c \in \Fld_p^d$, $c \not=0$.
As remarked above, $V$ has dimension $d -i$, and so has cardinality $p^{d-i}$.
Therefore there are $\f{p^{d-i}-1}{p-1}(p - 1) = p^{d-i} - 1$ affine hyperplanes of $V$ not containing $f$.
\end{proof}
\begin{lemma}
There are $\mathcal U_p(d) := \prod_{i=0}^{d-1} (p^{d-i}-1)$ maximal paths in the DAG $T_d$. (When $d = 0$, $\mathcal U_p(d) := 1$.)
\label{lemma:num_maximal_path_Td}
\end{lemma}
\begin{proof}
Immediately follows from the previous lemma.
\end{proof}

For example, suppose $p = 2$ and $\lhd$ is the right-to-left lexicographic order on the covectors: $0\cdots00 \lhd 0\cdots01 \lhd 0\cdots10 \lhd \cdots \lhd 1\cdots10 \lhd 1\cdots11$, where a covector $(x_1, \ldots, x_d) \mapsto a_1x_1 + \cdots a_d x_d$ is abbreviated as the bitstring $a_1a_2\cdots a_d$.
When $d=3$, the root is $(000, \Fld^{3*}_2)$.
There are then seven dimension $3 - 1 = 2$ affine planes in $\Fld_p^3$ not containing 000, so seven nodes at level 1:
\begin{itemize}
  \item Covector 001 for all affine planes containing 001, which are
  $$\{001, 111, 101, 011\}, \{001, 111, 100, 010\}, \{001, 101, 110, 010\}, \{001, 100, 110, 011\}.$$
  \item There are 3 other affine planes, which correspond to the following nodes
  \begin{enumerate}
    \item $(100, \{111, 110, 100, 101\})$
    \item $(010, \{111, 011, 010, 110\})$
    \item $(010, \{010, 011, 101, 100\})$
  \end{enumerate}
\end{itemize}


Or, suppose we choose to order covectors by the number of 1s and then lexicographically, $0\cdots000 \prec 0\cdots001 \prec 0\cdots010 \prec 0\cdots100 \prec \cdots 10\cdots000 \prec 0\cdots011 \prec 0 \cdots 101 \prec 0\cdots 110 \prec \cdots \prec 1\cdots11$.
Then the DAG will be exactly the same as above.


Once we have built such a DAG $T_d$, we can construct the corresponding cellular resolution $X$ of $\cani{\linfun^p_d}$. \footnote{If we treat $T_d$ as a poset, then the cellular resolution as a complex is a quotient of the order complex of $T_d$ by identifying $(f, V)$ with $(g, W)$ iff $f = g$.}
The cellular resolution will be simplicial and pure of dimension $d$.
Its vertex set is $\linthr_d^p \cong \Fld_p^{d*}$; each vertex has itself as the PF label.
For each maximal path 
$$(f_1, V_1 = \Fld_p^{d*}), (f_2, V_2), \ldots, (f_{d+1}, V_{d+1}),$$
we add a top simplex (of dimension $d$) with the vertex set $\{f_1, f_2, \ldots, f_{d+1}\}$.
As usual, the PF label of a face $F \sbe \linthr_d^p$ is just the intersection of the PF labels of its vertices.

\begin{lemma}
For an $k$-dimensional face $F$ of $X$, its PF label is a linear functional on a vector subspace of $\Fld_p^d$ of dimension $d - k$.
\label{lemma:flag_PFlabel_dim}
\end{lemma}
\begin{proof}
$F$ has $k+1$ vertices, $f_0, \ldots, f_k$.
Their intersection is the partial function defined on the subspace $W = \bigcap_{i=1}^k \ker (f_i - f_{i-1})$, and equals the restriction of $f_i$ to $W$ for any $i$.
The affine independence of $\{f_0, \ldots, f_k\}$ implies the vector independence of $\{(f_1 - f_0), \ldots, (f_k - f_{k-1})\}$.
Therefore $W$ has codimension $k$, as desired.
\end{proof}

Now, to show that $X$ is a minimal resolution, we will require the following lemma.

\begin{lemma}
Fix any linear order $\lhd$ on $\Fld_p^{d*}$.
Suppose $(g_1, g_2, \ldots, g_k)$ is a sequence of covectors such that $g_i$ is the $\lhd$-least element of the affine space generated by $(g_i, g_{i+1}, \ldots, g_k)$.
Then there is a maximal path in $T_d$ containing $(g_1, \ldots, g_k)$ as a subsequence.
\label{lemma:parity_path_extend}
\end{lemma}
\begin{proof}
We proceed by induction on $k$.
When $k=0$, the claim is vacuously true.
Assume $k \ge 1$.
We will show that there is a path $\wp$ from the root to a node $(g_1, V)$ with $V$ containing $W = \llp g_1, \ldots, g_k\rrp$, the affine subspace generated by $g_1, \ldots, g_k$.
Then we apply the induction hypothesis with $\Fld_p^{d*}$ replaced by $W$ and $(g_1, g_2, \ldots, g_k)$ replaced by $(g_2, g_3, \ldots, g_k)$ to obtain a path from $(g_1, V)$ to a leaf node, which would give us the desired result.

The first node of $\wp$ is of course the root.
We maintain the invariant that each node $(f, W')$ added to $\wp$ so far satisfies $W' \spe W$.
If we have added the nodes $(f_1, V_1), (f_2, V_2), \ldots, (f_j, V_j)$ in $p$, then either $f_j = g_1$, in which case we are done, or $V_j$ is strictly larger than $W$.
In the latter case, there exists an affine subspace $V_{j+1}$ of $V_j$ with $W \sbe V_{j+1} \subset V_j$ and $\dim V_{j+1} = \dim V_j - 1$, and we add $(f_{j+1}, V_{j+1})$ to the path $\wp$, with $f_{j+1}$ being the $\lhd$-least element of $V_{j+1}$.
This process must terminate because the dimension of $V_j$ decreases with $j$, and when it does, we must have $f_j = g_1$, and the path $\wp$ constructed will satisfy our condition.
\end{proof}
\begin{thm}
$X$ is a $d$-dimensional complex that minimally resolves $\linfun^p_d$.
\end{thm}

\begin{proof}
To prove that $X$ is a cellular resolution, it suffices to show that $X_{\spe \pf f}$ for any partial function $\pf f: \sbe [p^d \to p]$ is acyclic.
The set of $f \in \Fld_p^{d*}$ extending $\pf f$ is an affine subspace $W$.
Our strategy is to prove that if $\{g_1, \ldots, g_k\}$ generates an affine subspace of $W$ and is a face of $X$, then $\{g, g_1, \ldots, g_k\}$ is also a face of $X$, where $g$ is the $\lhd$-least element of $W$.
This would show that $X_{\spe \pf f}$ is contractible and thus acyclic.
But this is precisely the content of \cref{lemma:parity_path_extend}:
Any such $\{g_1, \ldots, g_k\}$ can be assumed to be in the order induced by being a subsequence of a maximal path of $T_d$.
This means in particular that $g_i$ is the least element of $\llp g_i, \ldots, g_k\rrp$.
A fortiori, $\{g, g_1, g_2, \ldots, g_k\}$ must also satisfy the same condition because $g$ is the least element of $W$.
Therefore \cref{lemma:parity_path_extend} applies, implying that $\{g, g_1, g_2, \ldots, g_k\}$ is a face of $X$, and $X$ is a cellular resolution as desired.

The resolution is minimal since the PF label of any face $P$ is a covector defined on a strictly larger subspace than those of its subfaces.
\end{proof}
\begin{defn}
The resolution $X$ is called the \textbf{flag resolution}, $\flagres^p_d$, of $\linfun^p_d$ with respect to $\lhd$.
\end{defn}

\begin{thm}
The Betti number $\betti_{i, \pf g}(\linfun_d^p)$ is nonzero only when $\pf g$ is a linear functional defined on a subspace of $\Fld_p^d$, and $i = d - \dim \dom \pf g$.
In this case, it is equal to $\mathcal U_p(i)$ (as defined in \cref{lemma:num_maximal_path_Td}).
\end{thm}
\begin{proof}
All the cells in the resolution $X$ have exponent labels of the form $\Gamma \pf g$ as stated in the theorem, and by \cref{lemma:flag_PFlabel_dim}, such cells must have dimension $i = d - \dim(\dom \pf g)$.
It remains to verify that the number $B$ of cells with PF label $\pf g$ is $\mathcal U_p(i)$.

The subset of $\Fld_p^{d*}$ that extends $\pf g$ is an affine subspace $W$ of dimension $d - \dim \dom \pf g = i$.
The number $B$ is the number of sequences $(g_0, \ldots, g_i) \in W^{i+1}$ such that $g_j$ is the $\lhd$-least element of $\llp g_j, \ldots, g_i\rrp$ for each $j$, and such that $\llp g_0, \ldots, g_i\rrp = W$.
If we treat $W \cong \Fld_p^{i*}$ and construct $T_i$ on $W$, then $B$ is exactly the number of maximal paths of $T_i$, which is $\mathcal U_p(i)$ by \cref{lemma:num_maximal_path_Td}.
\end{proof}

As discussed in \cref{sec:threshold_functions}, we have the following corollary because the Betti numbers of $\linfun_d^2$ can be greater than 1.
\begin{cor}
$\linfun_d^2$ is not a thresholded convex class.
\end{cor}



\subsubsection{Abnormal Resolutions}
\label{sec:abnormal}
\begin{figure}
\centering
\includegraphics[height=0.2\textheight]{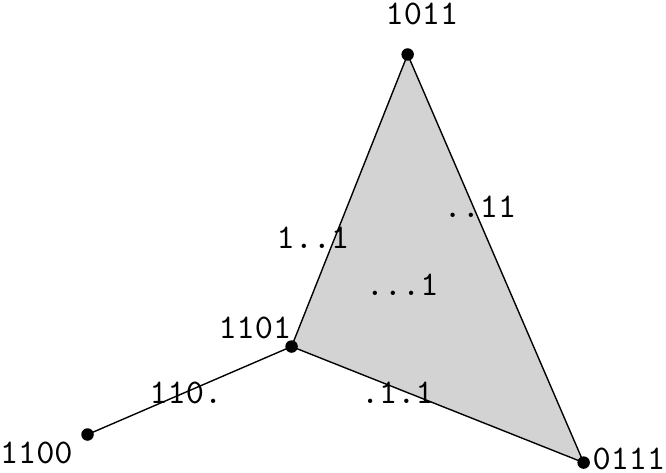}
\caption{An example of a nonpure minimal resolution of a boolean function class $\sbe [4 \to 2]$.
The labels are PF labels.
For example, $..11$ represents a partial function sending 2 and 3 to 1, and undefined elsewhere.}
\label{fig:nonpure_resolution}
\end{figure}
All of the classes exhibited above have pure minimal resolutions, but this need not be the case in general.
\cref{fig:nonpure_resolution} gives an example of a nonpure minimal resolution of a class $\sbe [4 \to 2]$.
It consists of a segment connected to a (solid) triangle.
This example can be generalized as follows.
Let $\clsf C \sbe [n+1 \to 2]$ be $\{\neg \delta_i = \Ind(u \not = i): i \in [n]\} \cup \{g := \ind(u \not\in \{n-1, n\})\}$.
Let $X$ be the simplicial complex on vertex set $\clsf C$, consisting of an $(n-1)$-dimensional simplex on $\{\neg \delta_i: i \in [n]\}$, and a segment attaching $\neg\delta_{n-1}$ to $g$.
With the natural PF labels, $X$ minimally resolves $\clsf C$ and is nonpure.

\begin{figure}
\centering
\includegraphics[width=.5\textwidth]{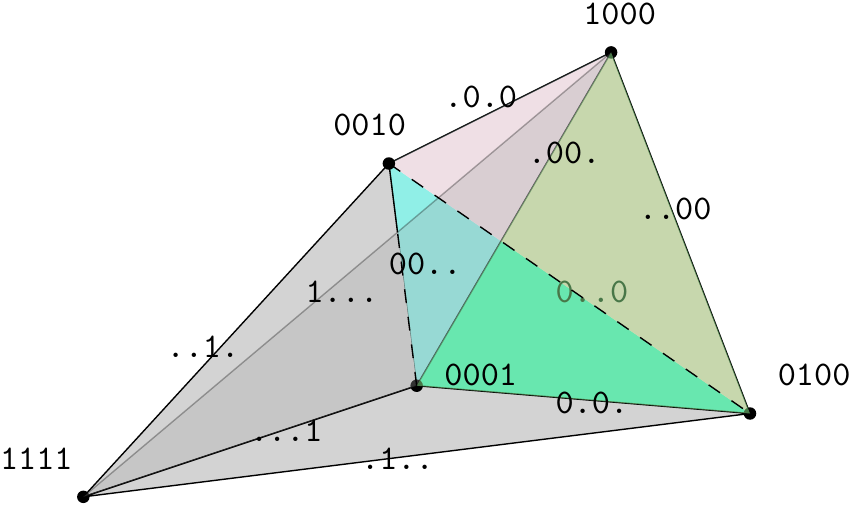}
\caption{Another example of nonpure minimal resolution of a boolean function class $\sbe [4 \to 2]$.
Only the vertices and edges are labeled (with PF labels).
Note that the maximal cells are the three triangles incident on the vertex $1111$ and the tetrahedron not incident on $1111$.
They have the same PF label, the empty function $\emptyfun$.
Therefore it is possible for a boolean function class to have nonzero Betti numbers in different dimensions for the same degree.}
\label{fig:mixed_homology}
\end{figure}

\begin{defn}
We say a class $\clsf C$ {\bf has pure Betti numbers} if for every PF $\pf f$, $\betti_{i, \pf f}(\clsf C) \not = 0$ for at most one $i$. 
\end{defn}
All of the classes above discussed in the previous sections have pure Betti numbers.
But this is not true in general.
\cref{fig:mixed_homology} shows a minimal resolution of a class $\clsf D\sbe [4 \to 2]$ that has three triangles and one tetrahedron as its top cells, and they all have the empty function as the PF label.
Thus $\betti_{2, \emptyfun}(\clsf D) = \betti_{3, \emptyfun}(\clsf D) = 1$.
This example can be generalized as follows.
Let $\clsf D \sbe [n \to 2]$ be $\{\delta_i: i \in [n]\} \cup \{\mathbf 1\}$.
Let $X$ be the simplicial complex on vertex set $\clsf D$, consisting of an $(n-1)$-dimensional simplex on $\{\delta_i: i \in [n]\}$ and triangles on each triple $\{\delta_i, \delta_j, \mathbf 1\}$ for each $i \not= j$.
With the natural PF labels, $X$ is a minimal cellular resolution of $\clsf D$, and the cells with PF label $\emptyfun$ are exactly the $(n-1)$-dimensional simplex and each of the triangles incident on $\mathbf 1$.
Thus the gap between the highest nontrivial Betti number and the lowest nontrivial Betti number for the same partial function can be linear in the size of the input space.

\subsection{Partial Function Classes}
\label{sec:PF_classes}
Most of the definitions we made actually apply almost verbatim to partial function classes $\pclsf C \sbe [\sbe n \to 2]$.
Here we list the corresponding definitions for PF classes and the propositions that hold PF classes as well as for function classes.
We omit the proofs as they are similar to the ones given before.

\begin{defn}
Let $\pclsf C \sbe [\sbe n \to m]$.
The {\bf canonical suboplex} $\Sbpx{\pclsf C}$ of $\pclsf C$ is the subcomplex of the complete $(n-1)$-dimensional $m$-suboplex consisting of all cells $F_{\pf f}$ where $\pf f$ has an extension in $\pclsf C$.

The {\bf canonical base ring} $S$ of $\pclsf C$ is the same as the canonical base ring of $[n \to m]$.
The {\bf Stanley-Reisner ideal} $I_{\pclsf C}$ of $\pclsf C$ is defined as the Stanley-Reisner ideal of $\Sbpx{\pclsf C}$ with respect to $S$.
The {\bf canonical ideal} of $\pclsf C$ is the dual ideal $\cani{\pclsf C}$ of its Stanley-Reisner ideal.
It is generated by $\{\xx^{\Gamma \pf f}: \pf f \in \pclsf C\}$, and generated minimally by $\{\xx^{\Gamma \pf f}: \pf f \in \pclsf C \text{ is maximal}\}$.

A Betti number $\betti_{i, \es b}(\cani{\pclsf C})$ is nonzero only if $\es b = \Gamma \pf f$ for some partial function $\pf f$ with extension in $\pclsf C$.
Thus we define $\betti_{i, \pf f}(\pclsf C) = \betti_{i, \Gamma \pf f}(\cani{\pclsf C})$.

\end{defn}

\begin{prop}[Counterpart of \cref{prop:mingen_SR}]\label{prop:mingen_SR_PF}
Let $\pclsf C \sbe [\sbe n \to m]$.
Each minimal generator of $I_{\pclsf C}$ is either 1) $x_{u, i} x_{u, j}$ for some $u \in [n], i \not= j \in [m]$, or 2) $\xx^{\graph \pf f}$ for some partial function $\pf f:\sbe [n] \to [m]$ such that $\pf f$ has no extension in $\clsf C$, but every proper restriction of $\pf f$ does.
In addition, the set of all such monomials is exactly the set of minimal generators of $I_{\clsf C}$.
\end{prop}

\begin{defn}
Let $\pclsf C \sbe [\sbe n \to m]$.
A labeled complex $(X, \lambda)$ is a {\bf (co)cellular resolution of partial class $\pclsf C$} if $(X, \lambda)$ is a (co)cellular resolution of $S/\cani{\pclsf C}$. 
\end{defn}

\begin{prop}[Counterpart of \cref{lemma:pf_label_propagate}]
If $(X, \lambda)$ is a cellular resolution of a partial class $\pclsf C \sbe [\sbe n \to m]$, then it is PF-labeled as well.
The PF label $\lambda(F)$ of a face $F$ is $\bigcap_{V \in F} \pf f_V$.
\end{prop}

\cref{lemma:pf_acyclicity} and \cref{lemma:pf_cocell_acyc} give conditions on when a PF-(co)labeled complex is a resolution, and they apply verbatim to resolutions of partial classes as well.
\cref{prop:extension_set_as_intersection} and \cref{cor:extension_complex_as_intersection} hold as well when $\clsf C$ is replaced by a partial class $\pclsf C$, but we will not use them in the sequel.

\subsection{Combining Classes}
\label{sec:combining}
We first give a few propositions on obtaining resolutions of a combination of two classes $\clsf C$ and $\clsf D$ from resolutions of $\clsf C$ and $\clsf D$.
\begin{prop}
Let $I$ and $J$ be two ideals of the same polynomial ring $S$.
If $(X_I, \lambda_I)$ is a polyhedral cellular resolution of $S/I$, and $(X_J, \lambda_J)$ is a cellular resolution of $S/J$, then the join $(X_I \star X_J, \lambda_I \star \lambda_J)$ is a cellular resolution of $S/(I + J)$, where we define $\lambda_I \star \lambda_J(F \star G) := \lcm(\lambda_I(F), \lambda_J(G))$.\label{prop:join_res}
\end{prop}

\begin{proof}
Let $\es a$ be an exponent sequence.
$(X_I \star X_J)_{\preceq \es a}$ is precisely $(X_I)_{\preceq \es a} \star (X_J)_{\preceq \es a}$, which is acyclic when both $(X_I)_{\preceq \es a}$ and $(X_J)_{\preceq \es a}$ are acyclic.
So $X_I \star X_J$ is a resolution.

The 0-cells of $X_I \star X_J$ are just the 0-cells of $X_I$ union the 0-cells of $X_J$, with the same labels, so $X_I \star X_J$ resolves $I + J$.
\end{proof}

Note however that in general $X_I \star X_J$ is not minimal even when $X_I$ and $X_J$ both are.

\begin{prop}
Let $\clsf C$ and $\clsf D$ be classes $\sbe \fun [n]m$.
If $(X_C, \lambda_C)$ is a cellular resolution of $\clsf C$, and $(X_D, \lambda_D)$ is a cellular resolution of ${\clsf D}$, then the join $(X_C \star X_D, \lambda_C \star \lambda_D)$ is a cellular resolution of ${\clsf C \cup \clsf D}$.
If $\mu_C$ is the PF labeling function of $X_C$ and $\mu_D$ is the PF labeling function of $X_D$, then the PF labeling function of $X_C \star X_D$ is given by
$$\mu_C \star \mu_D (F \star G):= \mu_C(F) \cap \mu_D(G).$$

\label{prop:union_res}
\end{prop}
\begin{proof}
By the above proposition, $(X_C \star X_D, \lambda_C \star \lambda_D)$ resolves $\cani{\clsf C} + \cani{\clsf D}$, which has minimal generators $\{\xx^{\Gamma f}: f \in \clsf C \cup \clsf D\}$.
The characterization of $\mu_C \star \mu_D$ follows from the the definition of $\lambda_C \star \lambda_D$.
\end{proof}

We will need to examine the ``difference'' between the Betti numbers of $I + J$ and those of $I$ and $J$.
The following lemma gives a topological characterization of this difference.

\begin{lemma} \label{lemma:product_join_exact_seq}
Let $I$ and $J$ be two monomial ideals of the same polynomial ring $S$.
Suppose $(X_I, \lambda_I)$ is a polyhedral cellular resolution of $S/I$, and $(X_J, \lambda_J)$ is a cellular resolution of $S/J$.
Label $X_I \times X_J$ by the function $\lambda_I \times \lambda_J: F\times G \mapsto \lcm(\lambda_I(F), \lambda_J(G))$ for nonempty cells $F$ and $G$; the empty cell has exponent label $\mathbf 0$.
If $\sigma$ is an exponent sequence, then there is a long exact sequence
\begin{align*}
\cdots \to \widetilde H_i((X_I \times X_J)_{\prec \sigma}) \to \widetilde H_i((X_I)_{\prec \sigma}) \oplus \widetilde H_i((X_J)_{\prec \sigma}) \to \widetilde H_i((X_I \star X_J)_{\prec \sigma}) \to \cdots
\end{align*}
where $i$ decreases toward the right.
\label{lemma:join_prec_long_exact}
\end{lemma}
\begin{proof}
One can check that $(X_I \star X_J)_{\prec \sigma}$ is the homotopy pushout of $(X_I)_{\prec \sigma} \gets (X_I \times X_J)_{\prec \sigma} \to (X_J)_{\prec \sigma}$.
The lemma then follows from the homotopy pushout exact sequence.
\end{proof}

We also have an algebraic version.
\begin{lemma}
Let $I$ and $J$ be two monomial ideals of the same polynomial ring $S$.
For each exponent sequence $\es a$, there is a long exact sequence
\begin{align*}
\cdots \to \kk^{\betti_{i, \es a}(I \cap J)} \to \kk^{\betti_{i, \es a}(I)} \oplus \kk^{\betti_{i, \es a}(J)} \to \kk^{\betti_{i, \es a}(I + J)} \to \kk^{\betti_{i-1, \es a}(I \cap J)} \to \cdots
\end{align*}
\end{lemma}
\begin{proof}
We have a short exact sequence
\begin{align*}
0 \to I \cap J \to I \oplus J \to I + J \to 0.
\end{align*}
By \cref{prop:Tor_betti}, we can apply $\Tor(-, \kk)$ to obtain the long exact sequence as stated.
\end{proof}

The ideal $I \cap J$ is generated by $\{\lcm (m_i, m_j): m_i \in \mingen(I), m_j \in \mingen(J)\}$.
When $I = \cani{\clsf C}$ and $J = \cani{\clsf D}$, $I \cap J = \la \xx^{\Gamma(f \cap g)}: f \in \clsf C, g \in \clsf D\}$.
Define the {\bf Cartesian Intersection} $\clsf C \cartcap \clsf D$ of $\clsf C$ and $\clsf D$ to be $\{f \cap g: f \in \clsf C, g \in \clsf D\}$.
This is a class of partial functions, and we can check $\cani{\clsf C} \cap \cani{\clsf D} = \cani{\clsf C \cartcap \clsf D}$.
So the above lemma can be restated as follows
\begin{lemma} \label{lemma:cartcap_is_diff}
Let $\clsf C, \clsf D \sbe [n \to m].$
For each PF $\pf f: \sbe [n] \to [m]$, there is a long exact sequence
\begin{align*}
\cdots \to \kk^{\betti_{i, \pf f}(\clsf C \cartcap \clsf D)} \to \kk^{\betti_{i, \pf f}(\clsf C)} \oplus \kk^{\betti_{i, \pf f}(\clsf D)} \to \kk^{\betti_{i, \pf f}(\clsf C \cup \clsf D)} \to \kk^{\betti_{i-1, \pf f}(\clsf C \cartcap \clsf D)} \to \cdots
\end{align*}
\end{lemma}

Next, we seek to produce from cellular resolutions of $\clsf C$ and $\clsf D$ a cellular resolution of the {\bf Cartesian Union} $\clsf C \amalg \clsf D$ of two classes $\clsf C \sbe [U \to V], \clsf D \sbe [U' \to V']$, defined as the class with elements $f \amalg g: U \sqcup U' \to V \sqcup V'$ for $f \in \clsf C, g \in \clsf D$, defined by 
$$f \amalg g(u) = \begin{cases}
f(u) & \text{if $u \in U$}\\
g(u) & \text{else.}
\end{cases}
$$
We start with the general version for ideals, and specialize to function classes.
\begin{prop}\label{prop:amalg_res}
Let $I$ be an ideal of polynomial ring $S$ and let $J$ be an ideal of polynomial ring $T$ such that $S$ and $T$ share no variables.
If $(X_I, \lambda_I)$ resolves $S/I$ and $(X_J, \lambda_J)$ resolves $S/J$, then $(X_I \times X_J, \lambda_I \amalg \lambda_J)$ resolves the ideal $S/(I \otimes J)$ with $I \otimes J := (I \otimes T)(S \otimes J)$ in the ring $S \otimes T$, where $\lambda_I \amalg \lambda_J(F \times G) = \lambda_I(F) \lambda_J(G)$ for any cells $F \in X_I$ and $G \in X_J$.
(Here, tensor $\otimes$ is over base ring $\kk$).
Furthermore, if $X_I$ and $X_J$ are both minimal then $(X_I \times X_J, \lambda_I \amalg \lambda_J)$ is minimal as well.
\end{prop}
\begin{proof}
Let $\bomega_0, \ldots, \bomega_{p-1}$ be minimal monomial generators of $I$ and let $\bgamma_0, \ldots, \bgamma_{q-1}$ be minimal monomial generators of $J$.
The ideal $I \otimes J$ is generated by $\{\bomega_i \bgamma_j: (i, j) \in [p] \times [q]\}$, which are furthermore minimal because $\{\bomega_i\}_i$ and $\{\bgamma_j\}_j$ are respectively minimal, and $S$ and $T$ share no variables.
The complex $X_I \times X_J$ has vertices $V_i \times V'_j$ for vertices $V_i \in X_I$ and $V_j \in X_J$.
If $V_i$ has label $\bomega_i$ and $V'_j$ has label $\bgamma_j$, then $V_i \times V'_j$ has label $\bomega_i \bgamma_j$ via $\lambda_I \amalg \lambda_J$.
Thus $X_I \times X_J$ resolves $S/(I \otimes J)$, if it is a resolution.

And in fact, it is, because for any exponent sequence $\es a$ wrt $S$ and exponent sequence $\es b$ wrt $T$, $(X_I \times X_J)_{\preceq \es a \amalg \es b} = (X_I)_{\preceq \es a} \times (X_J)_{\preceq \es b}$, which is acyclic (Here $\es a \amalg \es b$ is the exponent sequence whose values on variables in $S$ come from $\es a$ and whose values on variables in $T$ come from $\es b$).

The faces of a cell $F\times G \in X_I \times X_J$ are
$$\{F \times G': G' \sbe \pd G, \dim G' = \dim G -1\} \cup \{F' \times G: F' \sbe \pd F, \dim F' = \dim F - 1\}.$$
If $\lambda_I(F) \not= \lambda_I(F')$ for any $F' \subset F$ and $\lambda_J(G) \not=\lambda_J(G')$ for any $G' \subset G$, then $\lambda_I \amalg \lambda_J(F \times G) = \lambda_I(F) \lambda_J(G)$ is not equal to any of $\lambda_I(F') \lambda_J(G)$ or $\lambda_I(F) \lambda_J(G')$ for any of the above $F'$ or $G'$.
Therefore $(X_I \times X_J, \lambda_I \amalg \lambda_J)$ is minimal if $X_I$ and $X_J$ are.
\end{proof}

\begin{prop}
Let $\clsf C \sbe \fun [V]U$ and $\clsf D \sbe \fun[V']{U'}$.
If $(X_C, \lambda_C)$ is a cellular resolution of $\clsf C$, and $(X_D, \lambda_D)$ is a cellular resolution of $\clsf D$, then the product $(X_C \times X_D, \lambda_C \amalg \lambda_D)$ is a cellular resolution of $\clsf C \amalg \clsf D$.
Furthermore, if $X_C$ and $X_D$ are both minimal then $(X_C \times X_D, \lambda_C \amalg \lambda_D)$ is minimal as well.
\end{prop}

Finally, we want to construct cellular resolutions of restrictions of a function class to a subset of its input space.
\begin{defn}
Let $\clsf C \sbe [U \to V]$ and $U' \sbe U$.
Then the {\bf restriction class} $\clsf C \restrict U' \sbe [U' \to V]$ is defined as $\clsf C \restrict U' = \{f \restrict U': f \in \clsf C\}$.
\end{defn}

Again we start with a general algebraic version and then specialize to restriction classes.
\begin{prop}
Let $\mathbf X := \{x_i: i \in [n]\}$ and $\mathbf Y := \{y_j: j \in [m]\}$ be disjoint sets of variables.
Let $I$ be an ideal of polynomial ring $S = \kk[\mathbf X \sqcup \mathbf Y]$.
Suppose $(X, \lambda)$ resolves $I$.
Then $(X, \lambda \restrict_{\mathbf Y})$ resolves the ideal $I / \la x_i-1: x_i \in \mathbf X\ra$ in the ring $\kk[\mathbf Y]$, where
$$\lambda \restrict_{\mathbf Y} (F) := \lambda(F) / \la x_i - 1: x_i \in \mathbf X\ra.$$
\end{prop}

Essentially, if we just ignore all the variables in $\mathbf X$ then we still get a resolution, though most of the time the resulting resolution is nonminimal even if the original resolution is.

\begin{proof}
The subcomplex $(X, \lambda \restrict_{\mathbf Y})_{\preceq \mathbf y^{\es a}}$ for a monomial $\mathbf y^{\es a}$ in $\kk[\mathbf Y]$ is exactly the subcomplex $(X, \lambda)_{\preceq \xx^{\mathbf 1}\mathbf y^{\es a}}$, and hence acyclic. 
\end{proof}

One can easily see that the Stanley-Reisner ideal of $\clsf C \restrict U'$ is $I_{\clsf C} / \la x_{u, v} -1: u \not \in U', v \in V\ra$ and similarly the canonical ideal of $\clsf C \restrict U'$ is $\cani C / \la x_{u, v}-1: u \not \in U', v \in V\ra$ (both ideals are of the polynomial ring $S[x_{u, v}: u \in U', v \in V]$).
Then the following corollary is immediate.

\begin{prop}
Let $\clsf C \sbe \fun[V]U$ and $U' \sbe U$.
If $(X, \lambda)$ is a cellular resolution of $\clsf C$, then $(X, \lambda \restrict U' \times V)$ resolves $\clsf C \restrict U'$, where $\lambda \restrict U' \times V := \lambda \restrict \{x_{u, v}: u \in U', v \in V\}$.
Similarly, if $\resf C$ is an algebraic free resolution of $I_C^\star$, then $\resf C \restrict U' := \resf C / \la x_{u, v} - 1: u \not\in U', v \in V\ra$ is an algebraic free resolution of $I_{\clsf C \restrict U'}^\star$. 
\label{prop:restrict_resolution}
\end{prop}

Finally we show that there is a series of exact sequences relating the Betti numbers of $\clsf C \sbe [n \to 2]$ to the Betti numbers of $\clsf C \restrict U \sbe [n]$.
All of the below homology are with respect to $\kk$.
\begin{defn}
Let $\clsf C \sbe [n \to m]$ and $\pf f :\sbe [n] \to [m]$.
{\bf The class $\clsf C$ filtered by $\pf f$}, $\clsf C \filt \pf f$, is $\{f \setminus \pf f: \pf f \sbe f \in \clsf C\}$.
For any $U \sbe [n] \times [m]$ that forms the graph of a partial function $\pf f$, we also write $\clsf C \filt U = \clsf C \filt \pf f$.
\end{defn}

It should be immediate that $\Sbpx{\clsf C \filt U} = \link_U{\Sbpx{\clsf C}}$, so that by Hochster's dual formula,
$$\betti_{i, \pf f}(\clsf C) = \dim_\kk \rH_{i-1}(\link_{\graph \pf f}{\Sbpx{\clsf C}}) = \dim_\kk \rH_{i-1}(\Sbpx{\clsf C \filt\pf f}).$$

\begin{figure}
\centering
\includegraphics[width=.4\textwidth]{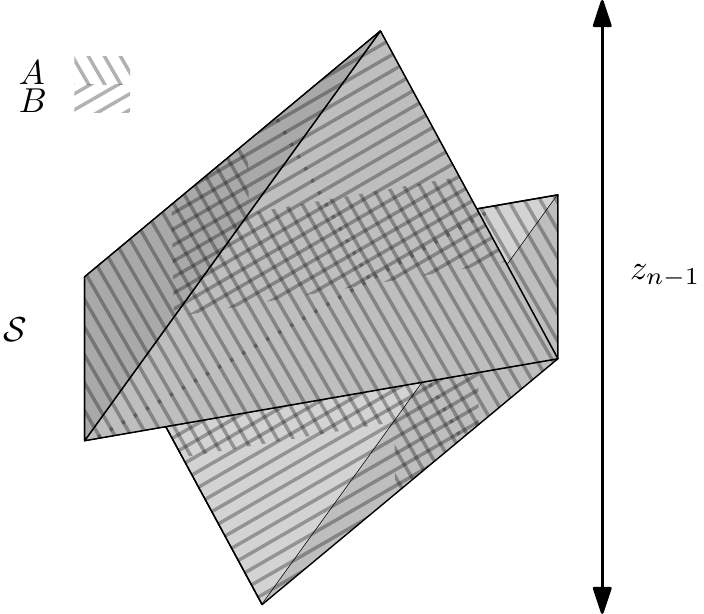}
\caption{$\Ss$ as the union $A \cup B$.}
\label{fig:restriction_mayer_vietoris}
\end{figure}
Consider the standard embedding of the complete $(n-1)$-dimensional suboplex $S^{n-1}_1 \cong \{z \in \R^n: \|z\|_1 = 1\}$.
Then $\Sbpx{\clsf C} \sbe S^{n-1}_1$ is the union of two open sets: $A := \Sbpx{\clsf C} \cap \{z \in \R^n: |z_{n-1}| < 2/3\}$ and $B := \Sbpx{\clsf C} \cap \{z \in \R^n: |z_{n-1}| > 1/3\}$ (see \cref{fig:restriction_mayer_vietoris}).
If all functions in $\clsf C$ sends $n-1$ to the same output, then $B$ is homotopy equivalent to a single point; otherwise $B$ contracts to 2 points.
$A$ deformation retracts onto $\Sbpx{\clsf C \restrict [n-1]}$.
The intersection $A \cap B$ deformation retracts to the disjoint union of two spaces, respectively homeomorphic to the links of $\Sbpx{\clsf C}$ with respect to the vertices $(n-1, 0), (n-1, 1) \in [n] \times [2].$ 
We therefore have the following long exact sequence due to Mayer-Vietoris
\begin{align*}
\cdots \to \rH_{i+1}(\Sbpx{\clsf C}) \to
\rH_i(\Sbpx{\clsf C\filt (n-1, 0)}) \oplus \rH_i(\Sbpx{\clsf C\filt (n-1, 1)}) \to  
\rH_i(\Sbpx{\clsf C \restrict [n-1]}) \oplus \rH_i(B) \to 
\rH_i(\Sbpx{\clsf C}) \to \cdots 
\end{align*}

If every function $f \in \clsf C$ has $f(n-1) = 1$, then $\clsf C \filt (n-1, 1) = \clsf C \restrict [n-1]$; a similar thing happens if all $f(n-1) = 0$.
So suppose $\clsf C \restrict \{n-1\} = [2]$.
Then $B \simeq \bullet \bullet$, and neither $\clsf C \filt (n-1, 0)$ nor $\clsf C\filt (n-1, 1)$ are empty.
Therefore the long exact sequence simplifies down to
\begin{align*}
\cdots \to \rH_{i+1}(\Sbpx{\clsf C}) \to
\rH_i(\Sbpx{\clsf C\filt (n-1, 0)}) \oplus \rH_i(\Sbpx{\clsf C\filt (n-1, 1)}) \to  
\rH_i(\Sbpx{\clsf C \restrict [n-1]}) \oplus \Z^{\ind(i=0)} \to 
\rH_i(\Sbpx{\clsf C}) \to \cdots 
\end{align*}

Note that for any simplicial complex $\Delta$, the link and restriction operations commute:
$$\link_\tau (\Delta \restrict \sigma) = (\link_\tau \Delta) \restrict \sigma.$$
Correspondingly, for function class $\clsf C$, filtering and restricting commute:
$$\clsf C \filt U \restrict V = \clsf C \restrict V \filt U.$$
Let $U := \graph \pf f$ for some $\pf f :\sbe [n-1] \to [2]$ and denote $U_0 := U \cup \{(n-1, 0)\}, U_1 := U \cup \{(n-1, 1)\}$.
The above long exact sequence generalizes to the following, by replacing $\clsf C$ with $\clsf C \filt U$ and applying the commutativity above:
\begin{align*}
\cdots \to \rH_{i+1}(\Sbpx{\clsf C\filt U}) \to
\rH_i(\Sbpx{\clsf C\filt U_0}) \oplus \rH_i(\Sbpx{\clsf C\filt U_1}) \to  
\rH_i(\Sbpx{\clsf C \restrict [n-1] \filt U}) \oplus \Z^{\ind(i=0)} \to 
\rH_i(\Sbpx{\clsf C \filt U}) \to \cdots 
\end{align*}
This yields via Hochster's formulas the following sequence relating the Betti numbers of $\clsf C$ and $\clsf C \restrict [n-1]$.
\begin{thm}\label{thm:restriction_betti_exact_seq}
Let $\clsf C \sbe [n \to 2]$, $\pf f :\sbe [n-1] \to [2]$, and $\pf f_0 := \pf f \cup (n-1\mapsto 0), \pf f_1 := \pf f \cup (n-1 \mapsto 1)$.
We have an exact sequence
\begin{align*}
\cdots \to \kk^{\betti_{i+1, \pf f}(\clsf C)} \to \kk^{\betti_{i, \pf f_0}(\clsf C) + \betti_{i, \pf f_1}(\clsf C)} \to \kk^{\betti_{i, \pf f}(\clsf C\restrict [n-1])+\ind(i=-1)} \to \kk^{\betti_{i, \pf f}(\clsf C)} \to \cdots
\end{align*}
\end{thm}

Using \cref{thm:restriction_betti_exact_seq} we can recapitulate the following fact about deletion in oriented matroids.
Below we write $V\setminus u$ for $V\setminus \{u\}$ in the interest of clarity.
\begin{cor}
Let $V$ be a point configuration with affine span $\R^d$ and $u \in V$.
Suppose $V\setminus u$ has affine span $\R^{d-e}$, where $e$ is either 0 or 1.
Then $\tau \in \{-, 0, +\}^{V\setminus u}$ is a covector of rank $r$ of $V \setminus u$ iff one of the following is true:
\begin{enumerate}
  \item $\tau_- := \tau \cup (u \mapsto -)$ is a covector of rank $r+e$ of $V$. \label{_tau-}
  \item $\tau_+ := \tau \cup (u \mapsto +)$ is a covector of rank $r+e$ of $V$. \label{_tau+}
  \item $\tau_0 := \tau \cup (u \mapsto 0)$ is a covector of rank $r+e$ of $V$, but $\tau_-$ and $\tau_+$ are not covectors of $V$. \label{_tau0}
\end{enumerate}
\end{cor}
\begin{proof}
Let $\clsf C = \linthr_V$ and $\clsf D = \linthr_{V \setminus u} = \clsf C \restrict (V \setminus u)$.
Write $\pf f := \inv\sigma \tau,\ \pf f_0 := \inv \sigma \tau_0,\ \pf f_+ := \inv \sigma \tau_+,\ \pf f_- := \inv \sigma \tau_-$.
$\betti_{i, \pf f}(\clsf C) = 1$ iff $\sigma \pf f$ is a covector of $V$ of rank $d - i$ by \cref{thm:linthr_betti}.

If \cref{_tau-} is true, but not \cref{_tau+}, then $\tau_0$ cannot be a covector of $V$ (or else subtracting a small multiple of $\tau_-$ from $\tau_0$ yields $\tau_+$).
As $\clsf C$ and $\clsf D$ both have pure Betti numbers, we have an exact sequence
\begin{align*}
0 \to \kk^{\betti_{j, \pf f_-}(\clsf C)} \to \kk^{\betti_{j, \pf f}(\clsf D)} \to 0
\end{align*}
where $j = d - \rank\tau_-$.
This yields that $\tau$ is a covector of rank $d-e-j = \rank\tau_- - e$.
The case that \cref{_tau+} is true but not \cref{_tau-} is similar.

If \cref{_tau-} and \cref{_tau+} are both true, then $\tau_0$ must also be a covector.
Furthermore, it must be the case that $\rank \tau_- = \rank\tau_+ = \rank \tau_0 + 1$.
Again as $\clsf C$ and $\clsf D$ have pure Betti numbers, we have an exact sequence
\begin{align*}
0 \to \kk^{\betti_{j+1, \pf f_0}(\clsf C)} \to \kk^{\betti_{j, \pf f_-}(\clsf C) + \betti_{j, \pf f_+}(\clsf C)} \to \kk^{\betti_{j, \pf f}(\clsf D)} \to 0 
\end{align*}
where $j = d - \rank\tau_-$.
Thus $\tau$ is a covector of rank $d-e-j=\rank\tau_- - e$.

Finally, if \cref{_tau0} is true, we immediately have an exact sequence
\begin{align*}
0 \to \kk^{\betti_{j, \pf f}(\clsf D)} \to \kk^{\betti_{j, \pf f_0}(\clsf C)} \to 0
\end{align*}
with $j = d - \rank \tau_0$, so $\tau$ is a covector of rank $d-e-j = \rank\tau_0 - e$.
\end{proof}

In general, if $\clsf C \sbe [n \to 2]$ and $\clsf C \restrict [n-1]$ are known to have pure Betti numbers, then \cref{thm:restriction_betti_exact_seq} can be used to deduce the Betti numbers of $\clsf C \restrict [n-1]$ directly from those of $\clsf C$.
This strategy is employed in the proof of \cref{cor:CM_under_restriction} in a later section.
It is an open problem to characterize when a class has pure Betti numbers.

\section{Applications}
\label{chap:app}

\subsection{Dimension Theory}
\label{sec:dimension}
In this section we investigate the relationships between VC dimension and other algebraic quantities derived from the Stanley-Reisner ideal and the canonical ideal.
\begin{defn}
Suppose $\clsf C \sbe \fun n$.
We say $\clsf C$ {\bf shatters} a subset $U \sbe [n]$ if $\clsf C \restrict U = [U \to 2]$.
The {\bf VC dimension} of $\clsf C$, $\vcdim \clsf C$, is defined as the largest $k$ such that there is a subset $U \sbe [n]$ of size $k$ that is shattered by $\clsf C$.
The {\bf VC radius} of $\clsf C$, $\vcr \clsf C$, is defined as the largest $k$ such that all subsets of $[n]$ of size $k$ are shattered by $\clsf C$.
\end{defn}
The VC dimension is a very important quantity in statistical and computational learning theory.
For example, suppose we can obtain data points $(u, f(u))$ by sampling from some unknown distribution $u \sim \mathcal P$, where $f$ is an unknown function known to be a member of a class $\clsf C$.
Then the number of samples required to learn the identity of $f$ approximately with high probability is $O(\vcdim \clsf C)$ \cite{kearns_introduction_1994}.
Simultaneous ideas also popped up in model theory \cite{shelah_combinatorial_1972}.
In this learning theory perspective, an extenture $\pf f$ of $\clsf C$ is what is called a minimal {\it nonrealizable sample}: there is no function in $\clsf C$ that {\it realizes} the input/output pairs of $\pf f$, but there is such functions for each proper subsamples (i.e. restrictions) of $\pf f$.

Note that $\clsf C$ shatters $U$ iff $I_{\clsf C\restrict U} = I_{\clsf C} \otimes_S S/ J_U$ equals $\la x_{u, 0} x_{u, 1}: u \in U\ra$ as an ideal of $S/J_U$, where $J_U = \la x_{\bar u, v} - 1: \bar u\not \in U, v \in V\ra$.
In other words, every nonfunctional minimal monomial generator of $I_{\clsf C}$ gets killed when modding out by $J_U$; so $\clsf C$ shatters $U$ iff every extenture of $\clsf C$ is defined on a point outside $U$.
Therefore if we choose $U$ to be any set with $|U| < \min\{|\dom \pf f|: \pf f \in \ex \clsf C \}$, then $\clsf C$ shatters $U$.
Since $\dom \pf f$ is not shattered by $\clsf C$ if $\pf f$ is any extenture, this means that
\begin{thm}\label{thm:vcdim+1_ge_min_extenture}
For any $\clsf C \subset \fun n$ not equal to the whole class $[n \to 2]$, 
$$\vcr \clsf C = \min\{|\dom \pf f|: \pf f \in \ex \clsf C \} - 1.$$
\end{thm}

Define the {\bf collapsing map} $\pi: \kk[x_{u, 0}, x_{u, 1}: u \in [n]] \to \kk[x_u: u \in [n]]$ by $\pi(x_{u, i}) = x_u$.
If $U \sbe [n]$ is shattered by $\clsf C$, then certainly all subsets of $U$ are also shattered by $\clsf C$.
Thus the collection of shattered sets form an abstract simplicial complex, called the {\bf shatter complex} $\Shx_{\clsf C}$ of $\clsf C$.
\begin{thm} \label{thm:shatter_iff_not_in_collapsed_ideal}
Let $I$ be the the Stanley-Reisner ideal of the shatter complex $\Shx_{\clsf C}$ in the ring $S' = \kk[x_u: u \in [n]]$.
Then $\pi_* I_{\clsf C} = I + \la x_u^2: u \in [n]\ra$.
Equivalently, $U \in \Shx_{\clsf C}$ iff $\xx^U \not\in \pi_* I_{\clsf C}$.
\end{thm}
\begin{proof}
$U$ is shattered by $\clsf C$ iff for every $\pf f: U \to [2]$, $\pf f$ has an extension in $\clsf C$, iff $\xx^{\graph \pf f} \not \in I_{\clsf C}, \forall \pf f:U \to [2]$, iff $\xx^U \not \in \pi_* I_{\clsf C}$. 
\end{proof}
We immediately have the following consequence.
\begin{thm}
$\vcdim \clsf C = \max\{|U|: \xx^U \not \in \pi_* I_{\clsf C}\}$.
\end{thm}

Recall the definition of projective dimension \cite{miller_combinatorial_2005}.
\begin{defn}
The length of a minimal resolution of a module $M$ is the called the {\bf projective dimension}, $\projdim M$, of $M$.
\end{defn}

We make the following definitions in the setting of function classes.
\begin{defn}
For any $\clsf C \sbe [n \to 2]$, the {\bf homological dimension} $\hdim \clsf C$ is defined as the projective dimension of $\cani{\clsf C}$, the length of the minimal resolution of $\cani{\clsf C}$.
The {\bf Stanley-Reisner dimension} $\srdim \clsf C$ is defined as the projective dimension of the Stanley-Reisner ring $S/I_{\clsf C}$.
\end{defn}

One can quickly verify the following lemma.
\begin{lemma}
If $S/I_{\clsf C}$ has a minimal cellular resolution $X$, then $\srdim \clsf C = \dim X + 1$.
If $\clsf C$ has a minimal cellular resolution $X$, then $\hdim \clsf C = \dim X$.
The same is true for cocellular resolutions $Y$ if we replace $\dim X$ with the difference between the dimension of a top cell in $Y$ and that of a bottom cell in $Y$.
\end{lemma}

Recall the definition of regularity \cite{miller_combinatorial_2005}.
\begin{defn}
The {\bf regularity} of a $\N^n$-graded module $M$ is
$$\reg M = \max\{|\es b| - i: \betti_{i, \es b}(M) \not= 0\},$$
where $|\es b| = \sum_{j=1}^n b_i.$
\end{defn}

There is a well known duality between regularity and projective dimension.
\begin{prop}\cite[thm 5.59]{miller_combinatorial_2005}\label{prop:regularity_projdim}
Let $I$ be a squarefree ideal.
Then $\projdim(S/I) = \reg(I^\star)$.
\end{prop}

This implies that the Stanley-Reisner dimension of $\clsf C$ is equal to the regularity of $\cani{\clsf C}$.
For each minimal resolutions we have constructed, it should be apparent that $\max\{|\Gamma \pf f| - i: \betti_{i, \pf f}(\clsf C) \not = 0\}$ occurs when $i$ is maximal, and thus for such an $\pf f$ with smallest domain it can be computed as $\#\text{variables} - |\dom \pf f| - \hdim \clsf C$.
Altogether, by the results of \cref{sec:resolutions}, we can tabulate the different dimensions for each class we looked at in this work in \cref{tab:dims}.
\begin{table}[]
\centering
\begin{tabular}{M|MMM}
             & \hdim                   & \srdim & \vcdim                  \\
\text{[}n \to 2\text{]}
             & n                       & n      & n                       \\
\{f\}        & 0                       & n      & 0                       \\
\cDelta_n    & n-1                     & n+1      & 1                       \\
\monconj_d   & d                       & 2^{d+1} - d - 1      & d\quad\text{\cite{natschlager_exact_1996}}\\
\CLconj_d    & d+1                     & 2^{d+1}-d-1      & d \quad\text{\cite{natschlager_exact_1996}}\\
\linthr_d    & d+1                     & 2^{d+1}-d-1       & d+1 \quad\text{\cite{anthony_discrete_2001}}
\\
\polythr_d^k & \Sigma_0^k              & 2^{d+1} - \Sigma_0^k = 2^d + \Sigma_{k+1}^d    & \Sigma_0^k\quad\text{\cite{anthony_discrete_2001}}
\\
\linfun_d^2  & d                       & 2^{d+1} - d - 1       & d                      
\end{tabular}
\caption{Various notions of dimensions for boolean function classes investigated in this work.
$\Sigma_j^k := \sum_{i=j}^k \binom d i$.
The VC dimensions without citation can be checked readily.}
\label{tab:dims}
\end{table}

For all classes other than $\cDelta_n$, we see that $\hdim$ is very close to $\vcdim$.
We can in fact show the former is always at least thte latter.
\begin{prop}\label{prop:hom_dim_ineq}
Let $\clsf C \sbe \fun[V]U$ and $U' \sbe U$.
Then $\hdim \clsf C \ge \hdim \clsf C \restrict U'$.
\end{prop}
\begin{proof}
Follows from \cref{prop:restrict_resolution}.
\end{proof}
\begin{thm} \label{thm:hdim_ge_vcdim}
For any $\clsf C \sbe [n \to 2]$, $\hdim \clsf C \ge \vcdim \clsf C$. 
\end{thm}
\begin{proof}
Let $U \sbe [n]$ be the largest set shattered by $\clsf C$.
We have by the above proposition that $\hdim \clsf C \ge \hdim \clsf C \restrict U$.
But $\clsf C \restrict U$ is the complete function class on $U$, which has the cube minimal resolution of dimension $|U|$.
Therefore $\hdim \clsf C \ge |U| = \vcdim \clsf C$. 
\end{proof}
As a consequence, we have a bound on the number of minimal generators of an ideal $I$ expressable as a canonical ideal of a class, courtesy of the Sauer-Shelah lemma \cite{kearns_introduction_1994}.
\begin{cor}
Suppose ideal $I$ equals $\cani{\clsf C}$ for some $\clsf C \sbe [n \to 2]$.
Then $I$ is minimally generated by a set no larger than $O(n^d)$, where $d$ is the projective dimension of $I$.
\end{cor}

However, in contrast to VC dimension, note that homological dimension is not monotonic: $\cDelta_{2^d} \sbe \CLconj_d$ but the former has homological dimension $2^d$ while the latter has homological dimension $d+1$.
But if we know a class $\clsf C \sbe [n \to 2]$ has $\hdim \clsf C = \vcdim \clsf C$, then $\clsf C \sbe \clsf D$ implies $\hdim \clsf C \le \hdim \clsf D$ by the monotonicity of VC dimension.
We write this down as a corollary.
\begin{cor} \label{cor:hdim_monotone}
Suppose $\clsf C, \clsf D \sbe [n \to 2]$.
If $\hdim \clsf C = \vcdim \clsf C$, then $\clsf C \sbe \clsf D$ only if $\hdim \clsf C \le \hdim \clsf D$.
\end{cor}

The method of restriction shows something more about the Betti numbers of $\clsf C$.
\begin{thm} \label{thm:shatter_iff_betti_ge1}
$\clsf C$ shatters $U \sbe [n]$ iff for every partial function $\pf f: \sbe U \to [2]$, there is some $\pf g: \sbe [n] \to [2]$ extending $\pf f$ such that $\betti_{|U| - |\dom \pf f|, \pf g}(\clsf C) \ge 1$.
\end{thm}
\begin{proof}
The backward direction is clear when we consider all total function $\pf f: U \to [2]$.

From any (algebraic) resolution $\resf F$ of $\cani{\clsf C}$, we get an (algebraic) resolution $\resf F \restrict U$ of $\cani{\clsf C \restrict U}$ by ignoring the variables $\{ \xx_{u, v}: u \not \in U, v \in [2]\}$.
If for some $\pf f: \sbe U \to [2]$, for all $\pf g: \sbe [n] \to [2]$ extending $\pf f$, $\betti_{|U| - |\dom \pf f|, \pf g}{\clsf C} = 0$, then there is the $(|U| - |\dom \pf f|)$th module of $\resf F \restrict U$ has no summand of degree $\Gamma\pf g$, which violates the minimality of the cube resolution of $\clsf C \restrict U$. 
\end{proof}

There is also a characterization of shattering based on the Stanley-Reisner ideal of a class.
We first prove a trivial but important lemma.

\begin{lemma} \label{lemma:top_hom_nonzero_iff_complete_suboplex}
Suppose $\Delta$ is an $n$-dimensional suboplex.
Then $\rH_n(\Delta) \not= 0$ iff $\Delta$ is complete.\footnote{
The proof given actually works as is when $\Delta$ is any pure top dimensional subcomplex of a simplicial sphere.}
\end{lemma}
\begin{proof}
The backward direction is clear.

Write $S^n_1$ for the complete $n$-dimensional suboplex. 
Suppose $\Delta \not= S^n_1$.
Choose an $n$-dimensional simplex $F$ not contained in $\Delta$.
Let $\nabla$ be the complex formed by the $n$-dimensional simplices not contained in $\Delta$ or equal to $F$.
By Mayer-Vietoris for simplicial complexes, we have a long exact sequence
\begin{align*} 
\cdots \to \rH_n(\nabla \cap \Delta) \to \rH_n(\nabla) \oplus \rH_n(\Delta) \to \rH_n(\nabla \cup \Delta) \to \rH_{n-1}(\nabla \cap \Delta) \to \cdots
\end{align*}
Now $\nabla \cup \Delta$ is just $S^n_1 \setminus \intrr F$, which is homeomorphic to an $n$-dimensional disk, and hence contractible.
Hence $\rH_m(\nabla \cup \Delta) = 0, \forall m > 0$, and therefore $\rH_m(\nabla \cap \Delta) \cong \rH_m(\nabla) \oplus \rH_m(\Delta), \forall m > 0$.
But $\nabla \cap \Delta$ has dimension at most $n-1$, so $\rH_n(\nabla \cap \Delta) = 0$, implying $\rH_n(\nabla) = \rH_n(\Delta) = 0$, as desired.
\end{proof}

\begin{thm} \label{thm:shatter_iff_nonzero_SR_betti}
Let $\clsf C \sbe [n \to 2]$.
Suppose $U \sbe [n]$ and let $\tau = U \times [2]$.
Then $\clsf C$ shatters $U$ iff $\betti_{|U| - 1, \tau}(I_{\clsf C}) \not= 0$.
\end{thm}
\begin{proof}
$\clsf C$ shatters $U$ iff $\clsf C \restrict U = [U \to 2]$.
The canonical suboplex of $\clsf C \restrict U$ is $\Sbpx{\clsf C \restrict U} = \Sbpx{\clsf C} \restrict \tau$.
By the above lemma, $\Sbpx{\clsf C} \restrict \tau$ is complete iff $\rH_{|U| - 1}(\Sbpx{\clsf C} \restrict \tau) \not = 0$ iff $\rH^{|U| - 1}(\Sbpx{\clsf C} \restrict \tau; \kk) \not = 0$.
By Hochster's formula (\cref{prop:hochster}), the dimension of this reduced cohomology is exactly $\betti_{|U|-1, \tau}(I_{\clsf C})$.
\end{proof}

The above yields another proof of the dominance of homological dimension over projective dimension.
\begin{proof}[Second proof of \cref{thm:hdim_ge_vcdim}]
By \cref{prop:regularity_projdim}, $\hdim \clsf C + 1 = \projdim(S/\cani{\clsf C}) = \reg(I_{\clsf C})$.
By \cref{thm:shatter_iff_nonzero_SR_betti}, the largest shattered set $U$ must satisfy $\betti_{|U|-1, U\times[2]}(I_{\clsf C}) \not = 0$, so by the definition of regularity,
$$\vcdim \clsf C = |U| = |U\times[2]| - (|U|-1) - 1 \le \reg(I_{\clsf C}) -1 = \hdim \clsf C.$$
\end{proof}

From the same regularity argument, we obtain a relation between homological dimension and the maximal size of any minimal nonrealizable samples.
\begin{thm}
For any minimal nonrealizable sample $\pf f$ of $\clsf C$, we have
$$|\pf f| \le \hdim \clsf C + 1.$$
\end{thm}
\begin{proof}
Again, $\hdim \clsf C + 1 = \reg(I_{\clsf C})$.
For each extenture (i.e. minimal nonrealizable sample) $\pf f$, $\xx^{\graph \pf f}$ is a minimal generator of $I_{\clsf C}$, so we have $\betti_{0, \graph \pf f}(I_\clsf C) = 1$.
Therefore,
$$|\pf f| \le \reg(I_{\clsf C}) = \hdim \clsf C + 1.$$
\end{proof}

It is easy to check that equality holds for $\clsf C = \monconj, \linfun, \polythr$.

Combining \cref{thm:shatter_iff_not_in_collapsed_ideal}, \cref{thm:shatter_iff_betti_ge1}, and \cref{thm:shatter_iff_nonzero_SR_betti}, we have the equivalence of three algebraic conditions
\begin{cor}
Let $\clsf C \sbe [n \to 2]$ and $U \sbe [n]$.
The following are equivalent
\begin{enumerate}
  \item $\clsf C$ shatters $U$.
  \item $\xx^U \not \in \pi_* I_{\clsf C}$.
  \item $\forall \pf f: \sbe U \to [2]$, there is some $\pf g: \sbe[n] \to [2]$ extending $\pf f$ such that $\betti_{|U|-|\dom f|, \pf g}(\clsf C) \ge 1$.
  \item $\betti_{|U|-1, U \times [2]}(I_{\clsf C}) \not = 0$.
\end{enumerate}
\end{cor}
The above result together with \cref{cor:hdim_monotone} implies several algebraic conditions on situations in which projective dimension of an ideal is monotone.
Here we write down one of them.
\begin{cor}
Let $S = \kk[x_{u, i}: u \in [n], i \in [2]]$.
Suppose ideals $I$ and $J$ of $S$ are generated by monomials of the form $\xx^{\Gamma f}, f \in [n \to 2]$.
If $\max\{|U|: \xx^U \not\in \pi_* I\} = \projdim I$, then $I \sbe J$ implies $\projdim I \le \projdim J$.
\end{cor}

\subsection{Cohen-Macaulayness}
\label{sec:CM}
We can determine the Betti numbers of dimension 1 of any class of boolean functions.
Let $\clsf C \sbe [n \to 2]$.
Write $\clsf C_{\spe \pf f} := \{h \in \clsf C: h \spe \pf f\}$.
Then we have the following theorem.

\begin{thm}\label{thm:dim1_betti}
The 1-dimensional Betti numbers satisfy 
$$\betti_{1, \pf f}(\clsf C) = \begin{cases}
1 & \text{if $|\clsf C_{\spe \pf f}| = 2$}\\
0 & \text{otherwise.}
\end{cases}$$

More precisely, let $\{\epsilon_f: f \in \clsf C\}$ be a set of basis, each with degree $\Gamma f$, and define
$$\phi: \bigoplus_{f \in \clsf C} S \epsilon_f \onto \cani{\clsf C},\quad
\phi(\epsilon_f) = \xx^{\Gamma f}.$$
Let $\bomega_{f, g} = \xx^{\Gamma f} / \xx^{\Gamma (f \cap g)}$ and $\zeta_{f, g} := \bomega_{f, g} \epsilon_g - \bomega_{g, f} \epsilon_f$.
Then $\ker \phi$ has minimal generators
$$\{\zeta_{f, g}: \clsf C_{\spe \pf f} = \{f, g\}, f \prec g\},$$
where $\prec$ is lexicographic ordering (or any linear order for that matter).
\end{thm}
We will use the following lemma from \cite{eisenbud_commutative_1994}.
\begin{lemma}[\cite{eisenbud_commutative_1994}~Lemma 15.1 bis]
$\ker \phi$ is generated by $\{\zeta_{h, h'}: h, h' \in \clsf C\}$. 
\end{lemma}
\begin{proof}[Proof of \cref{thm:dim1_betti}]
It's clear that the latter claim implies the former claim about Betti numbers.

We first show that $G = \{\zeta_{f, g}: \clsf C_{\spe \pf f} = \{f, g\}, f \prec g\}$ is a set of generators as claimed.
By the lemma above, it suffices to show that $\zeta_{h, h'}$ for any two functions $h\prec h' \in \clsf C$ can be expressed as a linear combinations of $G$.
Denote by $\|f - g\|_1$ the $L_1$ distance $n - |\dom(f \cap g)|$.
We induct on the size of the disagreement $p = \|h - h'\|_1$.
When $p = 1$, $\zeta_{f, g} \in G$, so there's nothing to prove.
Suppose the induction hypothesis is satisfied for $p \le q$ and set $p = q+1$.
Let $\pf f = h \cap h'$.
If $\clsf C_{\spe \pf f}$ has size 2 then we are done.
So assume $|\clsf C_{\spe \pf f}| \ge 3$ and let $h''$ be a function in $\clsf C_{\spe \pf f}$ distinct from $h$ or $h''$.
There must be some $u, u'\in [n] \setminus \dom \pf f$ such that $h(u) = h''(u) = \neg h'(u)$ and $h'(u') = h''(u') = \neg h(u')$.
Indeed, if such a $u$ does not exist, then $h''(v) = h'(v)$ for all $v \in [n] \setminus \dom \pf f$, and thus $h'' = h'$, a contradiction; similarly, if $u'$ does not exist, we also derive a contradiction.
Therefore $\|h - h''\|_1, \|h' - h''\| \le q$, and by induction hypothesis, $\zeta_{h, h''}$ and $\zeta_{h', h''}$ are both expressible as linear combination of $G$, and thus $\zeta_{h, h'} = \zeta_{h, h''} - \zeta_{h'', h'}$ is also expressible this way.
This proves that $G$ is a set of generators.

For any partial $\pf f$, if $\clsf C_{\spe \pf f} = \{f, g\}$, then the degree $\xx^{\Gamma \pf f}$ strand of $\phi$ is the map of vector spaces
$$\kk \bomega_{f,g} \epsilon_g\oplus \kk \bomega_{g, f} \epsilon_f \to \kk \xx^{\Gamma \pf f}, (\omega, \omega') \mapsto \omega + \omega'$$
whose kernel is obviously $\kk \zeta_{f, g}$.
Therefore, $G$ must be a minimal set of generators.
\end{proof}

\begin{defn}
Let $\clsf C \sbe [n \to 2]$ and $f,g \in \clsf C.$
If $\clsf C_{f \cap g} = \{f, g\}$, then we say {\bf $f$ and $g$ are neighbors in $\clsf C$}, and write $f \nei_{\clsf C} g$, or $f \nei g$ when $\clsf C$ is clear from context.
\end{defn}

Next, we discuss the conditions under which $S/I_{\clsf C}$ and $S/\cani{\clsf C}$ could be Cohen-Macaulay.
Recall the definition of Cohen-Macaulayness.
\begin{defn} [\cite{miller_combinatorial_2005}]
A monomial quotient $S/I$ is {\bf Cohen-Macaulay} if its projective dimension is equal to its {\bf codimension} $\codim S/I := \min\{\supp\bomega: \bomega \in \mingen(I^\star)\}$.
\end{defn}
Cohen-Macaulay rings form a well-studied class of rings in commutative algebra that yields to a rich theory at the intersection of algebraic geometry and combinatorics.
The mathematician Melvin Hochster famously wrote ``Life is really worth living'' in a Cohen-Macaulay ring \cite{hochster_canonical_1989}.

By \cite[Prop 1.2.13]{bruns_cohen-macaulay_1998}, we have that $S/I$ is Cohen-Macaulay for $I$ squarefree only if every minimal generator of $I^\star$ has the same support size.
Then the following theorem shows that requiring $S/\cani{\clsf C}$ to be Cohen-Macaulay filters out most interesting function classes, including every class considered above except for singleton classes.
We first make a definition to be used in the following proof and in later sections.
\begin{defn} \label{defn:full}
Let $\clsf D \sbe [n \to m]$.
We say $\clsf D$ is {\bf full} if for every pair $(u, v) \in [n] \times [m]$, there is some function $h \in \clsf D$ with $h(u) = v$ --- i.e. $\bigcup \{\graph h : h \in \clsf D\} = [n] \times [m]$. 
\end{defn}
\begin{thm} \label{thm:cohen_macaulay_canonical}
Let $\clsf C \sbe [n \to 2]$.
The following are equivalent
\begin{enumerate}
  \item $S/\cani{\clsf C}$ is Cohen-Macaulay.
  \item Under the binary relation $\sim$, $\clsf C_{\spe \pf f}$ forms a tree for every PF $\pf f: \sbe[n] \to [2].$
  \item $\hdim {\clsf C} \le 1$.
\end{enumerate}
\end{thm}
\begin{proof}
We will show the equivalence of the first two items; the equivalence of the second and third items falls out during the course of the proof.

First suppose that $\clsf C$ is not full.
Then $I_{\clsf C}$ has a minimal generator $x_{u, b}$ for some $u \in [n], b \in [2]$.
If $S/\cani{\clsf C}$ is Cohen-Macaulay, then all minimal generators of $I_{\clsf C}$ must have the same support size, so for each functional monomial $x_{v, 0} x_{v, 1}$, either $x_{v, 0}$ or $x_{v, 1}$ is a minimal generator of $I_{\clsf C}$.
This means that $\clsf C$ is a singleton class, and thus is a tree under $\sim$ trivially.
Conversely, $S/\cani{\{f\}}$ is Cohen-Macaulay for any $f \in [n \to 2]$ because the projective dimension of $S/\cani{\{f\}}$ is $\hdim \{f\} + 1 = 1$ which is the common support size of $I_{\{f\}}$ (\cref{thm:minres_singleton}).

Now assume $\clsf C$ is full.
Then $\mingen(I_{\clsf C}) \spe \FM$ and $\min\{|\supp \bomega|: \bomega \in \mingen(I_{\clsf C})\} = 2$.
Hence $S/\cani{\clsf C}$ is Cohen-Macaulay iff the projective dimension of $S/\cani{\clsf C}$ is 2 iff the homological dimension of ${\clsf C}$ is 1.
This is equivalent to saying that the 1-dimensional cell complex $X$ with vertices $f \in \clsf C$ and edges $f \sim g$ minimally resolves $\cani{\clsf C}$ with the obvious labeling, which is the same as the condition specified in the theorem.
\end{proof}

\begin{cor}
Let $\clsf C \sbe [n \to 2]$.
If $S/\cani{\clsf C}$ is Cohen-Macaulay, then $\clsf C$ has a minimal cellular resolution and has pure Betti numbers which are 0 or 1.
\end{cor}

\begin{exmp} \label{exmp:cDelta+o}
Let $o: [n] \to [2]$ be the identically zero function.
The class $\clsf C := \cDelta_n \cup \{o\}$ satisfies $S/\cani{\clsf C}$ being Cohen-Macaulay.
Indeed, $f \sim_{\clsf C} g$ iff $\{f, g\} = \{\delta_i, o\}$ for some $i$, so $\sim_{\clsf C}$ forms a star graph with $o$ at its center.
For each nonempty $\pf f:\sbe[n] \to [2]$, if $\im \pf f= \{0\}$, then $\clsf C_{\spe \pf f}$ contains $o$ and thus is again a star graph.
If $\pf f(i) = 1$ for a unique $i$, then $\clsf C_{\spe \pf f} = \delta_i$, which is a tree trivially.
Otherwise, $\clsf C_{\spe \pf f} = \emptyset$, which is a tree vacuously.   
\end{exmp}
It seems unlikely that any class $\clsf C$ with Cohen-Macaulay $S/\cani{\clsf C}$ is interesting computationally, as \cref{thm:cohen_macaulay_canonical} and \cref{thm:hdim_ge_vcdim} imply the VC dimension of $\clsf C$ is at most 1.
By the Sauer-Shelah lemma \cite{kearns_introduction_1994}, any such class $\clsf C \sbe [n \to 2]$ has size at most $n+1$.

In contrast, the classes $\clsf C \sbe [n \to 2]$ with Cohen-Macaulay $S/I_{\clsf C}$ form a larger collection, and they all have cellular resolutions.
For this reason, we say {\bf $\clsf C$ is Cohen-Macaulay} if $S/I_{\clsf C}$ is Cohen-Macaulay.
\begin{defn}
Let $\cube_n$ be the $n$-dimensional cube with vertices $[2]^n$.
A {\bf cublex} (pronounced Q-blex) is a subcomplex of $\cube_n$.

$\cube_n$ has a natural PF labeling $\eta =\eta^\cube$ that labels each vertex $V \in [2]^n$ with the corresponding function $\eta(V): [n] \to [2]$ with $\eta(V)(i) = V_i$, and the rest of the PF labels are induced via intersection as in \cref{lemma:pf_label_propagate}.
Specifically, each face $F_{\pf w}$ is associated to a unique PF $\pf w:\sbe [n] \to [2]$, such that $F_{\pf w}$ consists of all vertices $V$ with $\eta(V) \spe \pf w$;
we label such a $F_{\pf w}$ with $\eta(F_{\pf w}) = \pf w$.
A cublex $X$ naturally inherits $\eta^\cube$, which we call the {\bf canonical PF label function} of $X$. 
\end{defn}

Rephrasing Reisner's Criterion \cite[thm 5.53]{miller_combinatorial_2005}, we obtain the following characterization.
\begin{prop}[Reisner's Criterion] \label{prop:reisner_criterion}
$\clsf C \sbe [n \to 2]$ is Cohen-Macaulay iff
$$\betti_{i, \pf f}(\clsf C) = \dim_\kk \rH_{i-1}(\Sbpx{\clsf C \filt \pf f}; \kk) = 0 \text{ for all $i\not=n - |\dom \pf f|$}.$$
\end{prop}
\begin{thm}
Let $\clsf C \sbe [n \to 2]$.
The following are equivalent.
\begin{enumerate}
  \item $\clsf C$ is Cohen-Macaulay. \label{_CM}
  \item $\srdim \clsf C = n$. \label{_CM_iff_SRdim}
  \item $\clsf C = \{\eta^\cube(V): V \in X\}$ for some cublex $X$ such that $X_{\spe \pf f}$ is acyclic for all $\pf f:\sbe[n] \to [2]$. \label{_CM_iff_cubplex} 
\end{enumerate}
\end{thm}
\begin{proof}
($\ref{_CM} \iff \ref{_CM_iff_SRdim}$). This is immediate after noting that $\codim S/I_{\clsf C} = n$.

($\ref{_CM_iff_cubplex} \implies \ref{_CM_iff_SRdim}$). $X$ is obviously a minimal cellular resolution of $\clsf C$, and for each $\pf f$, the face $F_{\pf f}$ with PF label $\pf f$, if it exists, has dimension $n - |\dom \pf f|$, so Reisner's Criterion is satisfied.

($\ref{_CM_iff_SRdim} \implies \ref{_CM_iff_cubplex}$). Let $X$ be the cubplex containing all faces $F_{\pf f}$ such that $\betti_{i, \pf f}(\clsf C) \not=0$ for $i=n-|\dom \pf f|$.
This is indeed a complex: $\rH_{i-1}(\Sbpx{\clsf C \filt \pf f}; \kk) \not = 0$ iff $\Sbpx{\clsf C \filt \pf f}$ is the complete $(i-1)$-dimensional suboplex by \cref{lemma:top_hom_nonzero_iff_complete_suboplex}; hence for any $\pf g \spe \pf f$, $\Sbpx{\clsf C \filt \pf g}$ is the complete $(j-1)$-dimensional suboplex, where $j = n - |\dom \pf g|$, implying that $\betti_{j, \pf g}(\clsf C) = 1$.

We prove by induction on poset structure of $\pf f:\sbe[n] \to [2]$ under containment that $X_{\spe \pf f}$ is acyclic for all $\pf f$.
The base case of $\pf f$ being total is clear.
Suppose our claim is true for all $\pf g \supset \pf f$.
If $X_{\spe \pf f}$ is an $(n-|\dom \pf f|)$-dimensional cube, then we are done.
Otherwise, 
$$X_{\spe \pf f} = \bigcup_{\substack{\pf g \supset \pf f\\|\dom \pf g| = |\dom \pf f| + 1}} X_{\spe \pf g}.$$
By induction hypothesis, each of $X_{\spe \pf g}$ is acyclic, so the homology of $X_{\spe \pf f}$ is isomorphic to the homology of the nerve $\mathcal N$ of $\{X_{\spe \pf g}\}$.
We have for any collection $\mathcal F$ of such $\pf g$,
$$\bigcap_{\pf g \in \mathcal F} X_{\spe \pf g} \not = \emptyset \iff \exists f \in \clsf C\ \forall \pf g \in \mathcal F [f \spe \pf g].$$
Therefore $\mathcal N$ is isomorphic to $\Sbpx{\clsf C \filt \pf f}$ as simplicial complexes.
As $\rH_\bullet(\Sbpx{\clsf C \filt \pf f};\kk) = 0$ (since $X_{\pf f}$ is empty), $X_{\spe \pf f}$ is acyclic as well.

$X$ is obviously minimal since it has unique PF labels, and its vertex labels are exactly $\clsf C$.
\end{proof}

The minimal cublex cellular resolution of Cohen-Macaulay $\clsf C$ constructed in the proof above is called the {\bf canonical cublex resolution} of $\clsf C$.
\begin{cor} \label{cor:CM_implies_pure_Betti}
If $\clsf C \sbe [n \to 2]$ is Cohen-Macaulay, then $\clsf C$ has a minimal cellular resolution and has pure Betti numbers which are 0 or 1.
\end{cor}

It should be easy to see that if $\clsf C$ is Cohen-Macaulay, then so is the filtered class $\clsf C \filt \pf f$ for any PF $\pf f: \sbe [n] \to [2]$.
It turns out this is also true for restrictions of $\clsf C$.
\begin{cor}[Cohen-Macaulayness is preserved under restriction] \label{cor:CM_under_restriction}
If $\clsf C \sbe [n \to 2]$ is Cohen-Macaulay, then so is $\clsf C \restrict U$ for any $U \sbe [n].$
Its canonical cublex resolution is the projection of the canonical cublex resolution of $\clsf C$ onto the subcube $F_{\pf w}$ of $\cube_n$, where $\pf w:\sbe [n] \to [2]$ takes everything outside $U$ to 0.
Consequently, $\betti_{\bullet, \pf f}(\clsf C \restrict U) = 0$ iff $\betti_{\bullet, \pf f'}(\clsf C) = 0$ for all $\pf f' \spe \pf f$ extending $\pf f$ to all of $[n] \setminus U$. 
\end{cor}
\begin{proof}
It suffices to consider the case $U = [n-1]$ and then apply induction.
Fix $\pf f:\sbe[n-1] \to [2]$, and let $\pf f_0:= \pf f \cup (n-1 \mapsto 0), \pf f_1 = \pf f\cup (n-1 \mapsto 1)$.
We wish to show $\betti_{i, \pf f}(\clsf C \restrict U) = 0$ for all $i \not = n-1-|\dom \pf f|$. 
We have three cases to consider.
\begin{enumerate}
  \item $\betti_{\bullet, \pf f_0}(\clsf C) = \betti_{\bullet, \pf f_1}(\clsf C) = 0$.
  Certainly, $\betti_{\bullet, \pf f}(\clsf C)$ would also have to be 0 (the existence of the subcube $F_{\pf f}$ would imply the existence of $F_{\pf f_0}$ and $F_{\pf f_1}$ in the canonical cublex resolution of $\clsf C$).
  By \cref{thm:restriction_betti_exact_seq}, this implies $\betti_{\bullet, \pf f}(\clsf C \restrict U) = 0$ as well.
  \item WLOG $\betti_{i, \pf f_0}(\clsf C) = \ind(i = n - |\dom \pf f| - 1)$ and $\betti_{\bullet, \pf f_1}(\clsf C) = 0$.
  Again, $\betti_{\bullet, \pf f}(\clsf C) = 0$ for the same reason.
  So \cref{thm:restriction_betti_exact_seq} implies $\betti_{i, \pf f}(\clsf C \restrict U) = \ind(i = n - |\dom \pf f| - 1)$.
  \item $\betti_{i, \pf f_0}(\clsf C) = \betti_{i, \pf f_1}(\clsf C) = \ind(i = n - |\dom \pf f| - 1)$.
  Then $\clsf C = [n \to 2]$ and therefore  $\betti_{i, \pf f}(\clsf C) = \ind(i = n - |\dom \pf f|)$.
  \cref{thm:restriction_betti_exact_seq} yields an exact sequence
  $$0 \to \kk^{\betti_{j+1, \pf f}(\clsf C \restrict U)} \to \kk^{\betti_{j+1, \pf f}(\clsf C)} \to \kk^{\betti_{j, \pf f_0}(\clsf C) + \betti_{j, \pf f_1}(\clsf C)} \to \kk^{\betti_{j, \pf f}(\clsf C \restrict U)} \to 0,$$
  where $j = n - |\dom \pf f| - 1$.
  Because $\clsf C$ has pure Betti numbers by \cref{cor:CM_implies_pure_Betti}, the only solution to the above sequence is $\betti_{i, \pf f}(\clsf C \restrict U) = \ind(i = n - |\dom \pf f| - 1)$.
\end{enumerate}
This shows by \cref{prop:reisner_criterion} that $\clsf C \restrict U$ is Cohen-Macaulay.
The second and third statements then follow immediately.
\end{proof}

\begin{lemma}
If $\clsf C \sbe [n \to 2]$ is Cohen-Macaulay, then $\betti_{i, \pf f}(\clsf C) = \ind(i = n - |\dom \pf f|)$ iff $f \in \clsf C$ for all total $f$ extending $\pf f$.
\end{lemma}
\begin{proof}
$\betti_{i, \pf f}(\clsf C) = \ind(i = n - |\dom \pf f|)$ iff $\link_{\pf f}(\Sbpx C)$ is the complete suboplex iff $f \in \clsf C$ for all total $f$ extending $\pf f$.
\end{proof}
\begin{cor} \label{cor:CM_implies_hdim=vcdim}
If $\clsf C \sbe [n \to 2]$ is Cohen-Macaulay, then $\hdim \clsf C = \vcdim \clsf C$.
\end{cor}
\begin{proof}
$\hdim \clsf C$ is the dimension of the largest cube in the canonical cublex resolution of $\clsf C$, which by the above lemma implies $\clsf C$ shatters a set of size $\hdim \clsf C$.
Therefore $\hdim \clsf C \le \vcdim \clsf C$.
Equality then follows from \cref{thm:hdim_ge_vcdim}.
\end{proof}
\begin{exmp}
The singleton class $\{f\}$, $\cDelta \cup \{o\}$ as defined in \cref{exmp:cDelta+o}, and the complete class $[n \to 2]$ are all Cohen-Macaulay.
However, inspecting \cref{tab:dims} shows that, for $d \ge 1$, none of $\cDelta$, $\monconj$, $\CLconj$, $\linthr$, or $\linfun$ on $d$-bit inputs are Cohen-Macaulay, as their Stanley-Reisner dimensions are strictly greater than $2^d$.
Likewise, $\polythr_d^k$ is not Cohen-Macaulay unless $k = d$.
Consequently, the converse of \cref{cor:CM_implies_hdim=vcdim} cannot be true.
\end{exmp}
\newcommand{\cNeigh}{\textsc{nb}}
\begin{exmp}
We can generalize $\cDelta \cup \{o\}$ as follows.
Let $\cNeigh(f)^k_n$ be the class of functions on $[n]$ that differs from $f \in [n \to 2]$ on at most $k$ inputs.
Then $\cNeigh(f)^k_n$ is Cohen-Macaulay; its canonical cublex resolution is the cublex with top cells all the $k$-dimensional cubes incident on $f$.
For example, $\cDelta \cup \{o\} = \cNeigh(o)^1_n$.
\end{exmp}

Finally, we briefly mention the concept of sequential Cohen-Macaulayness, a generalization of Cohen-Macaulayness.
\begin{defn}[\cite{stanley_combinatorics_1996}]
A module $M$ is {\it sequential Cohen-Macaulay} if there exists a finite filtration
$$0 = M_0 \sbe M_1 \sbe \cdots \sbe M_r = M$$
of $M$ be graded submodules $M_i$ such that
\begin{enumerate}
  \item Each quotient $M_i/M_{I-1}$ is Cohen-Macaulay, and
  \item $\dim(M_1/M_0) < \dim(M_2/M_1) < \cdots < \dim(M_r/M_{r-1})$, where $\dim$ denotes Krull dimension.
\end{enumerate}
\end{defn}
Sequentially Cohen-Macaulay rings $S/I$ satisfy $\projdim S/I = \max\{|\supp\es a|: \xx^{\es a} \in \mingen(I^\star)\}$ by a result of \cite{faridi_projective_2013}.
If $S/I_{\clsf C}$ is sequentially Cohen-Macaulay, this means it is actually Cohen-Macaulay, since all minimal generators of $\cani{\clsf C}$ have the same total degree.
Thus what can be called ``sequentially Cohen-Macaulay'' classes coincide with Cohen-Macaulay classes.

\subsection{Separation of Classes}
\label{sec:separation}

In this section, unless specificed otherwise, all homologies and cohomologies are taken against $\kk$.
Suppose $\clsf C, \clsf D \sbe [n \to m]$.
If $\clsf C \sbe \clsf D$, then $\clsf C \cup \clsf D = \clsf D$, and $\cani{\clsf C} + \cani{\clsf D} = \cani{\clsf C \cup \clsf D} = \cani{\clsf D}$.
In particular, it must be the case that for every $i$ and $\sigma$,
$$\betti_{i, \sigma}(\cani{\clsf C} + \cani{\clsf D}) = \betti_{i, \sigma}(\cani{\clsf C \cup \clsf D}) = \betti_{i, \sigma}(\cani{\clsf D}).$$
Thus $\clsf C \subset \clsf D$ if for some $i$ and $\pf f$, $\betti_{i, \pf f}(\clsf C) \not = \betti_{i, \pf f}(\clsf C \cup \clsf D)$.
The converse is true too, just by virtue of $\betti_{0, -}$ encoding the elements of each class.
By \cref{thm:dim1_betti}, $\clsf C \subset \clsf D$ already implies that $\betti_{1, -}$ must differ between the two classes.
However, we may not expect higher dimensional Betti numbers to certify strict inclusion in general, as the examples in \cref{sec:abnormal} show.

This algebraic perspective ties into the topological perspective discussed in the introduction as follows.
Consider $\clsf C \sbe [2^d \to \{-1, 1\}]$ and a PF $\pf f :\sbe [2^d] \to \{-1, 1\}$.
By Hochster's dual formula (\cref{prop:dual_hochster}), $\betti_{i, \pf f}(\clsf C) = \dim_\kk \rH_{i-1}(\Sbpx{\clsf C \filt \pf f}; \kk) = \dim_\kk \rH_{i-1}(\link_{\graph \pf f}\Sbpx{\clsf C}; \kk)$.
When $\pf f = \emptyfun$, this quantity is the ``number of holes of dimension $i-1$'' in the canonical suboplex of $\clsf C$.
When $\graph \pf f = \{(u, \pf f(u))\}$ has a singleton domain, $\link_{\graph \pf f}\Sbpx{\clsf C}$ is the section of $\Sbpx{\clsf C}$ by a hyperplane.
More precisely, if we consider $\Sbpx{\clsf C}$ as embedded the natural way in $S^{2^d-1}_1 = \{z\in \R^{2^d}: \|z\|_1 = 1\}$ (identifying each coordinate with a $v \in [2^d] \cong [2]^d$), $\link_{\graph \pf f}\Sbpx{\clsf C}$ is homeomorphic to $\Sbpx{\clsf C} \cap \{z: z_u = \pf f(u)/2\}$.
\cref{fig:link_as_section} illustrates this.
\begin{figure}
\centering
\includegraphics[width=.4\textwidth]{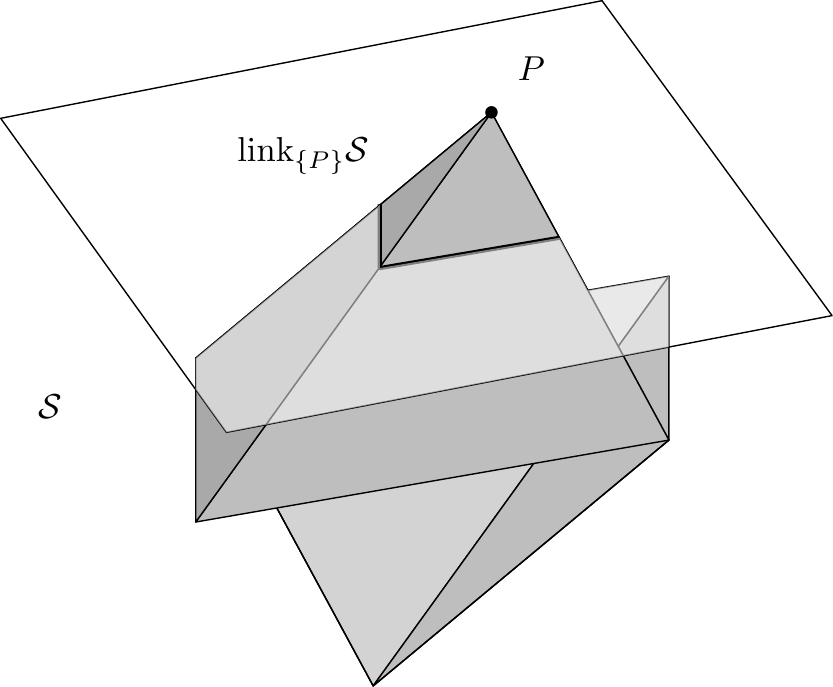}
\caption{The link of suboplex $\Ss$ with respect to vertex $P$ is homeomorphic to the intersection of $\Ss$ with a hyperplane.}
\label{fig:link_as_section}
\end{figure}
For general $\pf f$, we have the homeomorphism
$$\link_{\graph \pf f}\Sbpx{\clsf C} \cong \Sbpx{\clsf C} \cap \{z: z_u = \pf f(u)/2, \forall u \in \dom \pf f\}.$$
Thus comparing the Betti numbers of $\clsf D$ and $\clsf C \cup \clsf D$ is the same as comparing ``the number of holes'' of $\Sbpx{\clsf D}$ and $\Sbpx{\clsf C \cup \clsf D}$ and their corresponding sections.

If PF-labeled complex $(X_C, \mu_C)$ resolves $\clsf C$ and PF-labeled complex $(X_D, \mu_D)$ resolves $\clsf D$, then the join $(X_C \star X_D, \mu_C \star \mu_D)$ resolves $\clsf C \cup \clsf D$ by \cref{prop:union_res}.
The Betti numbers can then be computed by
$$\betti_{i, \pf f}(\clsf C \cup \clsf D) = \dim_\kk \widetilde H_{i-1}((X_C \star X_D)_{\supset \pf f}; \kk)$$
via \cref{prop:betti_from_cellular_resolution}.
Here are some simple examples illustrating this strategy.
\begin{thm}
Let $d \ge 2$.
Let $\indone \in [2^d \to 2]$ be the indicator function $u \mapsto \ind(u = \mathbf 1 = 1\cdots1 \in [2]^d)$.
Consider the partial linear functional $\pf g: \mathbf 0 \to 0, \mathbf 1 \to 1$.
Then $\betti_{i, \pf g}(\linfun^2_d \cup \{\indone\}) = 0$ for all $i$.
\label{exmp:ind_not_in_linfun}
\end{thm}
The proof below is in essence the same as the proof of $\indone \not \in \linfun^2_d$ given in the introduction, but uses the theory we have developed so far.
The application of the Nerve Lemma there is here absorbed into the Stanley-Reisner and cellular resolution machineries.
\begin{proof}
Let $(X, \mu)$ be the flag resolution of $\linfun^2_d$ and $\bullet$ be the one point resolution of $\{\indone\}$.
Then $X \star \bullet$ is the cone over $X$, with labels $\mu'(F \star \bullet) = \mu(F) \cap \indone$ and $\mu'(F) = \mu(F)$ for cells $F$ in $X$.

Consider $Z := (X \star \bullet)_{\supset \pf g}$.
Every cell $F$ of $X$ in $Z$ has PF label a linear functional on a linear subspace of $\Fld_2^d$ strictly containing $\mathcal V := \{\mathbf 0, \mathbf 1\}$.
As such, $\mu(F) \cap \indone$ strictly extends $\pf g$, because $\mu(F)$ sends something to 0 outside of $\mathcal V$.
This means $Z$ is a cone over $X_{\supset \pf g}$, and thus is acyclic.
Therefore $\betti_{i, \pf g}(\linfun^2_d \cup \{\indone\}) = 0$ for all $i$. 
\end{proof}

But $\betti_{d-1, \pf g}(\linfun^2_d)$ is nonzero, so we obtain the following corollary.
\begin{cor}
$\indone \not \in \linfun^2_d$ for $d \ge 2$.
\end{cor}
\cref{exmp:ind_not_in_linfun} says the following geometrically: the canonical suboplex of $\linfun_d^2 \filt \pf g$ (a complex of dimension $2^{2^d} - 2$) has holes in dimension $d-1$, but these holes are simultaneously covered up when we add $\indone$ to $\linfun_d^2$.

\begin{thm} \label{exmp:parity_polythr}
Let $\parity_d$ be the parity function on $d$ bits.
Then $\betti_{i, \emptyfun}(\polythr_d^k \cup \{\parity_d\}) = 0$ for all $i$ if $k < d$.
\end{thm}
Let us work over $\{-1, 1\}$ instead of $\{0, 1\}$, under the bijection $\{0, 1\} \cong \{1, -1\}, a \mapsto (-1)^a$, so that $\parity_d(u_0, \ldots, u_{d-1}) = u_0\cdots u_{d-1}$ for $u \in \{-1, 1\}^d$ and $\polythr_d^k$ consists of $\sgn(p)$ for polynomials $p$ with degree at most $k$ not taking 0 on any point in $\bcube^d$.
\begin{proof}
Fix $k < d$.
Let $(X, \mu)$ denote the ball resolution of $\polythr_d^k$ and $\bullet$ be the one point resolution of $\{f\}$.
Then $X \star \bullet$ is the cone over $X$, with labels $\mu'(F \star \bullet) = \mu(F) \cap f$ and $\mu'(F) = \mu(F)$ for cells $F$ in $X$.

Consider $Z := (X \star \bullet)_{\supset \emptyfun}$.
Every PF label $\pf f :\sbe \bcube^d \to \bcube$ of $X$ intersects $\parity_d$ nontrivially if $\pf f \not = \emptyfun$.
Otherwise, suppose $p$ is a polynomial function such that $p(u) > 0 \iff \pf f(u) = 1,\ p(u) < 0 \iff \pf f(u) = -1,$ and $p(u) = 0 \iff \pf u \not\in \dom \pf f$.
Then by discrete Fourier transform \footnote{See the opening chapter of \cite{odonnell_analysis_2014} for a good introduction to the concepts of Fourier analysis of boolean functions.}, the coeffient of $p$ for the monomial $\parity_d(u) = \prod_{i=0}^{d-1}u_i$ is
$$\sum_{a \in \{-1, 1\}^d} p(a) \parity_d(a) < 0$$
because whenever $p(a)$ is nonzero, its sign is the opposite of $\parity_d(a)$.
This contradicts $k < d$.
Thus in particular, the PF label of every cell of $X$ except for the top cell (with PF label $\emptyfun$) intersects $\parity_d$ nontrivially.
Therefore $Z$ is a cone and thus $\betti_{\bullet, \emptyfun}(\polythr_d^k \cup \{\parity_d\}) = 0$.
\end{proof}

But $\betti_{e, \emptyfun}(\polythr_d^k) = 1$, where $e = \sum_{j=0}^k \binom d j$ is the homological dimension of $\polythr_d^k$.
So we recover the following result by Minsky and Papert.
\begin{cor}[\cite{minsky_perceptrons:_1969}]
$\parity_d \not \in \polythr_d^k$ unless $k = d$.
\end{cor}

From the analysis below, we will see in fact that adding $\parity_d$ to $\polythr_d^k$ causes changes to Betti numbers in every dimension up to $\vcdim \polythr_d^k = \hdim \polythr_d^k$, so in some sense $\parity_d$ is {\it maximally homologically separated} from $\polythr_d^k$.
This ``maximality'' turns out to be equivalent to the lack of weak representation \cref{cor:weak_rep_betti}.

%

\newcommand{\Xcapf}[1]{{#1}^{\cartcap f}}
By \cref{lemma:product_join_exact_seq}, the ``differences'' between the Betti numbers of $\clsf C \cup \clsf D$ and those of $\clsf C$ and of $\clsf D$ are given by the homologies of $(X_C \times X_D, \mu_C \times \mu_D)_{\supset \pf f}$.  
Suppose $\clsf C$ consists of a single function $f$.
Then $X_C$ is a single point with exponent label $\Gamma f$.
$(X_C \times X_D, \mu_C \times \mu_D)$ is thus isomorphic to $X_D$ as complexes, but the exponent label of each nonempty cell $F \in X_C \times X_D$ isomorphic to cell $F' \in X_D$ is now $\lcm(\lambda_D(F'), \Gamma f)$, and the PF label of $F$ is $\mu_D(F') \cap f$; the empty cell $\emptycell \in X_C \times X_D$ has the exponent label $\mathbf 0$.
We denote this labeled complex by $\Xcapf{(X_D)}$.

Notice that $\Xcapf{(X_D)}$ is a (generally nonminimal) cellular resolution of the PF class $\Xcapf{\clsf D} := \clsf D \cartcap \{f\}$, because $\Xcapf{(X_D)}_{\spe \pf f} = (X_D)_{\spe \pf f}$ whenever $\pf f \sbe f$ and empty otherwise, and therefore acyclic.
So the (dimensions of) homologies of $(X_C \times X_D)_{\supset \pf f}$ are just the Betti numbers of $\Xcapf{\clsf D}$.
This is confirmed by \cref{lemma:cartcap_is_diff}.
Another perspective is that $\Sbpx{\clsf C \cartcap \clsf D}$ is the intersection $\Sbpx{\clsf C} \cap \Sbpx{\clsf D}$, so by Mayer-Vietoris, $\cani{\clsf C \cartcap \clsf D}$ gives the ``difference'' in Betti numbers between $\betti_{\bullet, -}(\clsf C) + \betti_{\bullet, -}(\clsf D)$ and $\betti_{\bullet, -}(\clsf C \cup \clsf D)$.

$\cani{\Xcapf{\clsf D}}$ determines the membership of $f$ through several equivalent algebraic conditions.
\begin{lemma} \label{lemma:xcapf_determines_membership}
Let $\clsf D \sbe [n \to m]$ be a full class (see \cref{defn:full}).
Then the following are equivalent:
\begin{enumerate}
  \item $f \in \clsf D$ \label{_f_in_D}
  \item $\cani{\Xcapf{\clsf D}}$ is principally generated by $\xx^{\Gamma f}$ \label{_principal_gen}
  \item $\cani{\Xcapf{\clsf D}}$ is principal \label{_principal}
  \item $\betti_{i, \pf f}(\Xcapf{\clsf D}) = 1$ for exactly one partial $\pf f$ when $i = 0$ and equals 0 for all other $i$. \label{_betti_0}
  \item $\betti_{i, \pf f}({\Xcapf{\clsf D}}) = 0$ for all $\pf f$ and all $i \ge 1$. \label{_betti_1}
  \item $\betti_{i, \pf f}({\Xcapf{\clsf D}}) = 0$ for all $\pf f \not= f$ and all $i \ge 1$. \label{_betti_notf}
\end{enumerate}
\end{lemma}
\begin{proof}
(\ref{_f_in_D} $\implies$ \ref{_principal_gen} $\implies$ \ref{_principal})
If $f \in \clsf D$, then $\cani{\Xcapf{\clsf D}}$ is principally generated by $\xx^{\Gamma f}$.

(\ref{_principal} $\implies$ \ref{_principal_gen} $\implies$ \ref{_f_in_D})
If $\cani{\Xcapf{\clsf D}}$ is principal, then it's generated by $\xx^{\Gamma \pf g}$ for some partial function $\pf g$.
This implies that $h \cap f \sbe \pf g \implies \graph h \sbe \Gamma f \cup \graph \pf g, \forall h \in \clsf D$.
But taking the union over all $h \in \clsf D$ contradicts our assumption on $\clsf D$ unless $\pf g = f$.
Thus there is some $h \in \clsf D$ with $h \cap f = f \implies h = f$.

(\ref{_principal}) $\iff$ \ref{_betti_0})
This should be obvious.

(\ref{_betti_0} $\iff$ \ref{_betti_1} $\iff$ \ref{_betti_notf})
The forward directions are obvious.
Conversely, if $\cani{\Xcapf{\clsf D}}$ has more than one minimal generator, then its first syzygy is nonzero and has degrees $\succ \Gamma f$, implying the negation of \cref{_betti_1} and \cref{_betti_notf}.
\end{proof}

Thus $\cani{\Xcapf{\clsf D}}$ by itself already determines membership of $f \in \clsf D$.
It also yields information on the Betti numbers of $\clsf C \cup \clsf D$ via \cref{lemma:cartcap_is_diff}.
Thus in what follows, we study $\cani{\Xcapf{\clsf D}}$ in order to gain insight into both of the membership question and the Betti number question.

Let us consider the specific case of $\clsf D = \linthr_U$, with minimal cocellular resolution $\coballres_U = (Y, \mu)$.
Then $\Xcapf{\linthr_U}$ has minimal cocellular resolution $(Y, \Xcapf\mu)$, where we relabel cells $F$ of $Y$ by $\Xcapf \mu(F) = \mu(F) \cap f$, so that, for example, the empty cell still has PF label the empty function.
Choose $U$ to be a set of $n$ points such that the vectorization $\vec U$ forms a set of orthogonal basis for $\R^n$.
Then $\linthr_U = [U \to 2]$, and $Y$ is homeomorphic to the unit sphere $S^{n-1}$ as a topological space and is isomorphic to the complete $(n-1)$-dimensional suboplex as a simplicial complex.
It has $2^{n}$ top cells $\splx_g$, one for each function $g \in [U \to 2]$; in general, it has a cell $\splx_{\pf f}$ for each PF $\pf f: \sbe U \to 2$, satisfying $\splx_{\pf f} = \bigcap_{f \spe \pf f} \splx_f$.
\newcommand{\flip}{\operatorname\Diamond}
\newcommand{\bdy}{\pbd_{\pf f} \splx_{f\flip \pf f}}
\newcommand{\splxf}{\splx_{f \flip \pf f}}

Let us verify that $\betti_{i, \pf f}(\Xcapf \linthr_U)$ equals $\betti_{i, \Gamma \pf f}(\la \xx^{\Gamma f}\ra) = \ind(i=0 \And \pf f = f)$ by \cref{lemma:betti_from_cocell}.
For any $\pf f \sbe f$, define $f \flip \pf f$ to be the total function
$$f \flip \pf f: u \mapsto f(u)\ \forall u \in \dom \pf f,\quad
u \mapsto \neg f(u)\ \forall u \not \in \dom \pf f.$$
Define the {\bf $(f, \pf f)$-star $\bigstar(f, \pf f)$} to be the collection of open cells $\ocell\splx_{\pf g}$ with PF label $\pf f \sbe \pf g \sbe f \flip \pf f$.
This is exactly the collection of open cells realized by the cellular pair $(\splxf, \bdy)$, where $\bdy$ denotes the partial boundary of $\splxf$ that is the union of the closed cells with PF labels $(f \flip \pf f) \setminus (i \mapsto \pf f(i))$ for each $i \in \dom \pf f$.
In particular, $\bigstar(f, f)$ is realized by $(\splx_f, \pd \splx_f)$, and $\bigstar(f, \emptyfun)$ is realized by $(\splx_{\neg f}, \{\})$ (where $\{\}$ is the void complex).
\renewcommand{\bdy}{\pbd_{\pf f} \splx}
\renewcommand{\splxf}{\splx}
In the following we suppress the subscript to write $(\splxf, \bdy)$ for the sake of clarity.
When $\pf f \not= \emptyfun, f$, $\bdy$ is the union of faces intersecting $\splx_{\neg f}$; intuitively, they form the subcomplex of faces directly visible from an observer in the interior of $\splx_{\neg f}$.
This is illustrated in \cref{fig:ff-star}. 

\begin{figure}
\centering
\includegraphics[width=.4\textwidth]{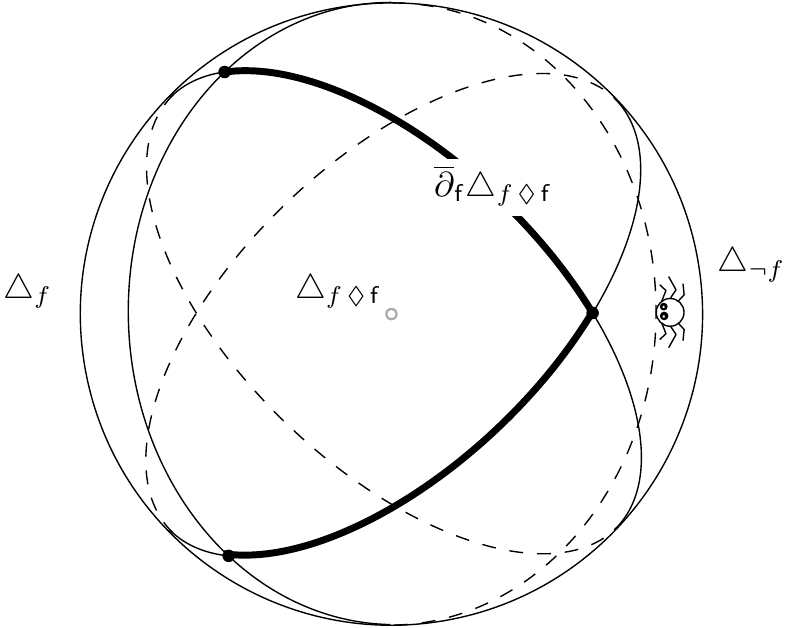}
\caption{The bold segments form the partial boundary $\pbd_{\pf f}\splx_{f \flip \pf f}$.
In particular, this partial boundary contains three vertices.
It is exactly the part of $\splx_{f \flip \pf f}$ visible to a spider in the interior of $\splx_{\neg f}$, if light travels along the sphere.}
\label{fig:ff-star}
\end{figure}
Then the part of $(Y_U, \Xcapf\mu_U)$ with PF label $\pf f$ is exactly the $(f, \pf f)$-star.
If $\pf f \not = f$, the closed top cells in $\bdy$ all intersect at the closed cell with PF label $f \flip \pf f \setminus \pf f = \neg(f \setminus \pf f)$, and thus their union $\bdy$ is contractible.
This implies via the relative cohomology sequence
$$\cdots \gets \rH^j(\bdy) \gets \rH^j(\splxf) \gets H^j(\splxf, \bdy) \gets \rH^{j-1}(\bdy) \gets \cdots$$
that
$0 = \dim_\kk \rH^j(\splxf) = \dim_\kk H^j(\splxf, \bdy) = \betti_{n-1-j, \pf f}(\Xcapf \linthr_U)$.
If $\pf f = f$, then $\bdy = \pd \splxf$, so $H^j(\splxf, \bdy) \cong \rH^j (\splxf/\pd) \cong \kk^{\ind(j = {n-1})}$.
This yields $\betti_{k, f}(\Xcapf \linthr_U) = \ind(k = 0)$.

The analysis of the Betti numbers of any thresholded linear class $\thr L$ is now much easier given the above.
As discussed in \cref{sec:threshold_functions}, the cocellular resolution $(Z, \mu_Z)$ of $\thr L$ is just the intersection of $\coballres_U = (Y, \mu)$ with $L$, with the label of an intersection equal to the label of the original cell, i.e. $Z = Y \cap L, \mu_Z(F \cap L) = \mu(F)$.
Similarly, the cocellular resolution of $\Xcapf{(\thr L)}$ is just $(Z, \Xcapf\mu_Z)$ with $Z = Y \cap L, \Xcapf \mu_Z(F \cap L) = \Xcapf\mu(F)$.
If $L$ is not contained in any coordinate hyperplane of $Y$, then $\thr L$ is full.
By \cref{lemma:xcapf_determines_membership}, $f \in \thr L$ iff $\betti_{i, \pf f}(\Xcapf{\thr L}) = 0$ for all $i \ge 1$.
This is equivalent by \cref{lemma:betti_from_cocell} to the statement that for all $\pf f$, the degree $\Gamma \pf f$ part of $(Z, \Xcapf \mu_Z)$, $\bigstar_L(f, \pf f) := \bigstar(f, \pf f) \cap L$, has the homological constraint
$$H^{\dim Z - i}(\overline{\bigstar_L(f, \pf f)}, \pbd \bigstar_L(f, \pf f)) = H^{\dim Z - i}(\splx_{f \flip \pf f} \cap L, \bdy_{f \flip \pf f}\cap L) = 0, \forall i \ge 0.$$
But of course, $f \in \thr L$ iff $L \cap \ocell \splx_f \not=\emptyset$.
We therefore have discovered half of a remarkable theorem.
\begin{thm}[Homological Farkas]\label{thm:hom_farkas}
Let $L$ be a vector subspace of dimension $l \ge 2$ of $\R^n$ not contained in any coordinate hyperplane, let $\mathbb P$ denote the positive cone $\{v \in \R^n: v > 0\}$, and let $\mathbf 1: [n] \to \{-1, 1\}, j \mapsto 1$.
For any $g: [n] \to \{-1, 1\}$, define $\Xi(g)$ to be the topological space represented by the complex $\pbd\bigstar_L(\mathbf 1, \mathbf 1 \cap g)$.
Then the following are equivalent:\footnote{Our proof will work for all fields $\kk$ of any characteristic, so the cohomologies can actually be taken against $\Z$.}
\begin{enumerate}
  \item $L$ intersects $\mathbb P$. \label{_HF_L_cap_P}
  \item For all $g\not=\mathbf 1,\neg\mathbf 1$, $\rH^{\bullet}(\Xi(g);\kk) = 0$ as long as $\splx_g \cap L \not=\emptyset$.\label{_HF_i_ge_0}
\end{enumerate}
\end{thm}

This theorem gives {\it homological certificates} for the non-intersection of a vector subspace with the positive cone, similar to how Farkas' lemma \cite{ziegler_lectures_1995} gives {\it linear certificates} for the same thing.
Let's give some intuition for why it should be true.
As mentioned before, $\pbd \bigstar(\mathbf 1, \mathbf 1 \cap g)$ is essentially the part of $\splx_{\mathbf 1 \flip (\mathbf 1 \cap g)} = \splx_g$ visible to an observer Tom in $\ocell\splx_{\neg \mathbf 1}$, if we make light travel along the surface of the sphere, or say we project everything into an affine hyperplane.
Since the simplex is convex, the image Tom sees is also convex.
If $L$ indeed intersects $\ocell\splx_{\mathbf 1}$ (equivalently $\ocell \splx_{\neg \mathbf 1}$), then for $\Xi(g)$ he sees some affine space intersecting a convex body, and hence a convex body in itself.
As Tom stands in the interior, he sees everything (i.e. his vision is bijective with the actual points), and the obvious contraction he sees will indeed contract $\Xi(g)$ to a point, and $\Xi(g)$ has trivial cohomology.

Conversely, this theorem says that if Tom is outside of $\ocell\splx_{\mathbf 1}$ (equivalently $\ocell \splx_{\neg \mathbf 1}$), then he will be able to see the nonconvexity of $\pbd \bigstar(\mathbf 1, \mathbf 1 \cap g)$ for some $g$, such that its intersection with an affine space is no longer contractible to a single point.

\begin{proof}[Proof of $\ref{_HF_L_cap_P} \implies \ref{_HF_i_ge_0}$]
Note that $\Xi(g)$ is a complex of dimension at most $l-2$, so it suffices to prove the following equivalent statement:
\begin{equation*}
\parbox{.8\textwidth}{For all $g\not=\mathbf 1,\neg\mathbf 1$, $\rH^{l -2- i}(\Xi(g);\kk) = 0$ for all $i \ge 0$ as long as $\splx_g \cap L \not=\emptyset$.}
\end{equation*}
$L$ intersects $\mathbb P$ iff $L$ intersects $\ocell\splx_{\mathbf 1}$ iff $\mathbf 1 \in\thr L$.
By \cref{lemma:xcapf_determines_membership}, this implies $\betti_{i, \pf f}(\Xcapf{\thr L}) = 0$ for all $\pf f \not = \mathbf 1, \emptyfun$ and $i \ge 0$.
As we observed above, this means
$$H^{l - 1 - i}(\overline{\bigstar_L(\mathbf 1, \pf f)}, \pbd \bigstar_L(\mathbf 1, \pf f)) = 0, \forall i \ge 0.$$
Write $A = \overline{\bigstar_L(\mathbf 1, \pf f)}$ and $B = \pbd \bigstar_L(\mathbf 1, \pf f)$ for the sake of brevity.
Suppose $A = \splx_{\mathbf 1 \flip \pf f} \cap L$ is nonempty.
Then for $\pf f \not = \emptyfun$ as we have assumed, both $A$ and $B$ contain the empty cell, and therefore we have a relative cohomology long exact sequence with reduced cohomologies,
\begin{align*}
\cdots \gets H^{l-1-i}(A, B) \gets \rH^{l-2-i}(B) \gets \rH^{l-2-i}(A) \gets H^{l-2-i}(A, B) \gets \cdots 
\end{align*}
Because $\rH^\bullet(A) = 0$, we have
$$ H^{l - 1 - i}(A, B) \cong \rH^{l - 2 - i}(B), \forall i.$$
This yields the desired result after observing that $\mathbf 1 \flip \pf f \not = \mathbf 1, \neg \mathbf 1$ iff $\pf f \not = \mathbf 1, \emptyfun$. 
\end{proof}
Note that we cannot replace $L$ with any general openly convex cone, because we have used \cref{lemma:xcapf_determines_membership} crucially, which requires $\thr L$ to be full, which can happen only if $L$ is a vector subspace, by \cref{lemma:openly_convex_dichotomy}.

The reverse direction is actually quite similar, using the equivalences of \cref{lemma:xcapf_determines_membership}.
But straightforwardly applying the lemma would yield a condition on when $g = \neg \mathbf 1$ as well which boils down to $L \cap \splx_{\neg \mathbf 1} \not=\emptyset$, that significantly weakens the strength of the theorem.\footnote{Note that the condition says $L$ intersects the closed cell $\splx_{\neg \mathbf 1}$, not necessarily the interior, so it doesn't completely trivialize it.}
To get rid of this condition, we need to dig deeper into the structures of Betti numbers of $\thr L$.

\begin{thm} \label{thm:thr_top_betti}
Suppose $L$ is linear of dimension $l$, $\thr L$ is a full class, and $f \not \in \thr L$.
Let $\pf g$ be such that $\sigma \pf g$ is the unique covector of $L$ of the largest support with $\pf g \sbe f$ (where we let $\pf g = \mathbf 0$ if no such covector exists).
We say $\pf g$ is the {\bf projection} of $f$ to $\thr L$, and write $\pf g = \Pi(f, L)$. 
Then the following hold: 
\begin{enumerate}
  \item $\betti_{i, \pf g}(\Xcapf{\thr L}) = \ind(i = l - 1 - \rank \sigma \pf g)$. (Here $\rank$ denotes the rank wrt matroid of $L$ as defined in \cref{sec:threshold_functions})\label{_top_betti_thrL}
  \item For any $\pf h \not \spe \pf g$, $\betti_{\bullet, \pf h}(\Xcapf{\thr L}) = 0$. \label{_top_betti_deg_is_least} 
  \item For any PF $\pf r$ with domain disjoint from $\dom \pf g$, $\betti_{i, \pf r \cup \pf g}(\Xcapf{\thr L}) = \betti_{i, \pf r}(\Xcapf{\thr L} \restrict ([n] \setminus \dom \pf g))$. \label{_betti_from_restrict}
\end{enumerate}
\end{thm}
Note that such a $\sigma \pf g$ would indeed be unique, since any two covectors with this property are consistent and thus their union gives a covector with weakly bigger support.
\begin{proof}
(\cref{_top_betti_thrL})
The assumption on $\pf g$ is exactly that $L$ intersects $\splx_f$ at $\ocell\splx_{\pf g} \sbe \splx_f$ and $\pf g$ is the maximal such PF.
Then $\bigstar(f, \pf g) \cap L = \bigcup_{f\flip \pf g \spe \pf h \spe \pf g} \ocell \splx_{\pf h} \cap L = \ocell \splx_{\pf g} \cap L$.
Therefore $\betti_{i, \pf g}(\Xcapf{\thr L}) = \rH^{l - 1 - i}((\splx_{\pf g} \cap L)/ \pd) = \ind(l - 1 - i = \dim (\splx_{\pf g} \cap L))$ (note that when $\splx_g \cap L$ is a point (resp. the empty cell), the boundary is the empty cell (resp. the empty space), so that this equality still holds in those cases).
But $\dim (\splx_{\pf g} \cap L)$ is $\rank \sigma \pf g$.
So the Betti number is $\ind(i = l - 1 - \rank \sigma \pf g)$.

(\cref{_top_betti_deg_is_least})
We show that $\cani{\Xcapf{\thr L}}$ is generated by monomials of the form $\xx^{\Gamma \pf f}$ for $\pf f \spe \pf g$.
It suffices to demonstrate that for any function $h \in \thr L$, the function $h \rtimes \pf g$ defined by
$$h\rtimes \pf g(u) := \begin{cases} 
\pf g(u) & \text{if $u \in \dom \pf g$}\\
h(u) & \text{otherwise.}
\end{cases}$$
is also in $\thr L$, as $f \cap (h \rtimes \pf g) \spe f \cap h$.

Let $\varphi \in L$ be a function $\varphi: U \to \R$ such that $\sgn(\varphi) = \sigma \pf g$.
If $\psi \in L$ is any function, then for sufficiently small $\epsilon > 0$, $\sgn(\epsilon \psi + \varphi) = \sgn(\psi) \rtimes \sgn(\varphi) = \sgn(\psi) \rtimes \pf g$.
Since $L$ is linear, $\epsilon \psi + \varphi \in L$, and we have the desired result.

(\cref{_betti_from_restrict})
As shown above, the minimal generators of $\cani{\Xcapf{\thr L}}$ are all divisible by $\xx^{\Gamma \pf g}$.
The result then follows from \cref{lemma:forget_var}.
\end{proof}
\begin{cor} \label{cor:thr_determines_membership}
Suppose $L$ is linear of dimension $l \ge 2$ and $\thr L$ is a full class.
Then $f \in \thr L$ iff $\betti_{i, \pf f}(\Xcapf{\thr L}) = 0$ for all $\pf f \not= f, \emptyfun$ and all $i \ge 1$.

If $l \ge 1$, then we also have $f \in \thr L$ iff $\betti_{i, \pf f}(\Xcapf{\thr L}) = 0$ for all $\pf f \not= f, \emptyfun$ and all $i \ge 0$. 
\end{cor}
\begin{proof}
We show the first statement.
The second statement is similar.

The forward direction follows from \cref{lemma:xcapf_determines_membership}.
If $\betti_{\bullet, \emptyfun}(\Xcapf{\thr L}) = 0$, then the same lemma also proves the backward direction.

So assume otherwise, and in particular, $f \not \in \thr L$.
By \cref{thm:thr_top_betti}, it has to be the case that $\betti_{i, \emptyfun}(\Xcapf{\thr L}) = \ind(i = l-1-(-1)) = \ind(i = l)$ since $\rank\mathbf 0 = -1$.
Consequently, $\betti_{l-1, \pf f}(\Xcapf{\thr L}) \not = 0$ for some $\pf f \supset \emptyfun$.
If $l \ge 2$, then this contradicts the right side of the equivalence, as desired.
\end{proof}

We can now finish the proof of \cref{thm:hom_farkas}.
\begin{proof}[Proof of $\ref{_HF_i_ge_0} \implies \ref{_HF_L_cap_P}$ in \cref{thm:hom_farkas}]

Assume $l \ge 2$.
(\ref{_HF_i_ge_0}) says exactly that $\betti_{j, \pf f}(\Xcapf{\thr L}) = 0$ for all $\pf f \not= \mathbf 1, \emptyfun$ and all $j \ge 1$. 
So by \cref{cor:thr_determines_membership}, $\mathbf 1 \in \thr L$ and therefore $\thr L$ intersects $\mathbb P$.
\end{proof}

From the literature of threshold functions in theoretical computer science, we say a real function $\varphi: U \to \R$ on a finite set $U$ {\bf weakly represents} a function $f: U \to \{-1, 1\}$ if $\varphi(u) > 0 \iff f(u) = 1$ and $\varphi(u) < 0 \iff f(u) = -1$, but we don't care what happens when $\varphi(u) = 0$.
In these terms, we have another immediate corollary of \cref{thm:thr_top_betti}.
\begin{cor}\label{cor:weak_rep_betti}
A function $f$ is weakly representable by $\polythr_d^k$ iff $\betti_{i, \emptyfun}(\Xcapf{(\polythr_d^k)}) = 0, \forall i$. 
\end{cor}
This is confirmed by \cref{exmp:parity_polythr}.
By \cref{lemma:cartcap_is_diff}, this result means that $f$ is weakly representable by $\polythr_d^k$ iff adding $f$ to $\polythr_d^k$ did not change the homology of $\Ss_{\polythr_d^k}$.

\begin{remk}
\cref{_betti_from_restrict} of \cref{thm:thr_top_betti} reduces the characterization of Betti numbers of $\Xcapf{\thr L}$ to the case when $f$ is not ``weakly representable'' by $\thr L$.
\end{remk}

\newcommand{\Xcapfu}[1]{{#1}^{\cartcap f^u}}
The following theorem says that as we perturb a function $f \not \in \thr L$ by a single input $u$ to obtain $f^u$, a nonzero Betti number $\betti_{i, \pf f}$ of ``codimension 1'' of $\Xcapf{\thr L}$ remains a nonzero Betti number of ``codimension 1'' of $\Xcapfu{\thr L}$ if we truncate $\pf f$.
\begin{thm} [Codimension 1 Stability] 
Suppose $L$ is linear of dimension $l \ge 2$ and $\thr L$ is a full class.
Let $f$ be a function not in $\thr L$ and write $\pf g = \Pi(f, L)$.
Assume $\rank \sigma \pf g = s$ (so that $\betti_{l-s-1, \pf g}(\Xcapf{\thr L}) = 1$) and let $\pf f: \sbe[n] \to [2]$ be such that $\betti_{l-s-2, \pf f}(\Xcapf{\thr L}) \not= 0$.
Then $\betti_{l-s-2, \pf f}(\Xcapf{\thr L}) = 1$.

Furthermore, if $|\dom (\pf f \setminus \pf g)| > 1$ and $u \in \dom (\pf f \setminus \pf g)$, set $\pf f' := \pf f \setminus (u \mapsto f(u))$.
Then we have
$$\betti_{l-s-2, \pf f'}(\Xcapfu{\thr L}) = 1$$
where 
$$f^u(v) := \begin{cases} 
f(v) & \text{if $v \not = u$}\\
\neg f(v) & \text{if $v = u$.}
\end{cases}$$
\end{thm}
\begin{proof}
By \cref{thm:thr_top_betti}, it suffices to show this for $\pf g = \emptyfun$; then $s = -1$.

Recall $\bigstar_L(f, \pf f) = \bigstar(f, \pf f) \cap L $.
By \cref{lemma:betti_from_cocell}, $\betti_{l-1, \pf f}(\Xcapf{\thr L}) = \dim_\kk H^{0}(\overline{\bigstar_L(f, \pf f)}, \pbd \bigstar_L(f, \pf f))$.
If $\pbd \bigstar_L(f, \pf f)$ contains more than the empty cell, then the RHS is the zeroth reduced cohomology of the connected space $\bigstar_L(f, \pf f)/ \pbd$, which is 0, a contradiction.
Therefore $\pbd \bigstar_L(f, \pf f) = \{\emptycell\}$, i.e. as geometric realizations, $L$ does not intersect $\pbd \bigstar(f, \pf f)$.
Consequently, $\betti_{l-1, \pf f}(\Xcapf{\thr L}) = \dim_\kk H^0(\splx_{f \flip \pf f}, \{\emptycell\}) = \dim_\kk H^0(\splx_{f \flip \pf f}) = 1$.

Now $f^u \flip \pf f' = f \flip \pf f$, and $\pbd_{\pf f'} \splx_{f \flip \pf f} \subset \bdy_{f \flip \pf f}$ since $\pf f' \subset \pf f$.
Therefore $\pbd \bigstar_L(f^u, \pf f') \sbe \pbd \bigstar_L(f, \pf f)$ also does not intersect $L$.
So $\betti_{l-1, \pf f'}(\Xcapfu{\thr L}) = \dim_\kk H^0(\splx_{f \flip \pf f}, \{\emptycell\}) = \dim_\kk H^0(\splx_{f \flip \pf f}) = 1$.
\end{proof}

Below we give some examples of the computation of Betti numbers of $\thr L^{\cartcap f}$.

\begin{thm}
Let $f := \parity_d \in [\{-1, 1\}^d \to \{-1, 1\}]$ and $\clsf C := \linthr_d \sbe [\{-1, 1\}^d \to \{-1, 1\}]$.
Then $\betti_{i, f \cap \onef}(\clsf C^{\cartcap f}) = (2^{d-1} - 1)\ind(i = 1)$.
\end{thm}
\begin{proof}
Let $(Y, \mu)$ be the coball resolution of $\linthr_d$.
Consider the cell $F_\onef$ of $Y$ with PF label $\onef$.
It has $2^d$ facets since $$\ind_s := \begin{cases}
r \mapsto -1 & \text{if $r = s$}\\
r \mapsto 1 & \text{if $r \not = s$}
\end{cases}$$
is a linear threshold function (it ``cuts'' out a corner of the $d$-cube), so that each facet of $F_\onef$ is the cell $G_s := F_{\onef \cap \ind_s}$ with PF label
$$\onef \cap \ind_s = \begin{cases} r \mapsto 1 &\text{if $r = s$}\\ \text{undefined} & \text{otherwise.}\end{cases}$$

If for $s, s' \in \{-1, 1\}^d$, $f(s) = f(s')$, then $G_s$ and $G_{s'}$ do not share a codimension 2 face (do not share a facet of their own).
(If they do, then
$$\begin{cases}
r \mapsto + & \text{if $r \not= s, s'$}\\
r \mapsto 0 & \text{else}
\end{cases}$$
is a covector of the $d$-cube.
But this means that $(s, s')$ is an edge of the $d$-cube, implying that $f(s) \not= f(s')$).

Now note that $\pbd_{f \cap \onef} F_\onef = \bigcup\{G_s: f(s) = 1\}$.
For each $G_s \not \sbe \pbd_{f \cap \onef} F_\onef$, we have $\pd G_s \sbe \pbd_{f \cap \onef} F_\onef$ by the above reasoning.
Let $\alpha := d + 1 = \hdim \clsf C$ and $n = 2^d$.
Therefore, $\pbd_{f \cap \onef } R_\onef \cong S^{\alpha - 2} \setminus \bigsqcup_{i=1}^{n/2} D^{n-2}$, the $(\alpha-2)$-sphere with $n/2$ holes.
So
\begin{align*}
\rH^{m}(\pbd_{f \cap \onef} F_\onef) = \begin{cases} \Z^{n/2 - 1} & \text{if $m = \alpha - 3$}\\
0 & \text{otherwise}
\end{cases}
\end{align*}
Hence
\begin{align*}
\betti_{i, f \cap \onef}(\clsf C) &= \dim_\kk \rH^{\alpha - 1 - i}(F_\onef, \pbd_{f \cap \onef}F_\onef; \kk) \\
	&= \rank \rH^{\alpha - 2 - i}(\pbd_{f \cap \onef} F_\onef)\\
	&= (n/2 - 1)\ind(\alpha - 2 -i = \alpha - 3)\\
	&= (2^{d-1} - 1) \ind(i = 1).
\end{align*}

\end{proof}

\begin{thm}
Let $f := \parity_d \in [\{-1, 1\}^d \to \{-1, 1\}]$ and $\clsf C := \polythr_d^{d-1} \sbe [\{-1, 1\}^d \to \{-1, 1\}]$.
Then $\betti_{i, f \cap \onef}(\clsf C^{\cartcap f}) = (2^{d-1} - 1)\ind(i = 1)$.
\end{thm}
\begin{proof}
Let $(Y, \mu)$ be the coball resolution of $\linthr_d$.
Consider the cell $F_\onef$ of $Y$ with PF label $\onef$.
It has $2^d$ facets since $$\ind_s := \begin{cases}
r \mapsto -1 & \text{if $r = s$}\\
r \mapsto 1 & \text{if $r \not = s$}
\end{cases}$$
is a linear threshold function (it ``cuts'' out a corner of the $d$-cube), so that each facet of $F_\onef$ is the cell $G_s := F_{\onef \cap \ind_s}$ with PF label
$$\onef \cap \ind_s = \begin{cases} r \mapsto 1 &\text{if $r = s$}\\ \text{undefined} & \text{otherwise.}\end{cases}$$

Note that a function $g: \{-1, 1\}^d \to \{-1, 0, 1\}$ is the sign function of a polynomial $p$ with degree $d-1$ iff $\im (gf) \spe \{-1, 1\}$ (i.e. $g$ hits both 1 and $-1$).
Indeed, by Fourier Transform, the degree constraint on $p$ is equivalent to
$$\la p, f\ra = \sum_{u \in \bcube^d} p(u)f(u) = 0.$$
If $\im gf = \im \sgn(p)f $ does not contain $-1$, then this quantity is positive as long as $p \not = 0$, a contradiction.
So suppose $\im gf \spe \bcube$.
Set $m_p := \#\{u: g(u)f(u) = 1\}$ and $m_n := \#\{u: g(u)f(u) = -1\}$.
Define the polynomial $p$ by $p(u) = \f {g(u)} {m_p}$ if $g(u)f(u) = 1$ and $p(u) = \f {g(u)}{m_n}$ if $g(u)f(u) = -1$.
Then $\la p, f \ra = 0$ and $\sgn p = g$ by construction, as desired.

As $\pbd_{f \cap \onef} F_\onef$ is the complex with the facets $\mathcal F := \{G_s: f(s) = 1\}$, to find its homology it suffices to consider the nerve of the facet cover.
For a function $g: \bcube^d \to \{-1, 0, 1\}$, write $\bar g$ for the partial function $g \restrict \inv g(\{-1, 1\})$ (essentially, we are marking as undefined all inputs that $g$ send to 0).
But by the above, any proper subset $\mathcal G \subset \mathcal F$ must have nontrivial intersection (which is a cell with PF label $\bar g$ for some $g$ with $\im g \spe \{-1, 1\}$), while $\bigcap\mathcal F$ must have PF label a subfunction $\pf g$ of $\bar h$ for
$$h(u) = \begin{cases} 
1 & \text{if $f(u) = -1$}\\
0 & \text{otherwise.}
\end{cases}$$
Again, by the last paragraph, this implies that $\bigcap \mathcal F = \emptyset$.
Therefore, in summary, the nerve $\mathcal N_{\mathcal F}$ is the boundary of a $(n/2 - 1)$-dimensional simplex, where $n = 2^d$, so that $\rH^j(\pbd_{f \cap \onef}F_\onef) \cong \rH^j(\mathcal N_{\mathcal F}) = \ind(j = n/2 - 2) \cdot \Z$.
Let $\alpha = \hdim \polythr^{d-1}_d = 2^d - 1$.
Then
\begin{align*}
\betti_{i, \onef \cap f}(\clsf C) &= \dim_\kk \rH^{\alpha - 1 - i}(F_\onef, \pbd_{\onef \cap f} F_\onef; \kk)\\
	&= \rank \rH^{\alpha - 2 - i}(\pbd_{\onef \cap f}F_\onef)\\
	&= \ind(\alpha - 2 - i = n/2 - 2)\\
	&= \ind(i = 2^{d-1} - 1)
\end{align*}
as desired.
\end{proof}
\subsection{The Maximal Principle for Threshold Functions}

\label{sec:maximal_principle}
By looking at the 0th Betti numbers of $\thr L^{\cartcap f}$, we can obtain a ``maximal principle'' for $\thr L$.
\begin{thm} \label{thm:maximal_principle_thr}
Suppose there exists a function $g \in \thr L$ such that
\begin{itemize}
    \item $g \not = f$ and,
    \item for each $h \in \thr L$ that differs from $g$ on exactly one input $u$, we have $g(u) = f(u) = \neg h(u)$. \label{_local_max}
\end{itemize}
Then $\betti_{i, f \cap g}(\thr L^{\cartcap f}) = \ind(i = 0)$ and $f \not \in \thr L$.
Conversely, any function $g \in \thr L$ satisfying $\betti_{i, f \cap g}(\thr L^{\cartcap f}) = \ind(i = 0)$ also satisfies condition (\ref{_local_max}).  
\end{thm}

Informally, \cref{thm:maximal_principle_thr} says that if we look at the partial order on $\thr L$ induced by the mapping from $\thr L$ to the class of partial functions, sending $g$ to $g \cap f$, then, assuming $f$ is in $\thr L$, any function $g$ that is a ``local maximum'' in $\thr L$ under this partial order must also be a global maximum and equal to $f$.
We shall formally call any function $g \in \thr L$ satisfying condition \ref{_local_max} a {\it local maximum} with respect to $f$.

\begin{proof}
Let $\coballres = (Y, \mu)$ be the minimal cocellular resolution of $\thr L$.
Let $F_g$ denote the face of $Y$ with label $g$.
Each facet of $F_g$ has the label $g \cap h$ for some $h$ differing from $g$ on exactly one input.
Condition (\ref{_local_max}) thus says that $\pbd_{f \cap g} F_g = \pd F_g$.
Therefore, if $l = \dim L$,
\begin{align*}
\betti_{i, f \cap g}(\thr L^{\cartcap f}) &= \dim_\kk \rH^{l - 1 - i}(F_g / \pbd_{f \cap g}; \kk)\\
	&= \dim_\kk \rH^{l-1-i}(F_g/\pd; \kk)\\
	&= \ind(l-1-i = \dim F_g)\\
	&= \ind(i = 0)
\end{align*}
This shows that $f \not \in \thr L$ as desired.

For the converse statement, we only need to note that the Betti number condition implies $\pbd_{f \cap g} F_g = \pd F_g$, by reversing the above argument.
\end{proof}
\newcommand{\thrdeg}{\operatorname{thrdeg}}
For any $f: \{-1, 1\}^d \to \{-1, 1\}$, define $\thrdeg f$ to be the minimal degree of any polynomial $P$ with $0 \not \in P(\{-1, 1\}^d)$ and $\sgn(P) = f$.
The maximal principle enables us to compute $\thrdeg f$ for any symmetric $f$ (a result that appeared in \cite{aspnes_expressive_1994}).
\begin{thm} 
Suppose $f: \{-1, 1\}^d \to \{-1, 1\}$ is symmetric, i.e. $f(u) = f(\pi \cdot u)$ for any permutation $\pi$.
Let $r$ be the number of times $f$ changes signs.
Then $\thrdeg f = r$.
\end{thm}
\begin{proof}
\textbf{To show $\thrdeg f \le r$:}
Let $s(u) := \sum_i (1 - u_i)/2$.
Because $f$ is symmetric, it is a function of $s$, say $\bar f(s(u)) = f(u)$ \footnote{$f$ can be expressed as a polynomial in $\{-1, 1\}^d$, and by the fundamental theorem of symmetric polynomials, $f$ is a polynomial in the elementary symmetric polynomials.
But with respect to the Boolean cube $\{-1, 1\}^d$, all higher symmetric polynomials are polynomials in $\sum_i u_i$, so in fact $f$ is a univariate polynomial in $s$.}.
Suppose WLOG $\bar f(0) > 0$ and $\bar f$ changes signs between $s$ and $s+1$ for $s = t_1, \ldots, t_r$.
Then define the polynomial $Q(s):= \prod_{i=1}^r (t_i + \f 1 2 - s)$.
One can immediately see that $\sgn Q(s) = \bar f(s) = f(u)$.
Therefore $\thrdeg f \le r$.

\noindent\textbf{To show $\thrdeg f \ge r$:}
Let $k = r - 1$ and consider the polynomial $Q'(s) = \prod_{i=1}^{r-1} (t_i + \f 1 2 - s)$ and its sign function $\bar g (s) = \sgn Q'(s) \in \polythr^k_d$.
We show that $g(u) = \bar g(s(u))$ a local maximum.
Since $\bar g(s) = \bar f(s)$ for all $s \in [0, t_r]$, it suffices to show that for any $v$ with $s(v) > t_r$, the function
$$
g^v(u) := \begin{cases}
g(u) & \text{if $u \not= v$}\\
\neg g(u) & \text{if $u = v.$}
\end{cases}$$
is not in $\polythr^k_d$.
WLOG, assume $v = (-1, \ldots, -1, 1, \ldots, 1)$ with $\sigma := s(v)$ $-1$'s in front.
For the sake of contradiction, suppose there exists degree $k$ polynomial $P$ with $\sgn P = g^v$.
Obtain through symmetrization the polynomial $R(z_1, \ldots, z_\sigma) := \sum_{\pi \in S_\sigma} P(\pi \cdot z, 1, \ldots, 1)$.
$R$ is a symmetric polynomial, so expressable as a univariate $R'(q)$ in $q := \sum_j (1 - z_j)/2 \in [0, \sigma]$ on the Boolean cube.
Furthermore, $\sgn R'(q) = \bar g(q)$ for all $q \not = \sigma$, and $\sgn R'(\sigma) = - \bar g(\sigma)$.
Thus $R'$ changes sign $k+1$ times on $[0, \sigma]$ but has degree at most $k$, a contradiction.
This yields the desired result.

\end{proof}

The proof above can be extended to give information on the zeroth Betti numbers of $\polythr^k_d \cartcap \{f\}$.
Suppose $f$ is again symmetric, and as in the proof above,
\begin{align*}
r &:= \thrdeg f\\
s(u) &:= \sum_i (1 - u_i)/2\\
\bar f(s(u)) &= f(u)\\
Q(s) &:= \bar f(0) \prod_{i=1}^r (t_i + \f 1 2 - s)
\end{align*}
where $\bar f$ changes signs between $s$ and $s+1$ for $s = t_1, \ldots, t_r$.
\begin{thm}
Let $k < r$ and $a < b \in [r]$ be such that $b - a = k-1$.
Set $Q'(s) = \bar f(t_a) \prod_{i=a}^{b}(t_i + \f 1 2 - s)$ and $g(u) := \bar g(s(u)) := \sgn Q'(s(u))$.
Then $\betti_{i, f \cap g}(\polythr^k_d \cartcap \{f\}) = \ind(i = 0)$.
\end{thm}
\begin{proof}
We prove the equivalent statement (by \cref{thm:maximal_principle_thr}) that $g$ is a local maximum.
Since $\bar g(s) = \bar f(s)$ for $s \in [t_a, t_b + 1]$, we just need to show that for any $v$ with $s(v) \not \in [t_a, t_b + 1]$, the function
$$
g^v(u) := \begin{cases}
g(u) & \text{if $u \not= v$}\\
\neg g(u) & \text{if $u = v.$}
\end{cases}$$
is not in $\polythr^k_d$.

If $s(v) > t_b + 1$, then WLOG assume $v = (-1, \ldots, -1, 1, \ldots, 1)$ with $\sigma := s(v)$ $-1$'s in front.
For the sake of contradiction, suppose there exists degree $k$ polynomial $P$ with $\sgn P = g^v$.
Obtain through symmetrization the polynomial $R(z_1, \ldots, z_\sigma) := \sum_{\pi \in S_\sigma} P(\pi \cdot z, 1, \ldots, 1)$.
$R$ is a symmetric polynomial, so expressable as a univariate $R'(q)$ in $q := \sum_j (1 - z_j)/2 \in [0, \sigma]$ on the Boolean cube.
Furthermore, $\sgn R'(q) = \bar g(q)$ on $q \in [0, \sigma-1]$, and $\sgn R'(\sigma) = - \bar g(\sigma)$.
Thus $R'$ changes sign $k+1$ times on $[0, \sigma]$ but has degree at most $k$, a contradiction.

If $s(v) < t_a$, then WLOG assume $v = (1, \ldots, 1, -1, \ldots, -1)$ with $\sigma := s(v)$ $-1$'s in the back.
For the sake of contradiction, suppose there exists degree $k$ polynomial $P$ with $\sgn P = g^v$.
Obtain through symmetrization the polynomial $R(z_1, \ldots, z_\sigma) := \sum_{\pi \in S_\sigma} P(\pi \cdot z, -1, \ldots, -1)$.
$R$ is a symmetric polynomial, so expressable as a univariate $R'(q)$ in $q := \sum_j (1 - z_j)/2 \in [0, d - \sigma]$ on the Boolean cube.
Furthermore, $\sgn R'(q) = \bar g(q + \sigma)$ on $q \in [1, \sigma-1]$, and $\sgn R'(0) = - \bar g(\sigma)$.
Thus $R'$ changes sign $k+1$ times on $[0, d-\sigma]$ but has degree at most $k$, a contradiction.
\end{proof}

\subsection{Homological Farkas}
\label{sec:hom_farkas}
\cref{thm:hom_farkas} essentially recovers \cref{thm:informal_hom_farkas}, after we define $\Lambda(g)$ to be $\pbd\bigstar_L(\neg\mathbf 1, \neg\mathbf 1 \cap g)$, and utilize the symmetry $\mathbf 1 \in \thr L \iff \neg \mathbf 1 \in \thr L$.
Then $\Lambda(g)$ indeed coincides with the union of facets of $\splx_g$ whose linear spans separates $\splx_g$ and $\splx_{\mathbf 1}$.

\newcommand{\hypar}{{\mathcal{H}}}
\newcommand{\boldplus}{{\boldsymbol{+}}}
\newcommand{\boldminus}{{\boldsymbol{-}}}
We can generalize the homological Farkas' lemma to arbitrary linear hyperplane arrangements.
Let $\hypar = \{H_i\}_{i=1}^n$ be a collection of hyperplanes in $\R^k$, and $\{w_i\}_i$ be a collection of row matrices such that
$$H_i = \{x \in \R^k: w_i x = 0\}.$$
Set $W$ to be the matrix with rows $w_i$.
For $\mathbf b \in \{-, +\}^n$, define $R_{\mathbf b}:= \{x \in \R^k: \sgn(Wx) = \mathbf b\}$.
Thus $R_{\boldplus} = \{x \in \R^k: Wx > 0\}.$
Suppose $W$ has full rank (i.e. the normals $w_i$ to $H_i$ span the whole space $\R^k$), so that $W$ is an embedding $\R^k \into \R^n$.
Each region $R_{\mathbf b}$ is the preimage of $P_{\mathbf b}$, the cone in $\R^n$ with sign $\mathbf b$.
Therefore, $R_{\mathbf b}$ is linearly isomorphic to $\im W \cap P_{\mathbf b}$, via $W$.

Let $L \sbe \R^k$ be a linear subspace of dimension $l$.
Then
\begin{align*}
L \cap R_{\boldplus} \not= \emptyset &\iff W(L) \cap P_{\boldplus} \not= \emptyset \\
	&\iff \forall \mathbf b \not = \boldplus, \boldminus, [W(L) \cap \Lambda(\mathbf b) \not= \emptyset \implies \rH^{l - 2- i}(W(L) \cap \Lambda(\mathbf b)) = 0 \forall i \ge 0]\\
	&\iff \forall \mathbf b \not = \boldplus, \boldminus, [L \cap \inv W\Lambda(\mathbf b) \not= \emptyset \implies \rH^{l - 2- i}(L \cap \inv W \Lambda(\mathbf b)) = 0 \forall i \ge 0]
\end{align*}

This inspires the following definition.
\begin{defn}
Let $\hypar = \{H_i\}_{i=1}^n$ and $W$ be as above, with $W$ having full rank.
Suppose $\mathbf b \in \{-, +\}^n$.
Then $\Lambda_\hypar(\mathbf b)$ is defined as the union of the facets of $R_{\mathbf b} \cap S^{k-1}$ whose linear spans separate $R_{\mathbf b}$ and $R_{\boldplus}$.
\end{defn}
One can immediately see that $\Lambda_\hypar(\mathbf b) = \inv W \Lambda(\mathbf b)$.

In this terminology, we have shown the following
\begin{cor} \label{cor:linear_hom_farkas}
Let $\hypar = \{H_i\}_{i=1}^n$ be a collection of linear hyperplanes in $\R^k$ whose normals span $\R^k$ (This is also called an {\it essential hyperplane arrangement.}).
Suppose $L \sbe \R^k$ is a linear subspace of dimension $l$.
Then either
\begin{itemize}
    \item $L \cap R_{\boldplus} \not= \emptyset$, or
    \item there is some $\mathbf b \not = \boldplus, \boldminus$, such that $L \cap \Lambda_\hypar(\mathbf b) \not = \emptyset$ and $\rH^{l - 2 - i}(L \cap \Lambda_\hypar(\mathbf b)) \not= 0$ for some $i\ge 0$,
\end{itemize}
but not both.
\end{cor}

This corollary can be adapted to the affine case as follows.
Let $\Aa = \{A_i\}_{i=1}^n$ be an essential oriented affine hyperplane arrangement in $\R^{k-1}$.
The hyperplanes $\Aa$ divide $\R^{k-1}$ into open, signed regions $R_{\mathbf b}, \mathbf b \in \{-, +\}^n$ such that $R_{\mathbf b}$ lies on the $\mathbf b_i$ side of $A_i$.
We can define $\Lambda_\Aa(\mathbf b)$ as above, as the union of facets $F$ of $R_{\mathbf b}$ such that $R_{\mathbf b}$ falls on the negative side of the affine hull of $F$, along with their closures in the ``sphere at infinity.''

Let $H_b := \{(x, b): x \in \R^{k-1}\}$.
Treating $\R^{k-1} \hookrightarrow \R^k$ as $H_1$, define $\vec \Aa = \{\vec A_i\}_{i=1}^n$ to be the oriented linear hyperplanes vectorizing $A_i$ in $\R^k$.
Vectorization produces from each $R_{\mathbf b}$ two cones $\vec R_{\mathbf b}, \vec R_{\neg \mathbf b} \sbe \R^k$, defined by
\begin{align*}
\vec R_{\mathbf b} &:= \{v \in \R^k: \exists c>0, cv \in R_{\mathbf b}\}\\
\vec R_{\neg\mathbf b} &:= \{v \in \R^k: \exists c<0, cv \in R_{\mathbf b}\}.
\end{align*}
Define $\neg \Aa$ as the hyperplane arrangement with the same hyperplanes as $\Aa$ but with orientation reversed.
Let $\neg R_{\mathbf b}$ denote the region in $\neg \Aa$ with sign $\mathbf b$.
Set $\Lambda_{\neg\Aa}(\mathbf b)$ analogously for $\neg \Aa$, as the union of facets $F$ of $\neg R_{\mathbf b}$ such that $\neg R_{\mathbf b}$ falls on the negative side of the affine hull of $F$, along their closures in the ``sphere at infinity.''
Thus the natural linear identity between $\neg \Aa$ and $\Aa$ identifies $\neg R_{\neg \mathbf b}$ with $R_{\mathbf b}$, and $\Lambda_{\neg\Aa}(\neg \mathbf b)$ with the union of facets {\it not in} $\Lambda_\Aa(\mathbf b)$.
See \cref{fig:affine_hom_farkas}.

\begin{figure}
\centering
\includegraphics[height=.28\textheight]{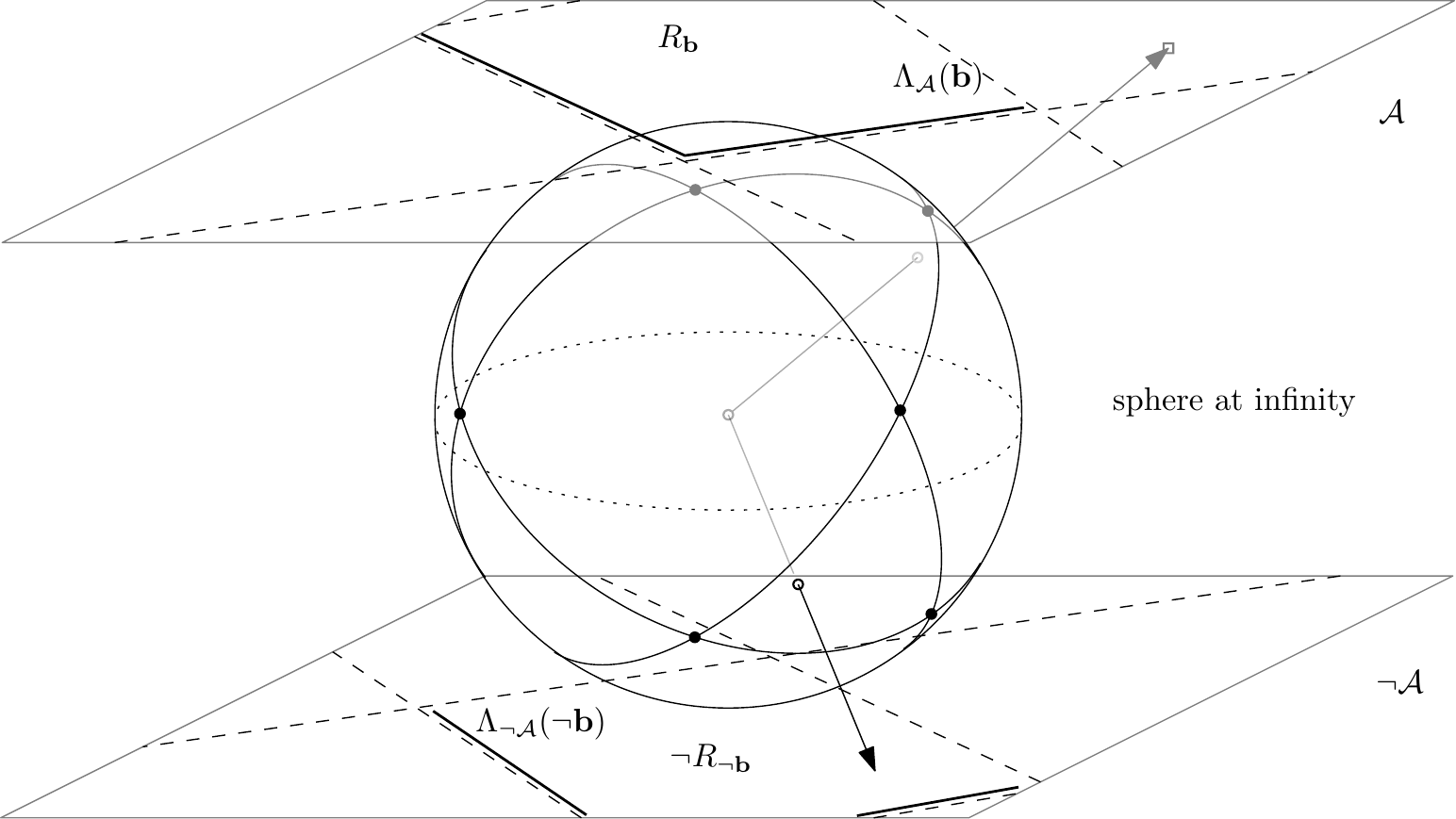}
\caption{Illustration of the symbols introduced so far.}
\label{fig:affine_hom_farkas}
\end{figure}
Note that, by construction, $\vec A_i \cap H_1 = A_i$ as oriented hyperplanes, and by symmetry, $\vec A_i \cap H_{-1} = \neg A_i$.
By projection with respect to the origin, $\Aa$ and $\neg \Aa$ can be glued along the ``sphere at infinity'' to form $\{\vec A_i \cap S^{n-1}\}_i$. 
Similarly, $R_{\mathbf b}$ and $\neg R_{\mathbf b}$ can be glued together along a subspace of the ``sphere at infinity'' to obtain $\vec R_{\mathbf b}$, and $\Lambda_\Aa(\mathbf b)$ and $\Lambda_{\neg\Aa}(\mathbf b)$ can be glued together likewise to obtain $\Lambda_{\vec \Aa}(\mathbf b)$.
We denote this ``gluing at infinity'' construction by $- \sqcup_\infty -$, so that we write $\vec R_{\mathbf b} = R_{\mathbf b} \sqcup_\infty \neg R_{\mathbf b}$ and $\Lambda_{\vec \Aa}(\mathbf b) = \Lambda_\Aa(\mathbf b) \sqcup_\infty \Lambda_{\neg \Aa}(\mathbf b).$

Let $N$ be an affine subspace of $\R^{k-1}$ of dimension $l-1$, and let $\vec N$ be its vectorization in $\R^k$.
Then 
\begin{align*}
N \cap R_{\boldplus} \not=\emptyset &\iff \vec N \cap \vec R_{\boldplus} \not=\emptyset \iff \vec N \cap \vec R_{\boldminus} \not=\emptyset\\
	&\iff \forall \mathbf b \not= \boldplus, \boldminus, [\vec N \cap \Lambda_{\vec \Aa}(\mathbf b) \not= \emptyset \implies \rH^{l-2-i}(\vec N \cap \Lambda_{\vec \Aa}(\mathbf b)) = 0 \forall i \ge 0]
\end{align*}
But $\vec N \cap \Lambda_{\vec \Aa}(\mathbf b) = (N \cap \Lambda_\Aa(\mathbf b)) \sqcup_\infty (N \cap \Lambda_{\neg \Aa}(\mathbf b))$, so we get the following

\begin{cor}\label{cor:affine_glue_hom_farkas}
$N$ does not intersect $R_\boldplus$ iff there is some $\mathbf b \not= \boldplus, \boldminus$ such that $(N \cap \Lambda_\Aa(\mathbf b)) \sqcup_\infty (N \cap \Lambda_{\neg \Aa}(\mathbf b))$ is nonempty and is not nulhomotopic.
\end{cor}

When $N$ does not intersect the closure $R_\boldplus$, we can just look at $\Lambda_\Aa(\mathbf b)$ and the component at infinity for a homological certificate.

\begin{cor} \label{cor:simple_affine_hom_farkas}
Let $\Aa = \{A_i\}_{i=1}^n$ be an affine hyperplane arrangement in $\R^{k-1}$ whose normals affinely span $\R^{k-1}$.
Suppose $R_\boldplus$ is bounded and let $N$ be an affine subspace of dimension $l-1$.
Then the following hold:
\begin{enumerate}
    \item If $R_\boldplus \cap N \not=\emptyset$, then for all $\mathbf b \not= \boldplus, \boldminus$, $N \cap \Lambda_\Aa(\mathbf b) \not = \emptyset \implies \rH^\bullet(N \cap \Lambda_\Aa(\mathbf b)) = 0$. \label{_intersects_int}
    \item If $\setcl R_\boldplus \cap N =\emptyset$, then for each $j \in [0, l-2]$, there exists $\mathbf b \not= \boldplus, \boldminus$ such that $N \cap \Lambda_\Aa(\mathbf b) \not= \emptyset$ and $\rH^j(N \cap \Lambda_\Aa(\mathbf b)) = 0$ for some $j$. \label{_miss_closure}
\end{enumerate}
\end{cor}
\begin{proof}
(\cref{_intersects_int})
Consider $\vec \Bb := \vec \Aa \cup \{H_0\}$, where $H_0$ is the linear hyperplane of $\R^k$ with last coordinate 0, oriented toward positive side.
Write $\vec R'_{\mathbf c}$ for the region with sign $\mathbf c$ with respect to $\Bb$ (where $\mathbf c \in \{-, +\}^{n+1}$).
Because $R_\boldplus$ is bounded, $\vec R_\boldplus$ does not intersect $H_0$ other than at the origin.
Then $N \cap R_\boldplus \not= \emptyset \iff \vec N \cap R'_\boldplus \not= \emptyset \iff \forall \mathbf c, \vec N \cap \Lambda_{\vec \Bb}(\mathbf c) \text{ is nulhomotopic if nonempty}$.
Note that for any $\mathbf b \in \{-, +\}^n$, we have $\Lambda_{\vec \Bb}(\mathbf b \append +) \cong \Lambda_\Aa(\mathbf b)$, and $\vec N \cap \Lambda_{\vec \Bb}(\mathbf c) \cong N \cap \Lambda_{\Aa}(\mathbf b)$. (Here $\mathbf b \append +$ means $+$ appended to $\mathbf b$).
Substituing $\mathbf c = \mathbf b \append +$ into the above yields the result.

(\cref{_miss_closure})
The most natural proof here adopts the algebraic approach.

Let $W_{\vec B}: \R^k \to \R^{n+1}$ be the embedding matrix for $\vec B$.
Consider $\clsf C := \thr W_{\vec B}(\vec N) \sbe [[n+1] \to \{-, +\}]$.
This is the class of functions corresponding to all the sign vectors achievable by $\vec N$ as it traverses through the regions of $\vec B$.
Define $\pclsf C := \clsf C ^{\cartcap \boldplus}$.
Since $N \cap \setcl R_\boldplus = \emptyset$, $\betti_{i, \emptyfun}(\pclsf C) = \ind(i = l)$ by \cref{thm:thr_top_betti}.
By the minimality of Betti numbers, for all $j \ge 0$, $\betti_{l-1-j, \pf f}(\pclsf C) \not = 0$ for some $\pf f:\sbe [n+1] \to \{-, +\}, \pf f \not= \emptyfun, \boldplus$ with $n+1 \not \in \dom \pf f$.
But this means that $\rH^{j}(\Xi_{\vec \Bb}(\boldplus \flip \pf f) \cap \vec N) \not = 0$ by the proof of \cref{thm:thr_top_betti}.
Of course, $(\boldplus \flip \pf f)(n+1) = -$, meaning that $\Xi_{\vec \Bb}(\boldplus \flip \pf f) \cap \vec N \cong \Lambda_\Aa(\neg(\boldplus \flip \pf f)) \cap N$.
For the desired result, we just set $\mathbf b = \neg(\boldplus \flip \pf f)$.
\end{proof}

\begin{figure}
\centering
\includegraphics[width=.4\textwidth]{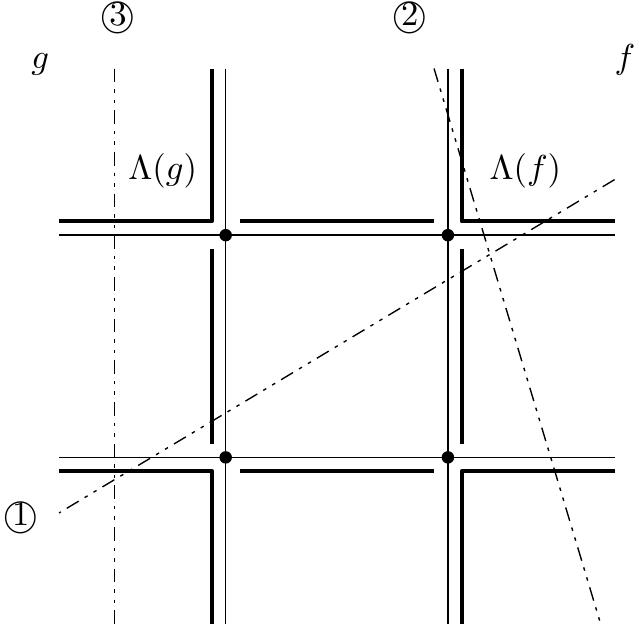}
\caption{Example application of \cref{cor:simple_affine_hom_farkas}.
Let the hyperplanes (thin lines) be oriented such that the square at the center is $R_\boldplus$.
The bold segments indicate the $\Lambda$ of each region.
Line 1 intersects $R_\boldplus$, and we can check that its intersection with any bold component is nulhomotopic.
Line 2 does not intersect $\setcl{R_\boldplus}$, and we see that its intersection with $\Lambda(f)$ is two points, so has nontrivial zeroth reduced cohomology.
Line 3 does not intersect $\setcl{R_\boldplus}$ either, and its intersection with $\Lambda(g)$ consists of a point in the finite plane and another point on the circle at infinity.}
\label{fig:generalized_homological_farkas}
\end{figure}

\cref{fig:generalized_homological_farkas} gives an example application of \cref{cor:simple_affine_hom_farkas}.

%
%

%
%

\subsection{Probabilistic Interpretation of Hilbert Function}
\label{sec:hilbert_function}
\newcommand{\bettip}{\mathcal{B}}
\newcommand{\eulerchar}{\chi}
\newcommand{\Kpoly}{\mathcal{K}}
\newcommand{\totdeg}{\operatorname{totdeg}}
\newcommand{\Hapx}{\aleph}
\newcommand{\Hilbs}{\mathcal{HS}}
\newcommand{\Hilbf}{\mathcal{HF}}
\newcommand{\Hilbp}{\mathcal{HP}}
\newcommand{\monop}{\wp}

In this section we exhibit a probabilistic interpretation of the Hilbert function of a canonical ideal.
For a graded module $M$ over $S = \kk[x_0, \ldots, x_{n-1}]$, the {\bf graded Hilbert function} $\Hilbf(M; \es a)$ takes an exponent sequence $\es a$ to the dimension over $\kk$ of the component of $M$ with degree $\es a$.
Its generating function
$$\Hilbs(M; \xx) = \sum_{\es a} \Hilbf(M; \es a)\xx^{\es a}$$
is called the {\bf graded Hilbert series} of $M$.
It is known \cite{miller_combinatorial_2005} that
$$\Hilbf(M; \es a) = \f{\Kpoly(M; \es a)}{\prod_{i=0}^{n-1}(1-x_i)}$$
for some polynomial $\Kpoly$ in $n$ variables.
This polynomial is called the {\bf K-polynomial} of $M$.
If one performs fractional decomposition on this rational function, then one can deduce that the Hibert function coincides with a polynomial when $\es a$ has large total degree.
(This is briefly demonstrated below for the $\N$-graded version).
This polynomial is called the {\bf Hilbert polynomial} and is written $\Hilbp(M; \es a)$.

Let $\eulerchar_{M}(\xx)$ denote the {\bf graded Euler characteristic} of a module $M$:
$$\eulerchar_{M}(\xx) = \sum_{i \ge 0}\sum_{\es a \succeq \mathbf 0} (-1)^{i}\betti_{i, \es a}(M)\xx^{\es a}.$$
For example, for $M = \cani{\clsf C}$, we write $\eulerchar_{\clsf C}(\xx) := \eulerchar_{\cani{\clsf C}}(\xx)$, and it takes the form
$$\eulerchar_{{\clsf C}}(\mathbf x) = \sum_{f \in \clsf C} \xx^{\Gamma f} - \sum_{f \nei_{\clsf C} g} \xx^{\Gamma(f \cap g)} + \cdots.$$
It is shown in \cite[Thm 4.11, 5.14]{miller_combinatorial_2005} that
$$\eulerchar_{I^\star}(1 - \xx) = \Kpoly(I^\star; 1 - \xx) = \Kpoly(S/I; \xx) = \eulerchar_{S/I}(\xx)$$
for any squarefree monomial ideal $I$.

Now let $\clsf C \sbe [n \to 2]$, and let $S$ be its canonical base ring.
For any $f \not \in \clsf C$, the minimal generators of $\cani{\Xcapf{\clsf C}}$ are $\xx^{\Gamma (f \cap g)}$ for $g \in \clsf C$ ``closest'' to $f$.
In particular, for every function $h \in \clsf C$, $|\dom f \cap h| \le |\dom \pf f| = 2^{d+1}- \totdeg \xx^{\Gamma \pf f}$ for some $\xx^{\Gamma \pf f} \in \mingen(\cani{\Xcapf{\clsf C}})$.
\begin{defn}
Define the {\bf hardness of approximating $f$ with $\clsf C$} as $\Hapx(f, \clsf C) = \min\{2^d - |\dom f \cap h|: h \in \clsf C\}$
\end{defn}
Then $\aleph(f, \clsf C)$ is the smallest total degree of any monomial appearing in $\eulerchar_{\Xcapf{\clsf C}}$ minus $2^d$.

Therefore,
\begin{align*}
\Hapx(f, \clsf C) &= \lim_{\zeta \to 0+} \f{\log \chi_{\Xcapf{\clsf C}}(\zeta, \ldots, \zeta)}{\log \zeta} - 2^d\\
&=\lim_{\vartheta \to 1-} \f{\log \Kpoly(S/I_{\Xcapf{\clsf C}}; \vartheta, \ldots, \vartheta)}{\log 1 - \vartheta} -2^d\\
&=\lim_{\vartheta \to 1-} \f{\log \Hilbs(S/I_{\Xcapf{\clsf C}}; \vartheta, \ldots, \vartheta)}{\log 1 - \vartheta} + 2^{d}\\
\end{align*}
where the last equality follows from
$$\Hilbs(S/I; t, \ldots t) = \Kpoly(S/I; t, \ldots, t)/(1 - t)^{2^{d+1}}.$$
The $\N$-graded Hilbert series expands into a Laurent polynomial in $(1 - t)$,
$$\Hilbs(S/I; t, \ldots, t) = \f{a_{-r-1}}{(1 - t)^{r+1}} + \cdots + \f{a_{-1}}{(1 - t)} + a_0 + \cdots a_s t^s$$
such that the $\N$-graded Hilbert polynomial $\Hilbp(S/I; t, \ldots, t)$ has degree $r$.
Thus
\begin{align*}
\Hapx(f, \clsf C) &= 2^{d} - (r+1)\\
&= 2^{d} - \deg \Hilbp(S/I_{\Xcapf{\clsf C}}; t, \ldots, t) - 1\\
&= 2^{d} - \totdeg \Hilbp(S/I_{\Xcapf{\clsf C}}) - 1
\end{align*}
Note that total number of monomials in degree $k$ is $\binom{k + 2^{d+1} - 1}{2^{d+1}-1} = \Theta(k^{2^{d+1}-1})$.
Therefore, if we define $\wp(k; f, \clsf C)$ to be the probability that a random monomial $\bomega$ of degree $k$ has $\supp \bomega \sbe f \cap h$ for some $h \in \clsf C$, then $\wp(k; f, \clsf C) = \Theta\left(\f{\Hilbp(S/I_{\Xcapf{\clsf C}}; k)}{k^{2^{d+1}-1}}\right)$, and
$$\Hapx(f, \clsf C) = - \lim_{k \to \infty} \f{\log \wp(k; f, \clsf C)}{\log k} - 2^d.$$


\newcommand{\mProb}{\mathcal{Q}}
Now, $\wp(k)$ is really the probability that a PF $\pf f$ has extension in $\clsf C$ \textit{and} is extended by $f$, where $\pf f$ is chosen from the distribution $\mProb_k$ that assigns a probability to $\pf f$ proportional to the number of monomials of total degree $k$ whose support is $\pf f$.
More precisely,
\begin{align*}
\mProb_k(\pf f) = \f{\binom{k - 1}{|\dom \pf f| - 1}}
					{\binom{k + 2^{d+1} - 1}{2^{d+1}-1}}
\end{align*}
Note that there is a nonzero probability of choosing an invalid partial function, i.e. a monomial that is divisible by $x_{u, 0}x_{u, 1}$ for some $u \in [2^d]$.
Under this distribution, a particular PF of size $d+1$ is $\Theta(k)$ times as likely as any particular PF of size $d$.
As $k \to \infty$, $\mProb_k$ concentrates more and more probability on the PFs of large size. 

By a similar line of reasoning using $\clsf C$ instead of $\Xcapf{\clsf C}$, we see that $\deg \Hilbp(S/I_{\clsf C}) + 1 = 2^d$, so we define $\Hapx(\clsf C) = 0$.
Therefore the probability that a PF $\pf f$ has extension in $\clsf C$ when $\pf f$ is drawn from $\mProb_k$ is
$$\wp(k; \clsf C) \sim k^{-2^d}.$$
We deduce that 
\begin{thm}
The probability that a PF $\pf f$ drawn from $\mProb_k$ is extended by $f$ when it is known to have extension in $\clsf C$ is $\Theta(k^{-\Hapx(f, \clsf C)})$.
\end{thm}
Note that we are assuming $f$ and $\clsf C$ are fixed, and in particular when we are interested in a parametrized family $(f_d, \clsf C_d)$, there might be dependence on $d$ that is not written here.
The main point we want to make, however, is that the Betti numbers of $\clsf C$ and $\Xcapf{\clsf C}$ affect the behavior of these classes under certain kinds of probability distributions.
By considering higher Betti numbers and their dependence on the parameter $d$, we may compute the dependence of $\wp$ on $d$ as well.
Conversely, carrying over results from subjects like statistical learning theory could yield bounds on Betti numbers this way.

\section{Discussion}

We have presented a new technique for complexity separation based on algebraic topology and Stanley-Reisner theory, which was used to give another proof of Minsky and Papert's lower bound on the degree of polynomial threshold function required to compute parity.
We also explored the connection between the algebraic/topological quantity $\hdim \clsf C$ and learning theoretical quantity $\vcdim \clsf C$, and surprisingly found that the former dominates the latter, with equality in common computational classes.
The theory created in this paper seems to have consequences even in areas outside of computation, as illustrated the Homological Farkas Lemma.
Finally, we exhibited a probabilistic interpretation of the Hilbert function that could provide a seed for future developments in hardness of approximation.

\subsection{Geometric Complexity Theory}
For readers familiar with Mulmuley's Geometric Complexity program \cite{mulmuley_geometric_2001}, a natural question is perhaps in what ways is our theory different?
There is a superficial similarity in that both works associate mathematical objects to complexity classes and focus on finding obstructions to equality of complexity classes.
In the case of geometric complexity, each class is associated to a variety, and the obstructions sought are of representation-theoretic nature.
In our case, each class is associated to a labeled simplicial complex, and the obstructions sought are of homological nature.
But beyond this similarity, the inner workings of the two techniques are quite distinct.
Whereas geometric complexity focuses on using algebraic geometry and representation theory to shed light on primarily the determinant vs permanent question, our approach uses combinatorial algebraic topology and has a framework general enough to reason about any class of functions, not just determinant and permanent.
This generality allowed, for example, the unexpected connection to VC dimension.
It remains to be seen, however, whether these two algebraic approaches are related to each other in some way.

\subsection{Natural Proofs}

So this homological theory is quite different from geometric complexity theory.
Can it still reveal new insights on the $\P = \NP$ problem?
Based on the methods presented in this paper, one might try to show $\P/\poly \not= \NP$ by showing that the ideal $\cani{\SIZE(d^c) \cartcap \{3\SAT_d\}}$ is not principal, for any $c$ and large enough $d$.
Could Natural Proofs \cite{razborov_natural_1997} present an obstruction?

A predicate $\mathcal P: [2^d \to 2] \to [2]$ is called {\bf natural} if it satisfies
\begin{itemize}
    \item (Constructiveness) It is polynomial time in its input size: there is an $2^{O(d)}$-time algorithm that on input the graph of a function $f \in [2^d \to 2]$, outputs $\mathcal P(f)$.
    \item (Largeness) A random function $f \in [2^d \to 2]$ satisfies $\mathcal P(f) = 1$ with probability at least $\f 1 n$.
\end{itemize}

Razborov and Rudich's celebrated result says that
\begin{thm}\cite{razborov_natural_1997}
Suppose there is no subexponentially strong one-way functions.
Then there exists a constant $c$ such that no natural predicate $\mathcal P$  maps $\SIZE(d^c) \sbe [2^d \to 2]$ to 0.
\end{thm}
This result implicates that common proof methods used for proving complexity separation of lower complexity classes, like Hastad's switching lemma used in the proof of $\parity \not \in \AC^0$ \cite{arora_computational_2009}, cannot be used toward $\P$ vs $\NP$. 

In our case, since $\SIZE(d^c)$ has $2^{\poly(d)}$ functions, naively computing the ideal $\cani{\SIZE(d^c) \cartcap \{3\SAT_d\}}$ is already superpolynomial time in $2^d$, which violates the ``constructiveness'' of natural proofs.
Even if the ideal $\cani{\SIZE(d^c)^{\cartcap 3\SAT_d}}$ is given to us for free, computing the syzygies of a general ideal is $\NP$-hard in the number of generators $\Omega(2^d)$ \cite{bayer_computation_1992}.
Thus a priori this homological technique is not natural (barring the possibility that in the future, advances in the structure of $\Ss_{\SIZE(d^c)}$ yield $\poly(2^d)$-time algorithms for the resolution of $\cani{\SIZE(d^c) \cartcap \{3\SAT_d\}}$).


\subsection{Homotopy Type Theory}
A recent breakthrough in the connection between algebraic topology and computer science is the emergence of Homotopy Type Theory (HoTT) \cite{the_univalent_foundations_program_homotopy_2013}.
This theory concerns itself with rebuilding the foundation of mathematics via a homotopic interpretation of type theoretic semantics.
Some of the key observations were that dependent types in type theory correspond to fibrations in homotopy theory, and equality types correspond to homotopies.
One major contribution of this subfield is the construction of a new (programming) language which ``simplifies'' the semantics of equality type, by proving, internally in this language, that isomorphism of types ``is equivalent'' to equality of types.
It also promises to bring automated proof assistants into more mainstream mathematical use.
As such, HoTT ties algebraic topology to the B side (logic and semantics) of theoretical computer science.

Of course, this is quite different from what is presented in this paper, which applies algebraic topology to complexity and learning theory (the A side of TCS).
However, early phases of our homological theory were inspired by the ``fibration'' philosophy of HoTT.
In fact, the canonical suboplex was first constructed as a sort of ``fibration'' (which turned out to be a cosheaf, and not a fibration) as explained in \cref{sec:cosheaf}.
It remains to be seen if other aspects of HoTT could be illuminating in future research.
\section{Future Work}

\label{sec:future_work}

In this work, we have initiated the investigation of function classes through the point of view of homological and combinatorial commutative algebra.
We have built a basic picture of this mathematical world but left many questions unanswered.
Here we discuss some of the more important ones.

\paragraph{Characterize when $\vcdim = \hdim$, or just approximately.}
We saw that all of the interesting computational classes discussed in this work, for example, $\linthr$ and $\linfun$, have homological dimensions equal to their VC dimensions.
We also showed that Cohen-Macaulay classes also satisfy this property.
On the other hand, there are classes like $\cDelta$ whose homological dimensions are very far apart from their VC dimensions.
A useful criterion for when this equality can occur, or when $\hdim = O(\vcdim)$, will contribute to a better picture when the homological properties of a class reflect its statistical/computational properties.
Note that adding the all 0 function to $\cDelta$ drops its homological dimension back to its VC dimension.
So perhaps there is a notion of ``completion'' that involves adding a small number of functions to a class to round out the erratic homological behaviors?

\paragraph{Characterize the Betti numbers of $\Xcapf{\thr L}$.}
We showed that the Betti numbers of $\Xcapf{\thr L}$ has nontrivial structure, and that some Betti numbers correspond to known concepts like weak representation of $f$.
However, we only discovered a corner of this structure.
In particular, what do the ``middle dimension'' Betti numbers look like?
We make the following conjecture.
\begin{conject}
Let $\clsf C = \polythr^k_d$ and $f \not \in \clsf C$.
For every PF $\pf f$ in $\Xcapf{\clsf C}$, there is some $i$ for which $\betti_{i, f \setminus \pf f}(\clsf C)$ is nonzero.
\end{conject}
It can be shown that this is not true for general $\thr L$ classes, but computational experiments suggest this seems to be true for polynomial thresholds.

\paragraph{How do Betti numbers of $\Xcapf{\thr L}$ change under perturbation of $f$?}
We proved a stability theorem for the ``codimension 1'' Betti numbers.
In general, is there a pattern to how the Betti numbers respond to perturbation, other than remaining stable?

\paragraph{Does every boolean function class have a minimal cellular or cocellular resolution?}
It is shown in \cite{velasco_minimal_2008} that there exist ideals whose minimal resolutions are not (CW) cellular.
A natural question to ask here is whether this negative results still holds when we restrict to canonical ideals of boolean, or more generally finite, function classes.
If so, we may be able to apply techniques from algebraic topology more broadly.

\paragraph{When does a class $\clsf C$ have pure Betti numbers?}
If we can guarantee that restriction preserves purity of Betti numbers, then \cref{thm:restriction_betti_exact_seq} can be used directly to determine the Betti numbers of restriction of classes.
Is this guarantee always valid?
How do we obtain classes with pure Betti numbers?

\paragraph{Under what circumstances can we expect separation of classes using high dimensional Betti numbers?}
Betti numbers at dimension 0 just encode the members of a class, and Betti numbers at dimension 1 encode the ``closeness'' relations on pairs of functions from the class.
On the other hand, the maximal dimension Betti number of $\Xcapf{\thr L}$ encodes information about weak representation of $f$.
So it seems that low dimension Betti numbers reflect more raw data while higher dimension Betti numbers reflect more ``processed'' data about the class, which are probably more likely to yield insights different from conventional means.
Therefore, the power of our method in this view seems to depend on the dimension at which differences in Betti number emerges (as we go from high dimension to low dimension).


\paragraph{Extend the probabilistic interpretation of Hilbert function.}
One may be able to manipulate the distribution $\mathcal Q_k$ in \cref{sec:hilbert_function} to arbitrary shapes when restricted to total functions, by modifying the canonical ideal.
This may yield concrete connections between probabilistic computation and commutative algebra.

\paragraph{Prove new complexity separation results using this framework}
We have given some examples of applying the homological perspective to prove some simple, old separation results, but hope to find proofs for nontrivial separations in the future.

\newpage

\appendix
\renewcommand\thefigure{\thesection.\arabic{figure}}
\setcounter{table}{0}
\renewcommand*{\thetable}{\thesection.\arabic{table}}

\setcounter{figure}{0}
\appendixpage
\section{Omitted Proofs}
\label{sec:omitted_proofs}
\begin{proof}[Proof of \cref{lemma:subcollection_realization}]
The set of open cells in $\overline{\mathcal U} \setminus \pbd \mathcal U$ is obviously $\mathcal U$.
So we need to show that $\overline {\mathcal U}$ and $\pbd \mathcal U$ are both subcomplex of $Y$.
The first is trivial by \cref{lemma:closed_iff_subcomplex}.

Suppose $\mathcal U = Y_{\preceq \es b}$.
An open cell $\ocell F$ is in $\pbd \mathcal U$ only if its label $\es a_F \not \preceq\es b$.
But then any cell in its boundary $\pd F$ must fall inside $\pbd \mathcal U$ as well, because its exponent label majorizes $\es a_F$.
Thus the closed cell satsifies $F \in \pbd \mathcal U$.
This shows $\pbd \mathcal U$ is closed and thus a subcomplex by \cref{lemma:closed_iff_subcomplex}.

The case of $\mathcal U = Y_{\prec \es b}$ has the same proof.

For $\mathcal U = Y_{\es b}$, the only difference is the proof of $\pbd \mathcal U$ being closed.
We note that an open cell $\ocell F$ is in $\pbd \mathcal U$ iff $\ocell F \in \overline{\mathcal U}$ and its label $\es a_F \succ \es b$.
Thus any open cell $\ocell G$ in its boundary $\pd F$ falls inside $\pbd \mathcal U$ as well, because its exponent label $\es a_G \succeq\es a_F \succ \es b$.
So $ F \in \pbd \mathcal U$, and $\pbd \mathcal U$ is closed, as desired.
\end{proof}

\begin{proof}[Proof of \cref{lemma:cocell_acyc}]

Let $\resf E$ be the chain complex obtained from cochain complex $\resf F^{(X, A)}$ by placing cohomological degree $d$ at homological degree 0.
For each $\es a$, we show the degree $\xx^{\es a}$ part $\resf E^{\es a}$ of $\resf E$ has rank 0 or 1 homology at homological degree 0 and trivial homology elsewhere iff one of the three conditions are satisfied.

As a homological chain complex, $\resf E^{\es a}$ consists of free modules $\resf E^{\es a}_i$ at each homological degree $i$ isomorphic to a direct sum $\bigoplus_{F \in \Delta_{d-i}((X, A)_{\preceq \es a})} S$, where $\Delta_i(X, A)$ denotes the pure $i$-skeleton of the pair $(X, A)$ (i.e. the collection of open cells of dimension $i$ in $X \setminus A$).
Writing $S_F$ for the copy of the base ring $S$ corresponding to the cell $F$, the differential is given componentwise by
$$d: \resf E^{\es a}_i \to S_G \in \resf E^{\es a}_{i-1}, a \mapsto \sum_{\text{facets} F \subset G} \mathrm{sign}(F, G) a_F.$$

If $K$ is void, this chain is identically zero.

Otherwise if $\pbd K$ is empty, then $\resf E^{\es a}$ just reproduces the reduced simplicial cochain complex of $K$ --- reduced because the empty cell is in $K$ and thus has a corresponding copy of $S$ at the highest homological degree in $\resf E^{\es a}$.
Then $H_i(\resf E^{\es a}) = \rH^{d-i}(K)$ is nonzero only possible at $i = 0$, as desired, and at this $i$, the rank of the homology is 0 or 1 by assumption.

Finally, if $\pbd K$ contains a nonempty cell, then $\resf E^{\es a}$ recovers the relative cochain complex for $(\overline K, \pbd K)$.
Then $H_i(\resf E^{\es a}) = \tilde H^{d-i}(\overline K, \pbd K)$ is nonzero only possible at $i = 0$, where the rank of the homology is again 0 or 1.

This proves the reverse direction ($\Leftarrow$).

For the forward direction ($\Rightarrow$), suppose $\pbd K$ only contains an empty cell (i.e. does not satisfy conditions 1 and 2).
Then $\resf E^{\es a}$ is the nonreduced cohomology chain complex of $K$, and therefore it must be the case that $H^i(K) = H_{d-i}(\resf E^{\es a}) = 0$ at all $i \not= d$.
But $H^0(K) = 0$ implies $K$ is empty, yielding condition 3.

Otherwise, if $\pbd K$ is void, this implies condition 2 by the reasoning in the proof of the backward direction.
Similarly, if $\pbd K$ is nonempty, this implies condition 1. 

\end{proof}

\begin{proof}[Proof of \cref{lemma:betti_from_cocell}]
Let $\resf E$ be the chain complex obtained from cochain complex $\resf F^Y$ by placing cohomological degree $d$ at homological degree 0.
Then $\betti_{i, \es b}(I) = \dim_{\kk} H_i(\resf E\otimes \kk)_{\es b} = \dim_\kk H^{d-i}(\resf F^Y \otimes \kk)_{\es b}$.
But the degree $\es b$ part of $\resf F^Y \otimes \kk$ is exactly the cochain complex of the collection of open cells $Y_{\es b}$.
By \cref{lemma:subcollection_realization}, $Y_{\es b}$ is realized by $(\overline Y_{\es b}, \pbd Y_{\es b})$, so $H^{d-i}(\resf F^Y \otimes \kk)_{\es b} = H^{d-i}(\overline Y_{\es b}, \pbd Y_{\es b})$, which yields the desired result. 
\end{proof}

\begin{proof}[Proof of \cref{lemma:openly_convex_dichotomy}]
WLOG, we can replace $\R^q$ with the span of $L$, so we assume $L$ spans $\R^q$.
We show by induction on $q$ that if $L \not\sbe H$ for every open halfspace, then $0 \in L$.
This would imply our result:
As $L$ is open in $\R^q$, there is a ball contained in $L$ centered at the origin.
Since $L$ is a cone, this means $L = \R^q$.

Note that $L \not\sbe H$ for every open coordinate halfspace $H$ is equivalent to that $L \cap H \not = \emptyset$ for every open coordinate halfspace $H$.
Indeed, if $L \cap H' = \emptyset$, then $\R^q \setminus H'$ contains the open set $L$, and thus the interior $\intrr (\R^q \setminus H)$ is an open coordinate halfspace that contains $L$.
If $L$ intersects every open coordinate halfspace, then certainly it cannot be contained in any single $H$, or else $\intrr (\R^q \setminus H)$ does not intersect $L.$

We now begin the induction.
The base case of $q = 1$: $L$ has both a positive point and negative point, and thus contains 0 because it is convex.

Suppose the induction hypothesis holds for $q = p$, and let $q = p+1$.
Then for any halfspace $H$, $L \cap H$ and $L \cap \intrr(\R^q \setminus H)$ are both nonempty, and thus $L$ intersects the hyperplane $\pd H$ by convexity.
Certainly $L \cap \pd H$ intersects every open coordinate halfspace of $\pd H$ because the latter are intersections of open coordinate halfspaces of $\R^q$ with $\pd H$.
So by the induction hypothesis, $L \cap \pd H$ contains 0, and therefore $0 \in L$ as desired.
\end{proof}

\section{Cosheaf Construction of the Canonical Suboplex}
\label{sec:cosheaf}

Let $\clsf C \sbe [n \to m]$ and $p$ be a probability distribution on $[n].$
$p$ induces an $L_1$ metric space $\clsf C_p$ by $d(f, g) = \f 1 n \|f - g\|_1$.
If we vary $p$ over $\splx^{n-1}$, then $\clsf C_p$ traces out some kind of shape that ``lies over'' $\splx^{n-1}$.
For $\clsf C = [2 \to 2]$, this is illustrated in \cref{fig:2to2_cosheaf}.
\begin{figure} 
\centering
\includegraphics[width=.5\textwidth]{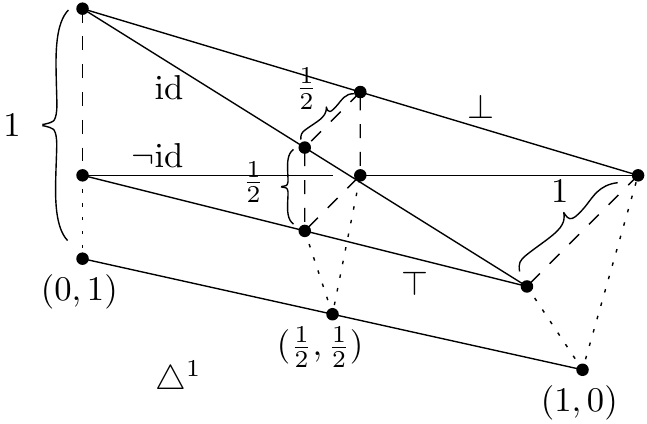}
\caption{``Gluing'' together the metric spaces $\clsf [2\to2]_p$ for all $p \in \splx^1$.
The distances shown are $L_1$ distances of functions within each ``fiber.''
$\bot$ is the identically 0 function; $\top$ is the identically 1 function; $\id$ is the identity; and $\neg \id$ is the negation function.
If we ignore the metric and ``untangle'' the upper space, we get the complete 1-dimensional suboplex.}
\label{fig:2to2_cosheaf}
\end{figure}
In this setting, Impagliazzo's Hardcore Lemma \cite{arora_computational_2009} would say something roughly like the following:
\begin{equation*}
\pbox{.8\textwidth}{Let $\clsf C \sbe [n \to 2]$ and $\overline {\clsf C}$ be the closure of $\clsf C$ under taking majority over ``small'' subsets of $\clsf C$.
For any $f \in [n \to 2]$, either in the fiber $[n \to 2]_{\mathcal U}$ over the uniform distribution $\mathcal U$, $f$ is ``close'' to $\overline{\clsf C}$, or in the fiber $[n \to 2]_q$ for some $q$ ``close'' to $\mathcal U$, $f$ is at least distance $1/2+\epsilon$ from $\clsf C$.}
\end{equation*}
Thus this view of ``fibered metric spaces'' may be natural for discussion of hardness of approximation or learning theory.

If we ignore the metrics and ``untangle'' the space, we get the canonical suboplex of $[2 \to 2]$, the complete 1-dimensional suboplex.
In general, the canonical suboplex of a class $\clsf C \sbe [n \to m]$ can be obtained by ``gluing'' together the metric spaces $\clsf C_p$ for all $p \in \splx^{n-1}$, so that there is a map $\Upsilon: \Sbpx C \to \splx^{n-1}$ whose fibers are $\clsf C_p$ (treated as a set).
But how do we formalize this ``gluing'' process?

In algebraic topology, one usually first tries to fit this picture into the framework of fibrations or the framework of sheaves.
But fibration is the wrong concept, as our ``fibers'' over the ``base space'' $\splx^{n-1}$ are not necessarily homotopy equivalent, as seen in \cref{fig:2to2_cosheaf}.
So $\Sbpx{\clsf C}$ cannot be the total space of a fibration over base space $\splx^{n-1}$. 
Nor is it the \'etal\'e space of a sheaf, as one can attest to after some contemplation.

It turns out the theory of {\it co}sheaves provide the right setting for this construction.

\subsection{Cosheaves and Display Space}
\newcommand{\shF}{\mathcal{F}}
\newcommand{\contf}{\psi}
\begin{defn}
A {\bf precosheaf} is a covariant functor $\shF: \topOpen(X) \to \catSet$ from the poset of open sets of a topological space $X$ to the category of sets.
For each inclusion $\imath: U \hookrightarrow V$, the set map $\shF\imath: \shF(U) \to \shF(V)$ is called the {\bf inclusion map} from $\shF(U)$ to $\shF(V)$.

A precosheaf $\shF$ is further called a {\bf cosheaf} if it satisfies the following cosheaf condition:
For every open covering $\{U_i\}_i$ of an open set $U \sbe X$ with $\bigcup \{U_i\}_i = U$,
$$\coprod_{k} \shF(U_k) \gets \coprod_{k\not=l}\shF(U_k \cap U_l)\to \coprod_{l} \shF(U_l)$$ 
has pushout $\shF(U)$.
Here each arrow is the coproduct of inclusion maps.
\end{defn}

There is a concept of costalk dual to the concept of stalks in sheaves.
\begin{defn}
Let $\shF: \topOpen(X) \to \catSet$ be a cosheaf and let $p \in X$.
Then the costalk $\shF_p$ is defined as the cofiltered limit
$$\shF_p := \lim_{U \in p} \shF(U),$$
of $\shF(U)$ over all open $U$ containing $p$.
\end{defn}

Analogous to the \'etal\'e space of a sheaf, cosheaves have something called a {\bf display space} \cite{funk_display_1995} that compresses all of its information in a topological space.
We first discuss the natural cosheaf associated to a continuous map.

Let $\contf: Y \to X$ be a continuous map between locally path-connected spaces $Y$ and $X$.
We have a cosheaf $\shF^\psi: \topOpen(X) \to \catSet$ induced as follows:
For each $U \in \topOpen(X)$, $\shF^\contf(U) = \pi_0(\inv \contf U)$, where $\pi_0$ denotes the set of connected components.
For an inclusion $\imath: U \hookrightarrow V$, $\shF^\contf(\imath)$ maps each component in $Y$ of $\inv \contf U$ into the component of $\inv \contf V$ that it belongs to.
For open cover $\{U_i\}_i$ with union $U$, 
$$\coprod_{k} \shF^\contf(U_k) \gets \coprod_{k\not=l}\shF^\contf(U_k \cap U_l)\to \coprod_{l} \shF^\contf(U_l)$$
has pushout $\shF^\contf(U)$.
Indeed, this is just the standard gluing construction of pushouts in $\catSet$ for each component of $\inv \contf U$.
(An alternative view of $\shF^\contf$ is that it is the direct image cosheaf of $\shF^\id$, where $\id: Y \to Y$ is the identity).

Now we reverse the construction.
Let $X$ be a topological space, and $\shF: \topOpen(X) \to \catSet$ be a cosheaf.
We construct the display space $Y$ and a map $\contf: Y \to X$ such that $\shF \cong \shF^\contf$.
For the points of $Y$, we will take the disjoint union of all costalks,
$$|Y| := \bigsqcup_{p \in X} \shF_p.$$
Then the set-map $|\contf|$ underlying the desired $\contf$ will be
$$\shF_p \ni y \mapsto p \in X.$$

Now we topologize $Y$ by exhibiting a basis.
For any $U \in \topOpen(X)$, there is a canonical map
$$g_U := \bigsqcup_{p \in U} m_{p, U}: \bigsqcup_{p \in U} \shF_p \to \shF(U)$$
formed by the coproduct of the limit maps $m_{p, U}: \shF_p \to \shF(U)$.
Then each fiber of $g_U$ is taken as an open set in $Y$:
For each $s \in \shF(U)$, we define
$$[s, U] := g_U^{-1}(s)$$
as an open set.
Note that $[s, U] \cap [t, U] = \emptyset$ if $s, t \in \shF(U)$ but $s \not= t$.
We claim that for $s \in \shF(U), t \in \shF(V)$,
\begin{align}
[s, U] \cap [t, V] = \bigsqcup \{[r, U \cap V]: \shF(i_U)(r) = s \And \shF(i_V)(r)=t\}
\label{open_intersect}
\end{align}
where $i_U: U \cap V \to U$ and $i_V: U \cap V \to V$ are the inclusions.
The inclusion of the RHS into the LHS should be clear.
For the opposite direction, suppose $p \in U \cap V$ and $y \in \shF_p$ with $g_U(y) = m_{p, U}(y) = s$ and $g_V(y) = m_{p, V}(y) = t$.
Since $\shF_p$ is the cofiltered limit of $\{\shF(W): p \in W\}$, we have the following commutative diagram
$$\begin{tikzcd}
        &	\shF_p \arrow[d, "m_{p, U \cap V}"] \arrow[ddl, bend right, "m_{p, U}" above left] \arrow[ddr, bend left, "m_{p, V}"]\\
        &	\shF(U \cap V) \arrow[dl, "\shF(j)"] \arrow[dr, "\shF(k)" below left]&	\\
\shF(U)	&			& \shF(V)
\end{tikzcd}$$
Therefore there is an $r \in \shF(U \cap V)$ such that $m_{p, U \cap V}(y) = r$ and $\shF(j)(r) = s$ and $\shF(k)(r) = t$.
Then $y \in [r, U \cap V] \sbe \text{RHS of \ref{open_intersect}}$.
Our claim is proved, and $\{[s, U]: s \in \shF(U)\}$ generates a topological basis for $Y$.

Finally, to complete the verification that $\shF \cong \shF^\contf$, we show that $\shF(U) \cong \pi_0(\inv \contf U)$, natural over all $U \in \topOpen(X)$.
It suffices to prove that for each $U \in \topOpen(X)$ and $s \in \shF(U)$, $[s, U]$ is connected; then $\shF(U) \ni s \mapsto [s, U]$ is a natural isomorphism.

Suppose for some $s \in \shF(U)$ this is not true: there exists a nontrivial partition $\bigcup_{i \in A}\{[s_i, U_i]\}\ \sqcup\ \bigcup_{j \in B}\{[s_j, U_j]\} = [s, U]$ of $[s, U]$ by open sets $\bigcup_{i \in A}\{[s_i, U_i]\}$ and $\bigcup_{j \in B}\{[s_j, U_j]\}$.
We assume WLOG that $\bigcup_{i \in A} U_i \cup \bigcup_{j \in B} U_j = U$ (in case that for some $x \in U$, $\shF_p = \emptyset$, we extend each $U_i$ and $U_j$ to cover $x$).
Then by the cosheaf condition, the pushout of the following
$$\coprod_{k \in A \cup B} \shF(U_k) \gets \coprod_{k\not=l}\shF(U_k \cap U_l)\to \coprod_{l \in A \cup B} \shF(U_l)$$
is $\shF(U)$.
By assumption, $\shF(U_i \into U)(s_i) = s$ for all $i \in A \cup B$.
So there must be some $i \in A, j \in B$ and $t \in \shF(U_i \cap U_j)$ such that $\shF(U_i \cap U_j \into U_i)(t) = s_i$ and $\shF(U_i \cap U_j \into U_j)(t) = s_j$.
This implies that $[t, U_i \cap U_j] \sbe [s_i, U_i] \cap [s_j, U_j]$.
If $X$ is first countable and locally compact Hausdorff, or if $X$ is metrizable, then by \cref{lemma_nonempty_open}, $[t, U_i \cap U_j]$ is nonempty, and therefore $\bigcup_{i \in A}\{[s_i, U_i]\}\ \cap\ \bigcup_{j \in B}\{[s_j, U_j]\} \not=\emptyset$, a contradiction, as desired.
\begin{lemma} \label{lemma_nonempty_open}
If $X$ is first countable and locally compact Hausdorff, or if $X$ is metrizable, then $[s, U]$ is nonempty for every $U \in \topOpen(X)$ and $s \in \shF(U)$.
\end{lemma}
\begin{proof}
We give the proof for the case when $X$ is first countable and locally compact Hausdorff.
The case of metrizable $X$ is similar.

For each $x \in X$, fix a countable local basis $x \sbe \cdots \sbe B_x^n \sbe B_x^{n-1} \sbe \cdots \sbe B_x^2 \sbe B_x^1$, with the property that $\setcl {B_x^n} \sbe B_x^{n-1}$ and is compact.
Fix such a $U$ and $s \in \shF(U)$.
Let $U_0 := U$ and $s_0 := s$.
We will form a sequence $\la U_i, s_i\ra$ as follows.
Given $U_{i-1}$ and $s_{i-1}$, for each point $x \in U_{i-1}$, choose a $k_x > i$ such that $\setcl{B_x^{k_x}}$ is contained in $U_{i-1}$.
These sets $\{B_x^{k_x}\}_x$ form an open covering of $U_{i-1}$, and by the sheaf condition, for some $x$, $\im \shF(B_x^{k_x} \into U_{i-1})$ contains $s_{i-1}$.
Then set $U_i := B_x^{k_x}$ and choose any element of $\shF(B_x^{k_x} \into U_{i-1})^{-1}(s_{i-1})$ to be $s_i$.
Hence by construction $s_i \in \shF(U_i)$.

Following this procedure for all $i \in \N$, we obtain a sequence $\la U_i, s_i \ra_{i \ge 0}$ with the property that $U_0 \spe \setcl{U_1} \spe U_1 \spe \setcl{U_2} \spe U_2 \cdots $.
As each of $\setcl{U_i}$ is compact, $\bigcap \setcl{U_i}$, and hence $\bigcap U_i = \bigcap \setcl{U_i}$, is nonempty.
Let $z$ be one of its elements.
Then $U_i \sbe B_z^i$ for all $i \ge 1$.
Therefore $z$ must be the unique element of $\bigcap U_i$, and the sequence $\la U_i \ra_i$ is a local basis of $z$.
Furthermore, $\la s_i\ra_i$ is an element of the costalk at $z$, as it can easily be seen to be an element of the inverse limit $\lim_{i \to \infty} \shF(U_i) = \lim \{\shF(V): z \in V\}$.
This shows that $[s, U]$ is nonempty.
\end{proof}

Note that without assumptions on $X$, \cref{lemma_nonempty_open} cannot hold.
In fact, something quite extreme can happen.
\begin{prop}
There exists a cosheaf $\shF: \topOpen(X) \to \catSet$ whose costalks are all empty.
\end{prop}
\begin{proof}
This proof is based on Waterhouse's construction \cite{waterhouse_empty_1972}.
Let $X$ be an uncountable set with the cofinite topology.
Define $\shF(U)$ to be the set of injective functions from the finite set $X\setminus U$ to the integers.
The map $\shF(U \into V)$ just restricts a function $g: X \setminus U \to \Z$ to $g\restrict (X \setminus V): X \setminus V \to \Z$.
One can easily check that the cosheaf sequence is a pushout.
Thus $\shF$ is a cosheaf.

For any $x \in X$, each point of the inverse limit of $\{\shF(U): x \in U\}$ has the following description: a sequence of injective functions $\la f_A: A \into \Z\ra_A$ indexed by finite sets $A \sbe X$, such that if $A \sbe B$ are both finite sets, then $f_A \sbe f_B$.
Such a sequence would determine an injective function $\bigcup_A f_A: X \to \Z$, but that is impossible as $X$ was assumed to be uncountable.
\end{proof}

Back to our case of canonical suboplex.
For any $\Ss = \Ss_{[n \to m]}$, there is a canonical embedding $\Xi: \Ss \into \splx^{mn-1} \sbe \R^{mn}$, defined by taking vertex $V_{u, i}, (u, i) \in [n] \times [m]$ to $e_{u, i}$, the basis vector of $\R^{mn}$ corresponding to $(u, i)$, and taking each convex combination $\sum_{u=0}^{n-1} p(u)V_{u, f(i)}$ in the simplex associated to $f: [n] \to [m]$ to $\sum_{u=0}^{n-1} p(u)e_{u, f(i)}$.
The map $\Upsilon: \Sbpx C \to \splx^{n-1}$ we sketched in the beginning of this section can then be formally described as $\Upsilon = \Pi \circ \Xi \restrict \Sbpx C$, where $\Pi$ is the linear projection defined by $e_{u, i} \mapsto e_u \in \splx^{n-1}$.
As we have shown, $\Upsilon$ induces a cosheaf $\shF^\Upsilon: \topOpen(\splx^{n-1}) \to \catSet$, sending each open $U \sbe \splx^{n-1}$ to $\pi_0(\inv \Upsilon U)$.
For example, if $U$ is in the interior of $\splx^{n-1}$, then $\shF^\Upsilon(U)$ has size equal to the size of $\clsf C$.
If $U$ is a small ball around the vertex $e_u$, then $\shF^\Upsilon(U)$ is bijective with the set of values $\clsf C$ takes on $u \in [n]$.
It is easy to check that the costalk $\shF^\Upsilon_p$ at each point $p \in \splx^{n-1}$ is just $\pi_0(\inv \Upsilon p) = |\clsf C_p|$, the set underlying the metric space $\clsf C_p$, so we have successfully ``glued'' together the pieces into a topological space encoding the separation information in $\clsf C$.

One may naturally wonder whether the cosheaf homology of such a cosheaf matches the homology of the display space.
One can show that this is indeed the case for our canonical suboplex, via identification of the cosheaf homology with Cech homology and an application of the acyclic cover lemma.

What is disappointing about this construction is of course that it ignores metric information in all of the costalks $\clsf C_p$.
Directly replacing $\catSet$ with the category $\catf{Met}$ of metric spaces with metric maps (maps that do not increase distance) does not work, because it does not have coproducts.
It remains an open problem whether one can find a suitable category to replace $\catSet$ such that $\Sbpx C$ can still be expressed as the display space of a cosheaf on $\splx^{n-1}$, while preserving metric information in each costalk, and perhaps more importantly, allows the expression of results like Impagliazzo's Hardcore Lemma in a natural categorical setting.
Perhaps a good starting point is to notice that the embedding $\Xi$ actually preserves the $L_1$ metric within each fiber $\clsf C_p$. 


\bibliographystyle{plain}
\bibliography{sheafbool}

\begin{thebibliography}{10}

\bibitem{anthony_discrete_2001}
Martin Anthony.
\newblock {\em Discrete {Mathematics} of {Neural} {Networks}}.
\newblock 2001.

\bibitem{arora_computational_2009}
Sanjeev Arora and Boaz Barak.
\newblock {\em Computational complexity: a modern approach}.
\newblock 2009.
\newblock OCLC: 443221176.

\bibitem{aspnes_expressive_1994}
J.~Aspnes, R.~Beigel, M.~Furst, and S.~Rudich.
\newblock The expressive power of voting polynomials.
\newblock {\em Combinatorica}, 14(2):135--148, June 1994.

\bibitem{bayer_computation_1992}
Dave Bayer and Mike Stillman.
\newblock Computation of {Hilbert} functions.
\newblock {\em Journal of Symbolic Computation}, 14(1):31--50, 1992.

\bibitem{bruns_cohen-macaulay_1998}
Winfried Bruns and Jurgen Herzog.
\newblock {\em Cohen-{Macaulay} {Rings}}.
\newblock Cambridge University Press, 1998.

\bibitem{eisenbud_commutative_1994}
David Eisenbud.
\newblock {\em Commutative algebra: with a view toward algebraic geometry}.
\newblock 1994.
\newblock OCLC: 891662214.

\bibitem{faridi_projective_2013}
Sara Faridi.
\newblock The projective dimension of sequentially {Cohen}-{Macaulay} monomial
  ideals.
\newblock {\em arXiv:1310.5598 [math]}, October 2013.
\newblock arXiv: 1310.5598.

\bibitem{funk_display_1995}
J.~Funk.
\newblock The display locale of a cosheaf.
\newblock {\em Cahiers de Topologie et Géométrie Différentielle
  Catégoriques}, 36(1):53--93, 1995.

\bibitem{hochster_canonical_1989}
Melvin Hochster.
\newblock The canonical module of a ring of invariants.
\newblock {\em Contemp. Math}, 88:43--83, 1989.

\bibitem{kearns_introduction_1994}
Michael Kearns and Umesh Vazirani.
\newblock {\em An {Introduction} to {Computational} {Learning} {Theory}}.
\newblock January 1994.

\bibitem{miller_combinatorial_2005}
Ezra Miller and Bernd Sturmfels.
\newblock {\em Combinatorial commutative algebra}.
\newblock Number 227 in Graduate texts in mathematics. Springer, New York,
  2005.
\newblock OCLC: ocm55765389.

\bibitem{minsky_perceptrons:_1969}
Marvin Minsky and Seymour Papert.
\newblock {\em Perceptrons: {An} {Introduction} to {Computational} {Geometry}}.
\newblock MIT Press, 1969.

\bibitem{mulmuley_geometric_2001}
K.~Mulmuley and M.~Sohoni.
\newblock Geometric {Complexity} {Theory} {I}: {An} {Approach} to the {P} vs.
  {NP} and {Related} {Problems}.
\newblock {\em SIAM Journal on Computing}, 31(2):496--526, January 2001.

\bibitem{natschlager_exact_1996}
Thomas Natschläger and Michael Schmitt.
\newblock Exact {VC}-dimension of {Boolean} monomials.
\newblock {\em Information Processing Letters}, 59(1):19--20, July 1996.

\bibitem{odonnell_analysis_2014}
Ryan O'Donnell.
\newblock {\em Analysis of boolean functions}.
\newblock Cambridge University Press, New York, NY, 2014.

\bibitem{razborov_natural_1997}
Alexander~A Razborov and Steven Rudich.
\newblock Natural {Proofs}.
\newblock {\em Journal of Computer and System Sciences}, 55(1):24--35, August
  1997.

\bibitem{shelah_combinatorial_1972}
Saharon Shelah.
\newblock A combinatorial problem; stability and order for models and theories
  in infinitary languages.
\newblock {\em Pacific Journal of Mathematics}, 41(1):247--261, 1972.

\bibitem{stanley_combinatorics_1996}
Richard~P. Stanley.
\newblock {\em Combinatorics and commutative algebra}.
\newblock Number v. 41 in Progress in mathematics. Birkhäuser, Boston, 2nd ed
  edition, 1996.

\bibitem{the_univalent_foundations_program_homotopy_2013}
{The Univalent Foundations Program}.
\newblock Homotopy {Type} {Theory}: {Univalent} {Foundations} of {Mathematics}.
\newblock {\em arXiv preprint arXiv:1308.0729}, 2013.

\bibitem{velasco_minimal_2008}
Mauricio Velasco.
\newblock Minimal free resolutions that are not supported by a {CW}-complex.
\newblock {\em Journal of Algebra}, 319(1):102--114, January 2008.

\bibitem{waterhouse_empty_1972}
William~C. Waterhouse.
\newblock An empty inverse limit.
\newblock {\em Proceedings of the American Mathematical Society}, 36(2):618,
  1972.

\bibitem{ziegler_lectures_1995}
Günter~M. Ziegler.
\newblock {\em Lectures on {Polytopes}}, volume 152 of {\em Graduate {Texts} in
  {Mathematics}}.
\newblock Springer New York, New York, NY, 1995.

\end{thebibliography}
\end{document}